\documentclass[11pt,twoside]{preprint}
\usepackage[a4paper,innermargin=1.2in,outermargin=1.2in,
bottom=1.5in,marginparwidth=1in,marginparsep
=3mm]{geometry}
\usepackage{amsmath,amsthm,amssymb,enumerate}
\usepackage{hyperref}
\usepackage{mathrsfs}
\usepackage{breakurl}
\usepackage[nameinlink,capitalize]{cleveref}
\usepackage{apptools}
\usepackage{chngcntr}
\usepackage[osf]{newtxtext}
\usepackage[varqu,varl]{inconsolata}
\usepackage[cal=boondoxo]{mathalfa}
\usepackage[vvarbb]{newtxmath}
\usepackage{longtable}
\usepackage{booktabs}

\usepackage[dvipsnames]{xcolor}
\usepackage{mhequ}
\usepackage{comment}
\usepackage{microtype}
\usepackage{pifont}% http://ctan.org/pkg/pifont

\makeatletter
\newcommand{\subjclass}[2][2020]{%
  \let\@oldtitle\@title%
  \gdef\@title{\@oldtitle\footnotetext{#1 \emph{Mathematics subject classification.} #2}}%
}
\makeatother

\colorlet{darkblue}{blue!90!black}
\colorlet{darkred}{red!90!black}
\colorlet{darkgreen}{green!60!black}
\colorlet{purp}{blue!30!red}
\colorlet{lblue}{blue!60!white}

\def\chengcheng#1{\comment[darkblue]{CL: #1}}
\def\mate#1{\comment[darkred]{MG: #1}}

\newtheorem{theorem}{Theorem}[section]
\newtheorem{assumption}[theorem]{Assumption}
\newtheorem{lemma}[theorem]{Lemma}
\newtheorem{proposition}[theorem]{Proposition}
\newtheorem{corollary}[theorem]{Corollary}

\theoremstyle{definition}
\newtheorem{definition}[theorem]{Definition}

\theoremstyle{remark}
\newtheorem{remark}[theorem]{Remark}
\def\scal#1{\langle #1 \rangle}

\newcommand{\vn}[1]{{\vert\kern-0.23ex\vert\kern-0.23ex\vert #1 
    \vert\kern-0.23ex\vert\kern-0.23ex\vert}}
\usepackage{mathtools}

\allowdisplaybreaks

\newcommand{\I}{\mathbb{I}}

\newcommand\bone{\mathbf{1}}

\newcommand\cR{\mathcal{R}}

\newcommand\cK{\mathcal{K}}
\newcommand\cL{\mathcal{L}}
\newcommand\cH{\mathcal{H}}
\newcommand\cD{\mathcal{D}}

\newcommand\cA{\mathcal{A}}

\newcommand\cC{\mathcal{C}}

\newcommand\cJ{\mathcal{J}}

\newcommand\cZ{\mathcal{Z}}

\newcommand\cV{\mathcal{V}}

\newcommand\cM{\mathcal{M}}

\newcommand\cY{\mathcal{Y}}

\newcommand\cP{\mathcal{P}}
\newcommand\cF{\mathcal{F}}

\newcommand\CH{\mathcal{H}}

\newcommand{\bbg}{\mathbb{g}}

\newcommand{\cX}{\mathcal{X}}

\newcommand{\D}{\partial}

\def\eps{\varepsilon}

\newcommand{\R}{{\mathbb{R}}}

\renewcommand{\P}{\mathbb{P}}

\def\Q{\mathbb{Q}}

\newcommand{\E}{\mathbb{E}}

\newcommand{\bbI}{\mathbb{I}}
\newcommand{\mR}{\mathbb{R}}
\newcommand{\bD}{\mathbb{D}}
\newcommand{\bF}{\mathbb{F}}

\newcommand{\N}{\mathbb{N}}
\newcommand{\Var}{\mathbf{Var}}

\def\supp{\mathop{\mathrm{supp}}}

\def\id{\mathrm{id}}
\def\path{\mathrm{path}}
\def\cut{\mathrm{Cut}}
\def\geo{\mathrm{geo}}
\def\dom{\mathrm{dom}}

\def\pol{\mathrm{pol}}

\newcommand{\blue}[1]{{\color{blue}#1}}
\newcommand{\black}[1]{{\color{black}#1}}
\newcommand{\red}[1]{{\color{red}#1}}
\newcommand{\purp}[1]{{\color{purp}#1}}
\begin{document}
\title{Regularisation by Gaussian rough path lifts of fractional Brownian motions}
\author{Konstantinos Dareiotis\thanks{School of Mathematics, University of Leeds, Leeds, U.K., \url{K.Dareiotis@leeds.ac.uk}}\, , M\'at\'e Gerencs\'er\thanks{Institute of Analysis and Scientific Computing, TU Wien, Austria, \url{mate.gerencser@asc.tuwien.ac.at}}\, , Khoa L\^e\thanks{School of Mathematics, University of Leeds, Leeds, U.K., \url{K.Le@leeds.ac.uk}}\, , Chengcheng Ling\thanks{Institut f\"ur Mathematik, University of Augsburg, Augsburg,  Germany.
\url{chengcheng.ling@uni-a.de}}}
\maketitle

\begin{abstract}
The aim of the paper is to show the probabilistically strong well-posedness of rough differential equations with distributional drifts driven by the Gaussian rough path lift of fractional Brownian motion with Hurst parameter $H\in(1/3,1/2)$. We assume that the noise is nondegenerate and the drift lies in  the Besov-H\"older space $\cC^\alpha$ for some $\alpha>1-1/(2H)$. The latter condition matches the one of the additive noise case, thereby providing a multiplicative analogue of Catellier-Gubinelli in the said regime.% $H\in(1/3,1/2)$. 
\end{abstract}
\vspace{1em}
\\
\indent 

\tableofcontents

\section{Introduction}
Let $B$ be a fractional Brownian motion with Hurst parameter $H\in(1/3,1/2)$.
%and let $\blue{B}$ be its Gaussian rough path lift.
Our goal is to solve equations of the form
\begin{equ}\label{eq:main}
X_t=\int_{0}^t b(X_s)\,ds+\int_{0}^t \black{\sigma(X_s)}\,d\black{B_s},%=:x_0+D^X_t+S^X_t.
\end{equ}
with \emph{distributional} drift $b$.
This is a natural continuation of the work \cite{KM}, whose aim was to simultaneously treat two ``sides'' of the fractional Brownian motion. On the one hand, the sample paths of $B$ are too irregular to have any canonical solution theories of differential equations driven by them (even without a drift component $b$), which makes them a prime example for Lyons' theory of rough paths \cite{Lyons, FV10}.
On the other hand, when one avoids the issue of rough integration by taking $\sigma$ to be constant, it is known \cite{NO1, CATELLIER20162323}
that the equation with nondegenerate $\sigma$ behaves much better than the equation without the noise,
in the sense that the regularity assumptions on $b$ that are required for well-posedness are far below the ones of the classical Cauchy-Lipschitz theorem.
This \emph{regularisation by noise} phenomenon has been recently widely studied and extended in a variety of directions, see, e.g.,
\cite{ Khoa, MR4342752, MR4404773, MR4664455, GG, butkovsky2023stochasticequationssingulardrift, oleg-weak}.
A common theme is the rule of thumb \emph{``the more irregular the noise, the more regularisation effect it has''}, having small parameters $H$ is not only relevant for rough path theory, but also interesting from the regularisation perspective. In fact, the regime $H<1/2$ is where the regularisation is strong enough so that equations with \emph{distributional} $b$ can be strongly well-posed. More precisely, in the Besov-H\"older scale the condition on the drift reads $b\in \cC^\alpha$, $\alpha>1-1/(2H)$, and for negative exponents such spaces contain distributions.

In recent works \cite{KM, cate} equations of the kind \eqref{eq:main} with genuinely multiplicative noise (i.e. non-constant $\sigma$) have been considered.
The approach of \cite{KM} yields in fact optimal results in the regime $H>1/2$ in the sense that the requirement on $b$ matches the one from the additive noise case. 
In the rough regime both works go well beyond the classical Cauchy Lipschitz condition but fail to include distributional drifts: strong well-posedness is shown in \cite{KM} under the condition $\alpha>0$, $H\in(1/3,1/2)$, while in \cite{cate} under the condition $\alpha>\max(0,3/2-1/(2H))$, $H\in(1/4,1/2)$.
%The goal of the present paper is to
%While in the regime $H>1/2$ \cite{KM} gives the
In the present paper, we prove strong well-posedness for equations with distributional drifts. Loosely speaking, the main result is that provided $\sigma$ is smooth and nondegenerate, \eqref{eq:main} is well-posed for any $b\in\cC^\alpha$ under the ``classical'' condition $\alpha>1-1/(2H)$ \cite{CATELLIER20162323}.
The precise formulation requires some setup and is stated in \cref{thm:main} below.
As in \cite{KM}, the present approach relies on the stochastic sewing lemma (SSL) from \cite{Khoa}, but with rather significant novelties that we outline below.

\begin{remark}
Our assumption of smooth $\sigma$ is only for convenience, but the methods require at least $\sigma\in \cC^5$.  
    We note that somewhat orthogonally to the direction of the present paper, significant progress has also been made recently in reducing the regularity requirements on $\sigma$ compared to the conditions in the purely analytic Young/rough theories, see \cite{Avi_Toyomu, MR4730255}. Combining their method with the present work seems to be challenging, given the several Malliavin derivatives we require on the flow of the driftless equation.
\end{remark}

\iffalse
{\color{blue}(some short introductory blabla)There is an increasing interest on studying the well-posedness of \eqref{eq:main}. It is nowadays a well understood subject when  $B$ is a Brownian motion by using of the well-developed tools from It\^o's theory (\cite{MR1867429, Fef, zhang2018heatkernelergodicitysdes}).  Those results usually are interpreted as a phenomenon of \emph{regularisation by noise}. Unfortunately but interestingly,  the
approach based on It\^o's calculus, used in most of the papers  for
Brownian motion, does not easily extend to the fractional  case.  Starting from  \cite{NO1, CATELLIER20162323}, the analogue of the results on Brownian noise gradually are obtained for fractional  Brownian motion with  $\sigma\equiv\mathbb{I}_{d}$. In last five years also, due to the contribution from \emph{sewing lemma} \cite{MR2091358} and \emph{stochastic sewing lemma} \cite{Khoa}, the study on \eqref{eq:main} with singular drift regains a lot attention and decent development (e.g. \cite{MR4342752, MR4404773, GG, butkovsky2023stochasticequationssingulardrift}). However,  to be noticed that among all of the aforementioned results, the one  on multiplicative fractional noise (i.e. $\sigma$ is a non-constant matrix valued function) is either missing or left behind comparing with  the ``classical''  additive one (\cite{CATELLIER20162323}). } 
\fi

%That is, $b\in C^\alpha$ with some $\alpha<0$.
\subsection{The overview of the strategy}

Let us give an overall view on the strategy of the paper and highlight what novel difficulties have to be tackled. First of all, let us recall that \eqref{eq:main} in fact employs an abuse of notation. Indeed, the integral
\begin{equ}
    \int_{0}^t \black{\sigma(X_s)}\,d\black{B_s}
\end{equ}
is not an object that depends only on $B$ and $X$, but also on other components: the second order iterated integrals of the coordinates of $B$ and the so-called Gubinelli derivative of $\sigma(X)$, which in turn is obtained not only from $X$ but also its Gubinelli derivative.
To distinguish between standard processes and their enhancements with  other components, we shall use colors (similar to, but not exactly the same convention as, in \cite{Friz-Hairer}), their precise meaning is defined in the notation section below.
For example, the integral above will be denoted by 
\begin{equ}
    \int_{0}^t \blue{\sigma(X_s)}\,d\blue{B_s}
\end{equ}
to indicate that \(\blue{\sigma(X)}, \blue{B}\) are enhanced processes.
%Since the strategy deviates significantly from typical well-posedness proofs for SDEs, and even compared to the established stochastic sewing literature there are new challenges to be handled, we give a sketch of the overall strategy of the article.
%\mate{Not sure if this is better here or even before the notation section? }

%\emph{Step 0 }(Cauchy-Lipschitz).
We start by recalling the standard rough Cauchy-Lipschitz argument.
Take two strong solutions $X$ and $Y$ that are adapted to the same filtration $(\cF_s)_{s\geq 0}$, with respect to which $B$ is a fractional Brownian motion. Their difference can be written as
\begin{equ}
X_t-Y_t=\int_0^t \big(b(X_s)-b(Y_s)\big)\,ds+\int_0^t \big( \blue{\sigma(X_s)}-\blue{\sigma(Y_s)}\big) \, d \blue{B_s}.
\end{equ}
If $b$ were Lipschitz, then using that for sufficiently regular $\sigma$, the map $\blue{X}\mapsto\blue{\sigma(X)}$ is locally Lipschitz in the norm of controlled paths, one easily concludes that $X=Y$ by a Gronwall-type argument.
To formulate slightly differently, one can write by Newton-Leibniz and ``rough Newton-Leibniz'' (see \eqref{eq:rough_FTC} below) formulae
\begin{equs}
X_t-Y_t&=\int_0^t \Big(\int_0^1\nabla b\big(\theta X_s+(1-\theta) Y_s\big)\,d\theta\Big)(X_s-Y_s)\,ds
\\
&\qquad+\int_0^t \Big(\int_0^1\blue{\nabla \sigma\big(\theta' X_s+(1-\theta') Y_s\big)}\,d\theta'\Big)(\blue{X_s}-\blue{Y_s})\,d\blue{B_s}.
\end{equs}
In this form, it is also clear that the Lipschitzness of $b$ guarantees that the above is a well-defined linear rough equation for $Z=X-Y$ with initial condition $0$, therefore the solution is identically $0$.
If $b$ fails to be Lipschitz, the meaning of the integrand with respect to $\theta$ becomes unclear. Note that since $\sigma$ is always assumed to be fairly regular, the integrand with respect to $\theta'$ is always well-defined and is henceforth denoted by $\blue{\Sigma_s}$.

\emph{An attempt at Young formulation}. To give a meaning to the above equation for $Z$, we first aim to \emph{construct} the integral
\begin{equ}\label{eq:L-early}
L_t=\int_0^1\int_0^t\nabla b\big(\theta X_s+(1-\theta) Y_s)\big)\,ds\,d\theta
\end{equ}
by probabilistic methods.
Such objects, i.e. integrals of distributions along stochastic processes play a key role in regularisation by noise, see \cite{ Davie,flandoli2010well, CATELLIER20162323,Khoa}
for some prominent examples.
%MR4342752,  KM, butkovsky2023stochasticequationssingulardrift, MR4730255,  GG}
On the first reading, one may also think of taking smooth $b$ to avoid any confusion of meaning, but with the aim to control $L$ in some norm using only the distributional norm of $\nabla b$.
Controlling such objects are easiest when the distribution of the process is exactly known, for example, when it is a Brownian motion, fractional Brownian motion, or L\'evy process. It is more involved when the process is an additive perturbation of an exactly known process, see e.g. \cite{GG} for a general criterion when this is possible.
%Solutions of equations with multiplicative noise are not of this form, and constructing such integrals become more sophisticated.
%This issue was first addressed recently independently and with different methods in \cite{KM}. 

With the first goal being the construction of $L$, we have now created two new difficulties, which we address one by one below. First of all, the process in \eqref{eq:L-early} along which we wish to integrate is not a perturbation of an exactly known process. Secondly, even if $L$ is constructed, its regularity can be guessed (by replacing $X$ and $Y$ with $B$ and recalling e.g. \cite{CATELLIER20162323}): we can only expect\footnote{For the sake of readability in this introduction we omit the necessary $-\eps$-s in regularities.} $L\in \cC^{1+(\alpha-1)H}$. The condition $\alpha>1-1/(2H)$ guarantees that this exponent is bigger than but arbitrarily close to $1/2$. Therefore rewriting 
the equation for $Z=X-Y$ as \footnote{%{\color{red}For simplicity in the present exposition let us assume that we are in dimension $1$, where one does not (....hmm not perfect.. not consistent with the indices for iterated guys later)..how about this:}\\
Here we admit to being careless in how the indices of different matrix- and vector-valued processes are contracted. If this worries the reader, they may assume for the present exposition that we are in dimension $1$.}
\begin{equ}\label{eq:Z-intro}
Z_t=\int_0^t Z_s\,dL_s+\int_0^t \blue{\Sigma_sZ_s}\,d\blue{B_s},
\end{equ} 
the first integral is a well-defined Young integral only if the regularity of $Z$ is almost $1/2$.
This would be the case, for example, if $\blue{\Sigma}$ were $0$: then $Z$ would inherit the regularity of $L$. Of course the case $\Sigma=0$  corresponds precisely to the additive case so it is not in our focus. 
For the correct power counting, we first realise that $Z$ in general inherits only regularity $H$ from the rough integral, so to define the first integral in \eqref{eq:Z-intro} as a  Young integration, we would need $H+1+(\alpha-1)H>1$, that is, $\alpha>0$. This excludes distributional drift and leads to the same threshold as in \cite{KM}.
To resolve this issue, we will construct a  joint rough path lift of \((L,B)\) and  treat \eqref{eq:Z-intro} as a rough differential equation driven by such emerging joint lift.
% Note that there are two different low-regularity processes driving \eqref{eq:Z-intro}, but at this point they cannot be seen as a single rough path above $(L,B)$ since the cross integrals, i.e. the integral of $L$ against $B$ and vice versa, are not well-defined a priori.

\emph{Conditional densities of RDEs as input to SSL}. Let us  comment on the construction of $L$. In recent years the stochastic sewing lemma of \cite{Khoa} has been rather successful in constructing such additive functionals. The lemma itself gives a simple criterion for when bounds on approximating increments $A_{s,t}\approx L_{t}-L_s$ transfer to bounds on $L$ itself.
The ``art'' is then choosing appropriate $A_{s,t}$, which on one hand can be analysed efficiently, and on the other hand fit in the stochastic sewing framework. 
%These two goals are competing: the better the approximation is, the easier it is expected to sew it to $L$, 
In \cite{KM}, for example, when dealing with $L$ (or rather with an analogous but not completely identical object),  the approximation
\begin{equ}
    A_{s,t}=\int_0^1\E_s\int_{s}^t\nabla b\Big(\theta\big(X_s+\sigma(X_s)(B_r-B_s)\big)+
    (1-\theta)\big(Y_s+\sigma(Y_s)(B_r-B_s)\big)\Big)\,dr\,d\theta
\end{equ}
is used,
where $\E_s$ denotes the conditional expectation given $\cF_s$. %\kl{we should indicate the filtration somewhere earlier} 
Clearly this choice satisfies the first of the two aforementioned criterion: this $A_{s,t}$ is fairly easy to bound since the only randomness to be averaged out is Gaussian. It turns out, however, that this approximation can only be sewed when $\alpha>0$.
In the present paper we chose instead the approximation
\begin{equ}
       A_{s,t}=\int_0^1\E_s\nabla b(\theta \phi_r^{s,X_s}+(1-\theta) \phi_r^{s,Y_s})\,dr\,d\theta,
\end{equ}
where $\phi$ is (the first component of) the flow of the driftless rough differential equation, i.e. of \eqref{eq:main} with $b=0$.
Intuitively it is clear that this is a better approximation to $L_t-L_s$, since $\phi^{s,X_s}_r$ is a better approximation of $X_r$ than $X_s+\sigma(X_s)(B_r-B_s)$ and similarly for $Y$. Indeed, it turns out that the approximating error for the first one is of order $|r-s|^{1+\alpha H}$ (see \cref{Prop:sta-ini}) while  for the second one is of order $|r-s|^{2H}$.  In the range $\alpha> 1-1/(2H)$ and \(r-s\) is small, the first one is clearly better.  However, estimating $A_{s,t}$ becomes much harder. Indeed, we now have to use information about the conditional distribution of the flow given $\cF_s.$\footnote{In fact, we consider convex combinations of the flow starting from different points, and a priori it is not even clear that the nondegeneracy of the two do not cancel each other.}

A detailed theory on the density of the solutions of rough differential equations driven by fractional Brownian motions has been developed in the articles \cite{CF10, Lyons-Cass-Litterer, INAHAMA, CHLT15, GOT}. Although it is natural to expect that any such result also holds with the obvious modifications for conditional densities, making this rigorous is fairly nontrivial due to the lack of Markovianity.
In \cref{sec:partial-malliavin} we approach this question  by developing \emph{partial Malliavin calculus}, as introduced in \cite{stroock, Nualart_PMC}, for rough differential equations.
We comment on an alternative approach to this problem in \cite{cate} in \cref{rem:cate} below.

\emph{An equation driven by a joint rough path.}
As for the second difficulty above, we need another probabilistic construction that goes beyond an analytic threshold, this time for integrals against $B$ and/or $L$: we need to give a meaning to the iterated integrals 
\begin{equ}\label{eq:Q-intro}
\int_s^t(B_r-B_s)\otimes dL_r,\qquad\int_s^t(L_r-L_s)\otimes dB_r. 
\end{equ}
The two integrals are essentially equivalent due to the ``product rule'' 
\begin{equ}
    \int_s^t(B^i_r-B^i_s) dL_r^{jk}+\int_s^t(L_r^{jk}-L_s^{jk}) dB_r^i=(L^{jk}_t-L^{jk}_s)(B^i_t-B^i_s),
\end{equ}
which needs to be enforced since $L$ is a limit of smooth approximations.
Constructing these iterated integrals is our second main step and is the content of \cref{cor:K} below.
%The integral of $B$ against $L$ can then be automatically given by imposing integration by parts:\mate{I guess some transpose or something is missing, the dimensions are a bit screwed up here}\mate{To be rewritten anyway because we define it the opposite way. Plus maybe with other notations...}
%\begin{equ}
%=L_{s,t}\otimes B_{s,t}-\mathbb{Q}_{s,t}.
%\end{equ}
Putting it together with the rest of the objects, we therefore build a rough path of mixed regularity
\begin{equ}
\purp{G}=\Big( B, L, \int B  \otimes dB , \int L  \otimes  dB, \int B   \otimes dL \Big) \in \mathcal{C}^{1+(\alpha-1)H}\times \mathcal{C}^{H}\times \mathcal{C}_2^{2H}\times \mathcal{C}_2^{1+\alpha H}\times \mathcal{C}_2^{1+\alpha H}.
\end{equ}
%satisfying the Chen's relations $\delta \mathbb{B}_{s,u,t}=B_{s,u}\otimes B_{u,t}$, $\delta \mathbb{Q}_{s,u,t}=L_{s,u}\otimes B_{u,t}$, $\delta \tilde{\mathbb{Q}}_{s,u,t}=B_{s,u}\otimes L_{u,t}$.
%of mixed regularity: we impose 
%\begin{equ}
%L\in C^{1+(\alpha-1)H},\quad B\in C^{H},\quad\mathbb{B}\in C_2^{2H},\quad \mathbb{Q}\in C_2^{1+(\alpha-1)H+H}
%\end{equ}
%Here $\mathbb{Q}$ stands for the iterated integral 
Note that the iterated integral of $L$ against itself  is canonically given as a Young integral.% as $1+(\alpha-1)H>1/2$.
\begin{remark}
We are not aware of standard references concerning rough paths of mixed regularity, but one can easily transfer the tools from the usual theory without any difficulty.
In any case, since regularity structures \cite{H0} accommodate different regularities, viewing rough paths as a special case of regularity structures, one may take all standard results as valid also in such mixed regularity setting.
\end{remark}
Having given a meaning to \eqref{eq:Z-intro} as a rough differential equation above the rough path $\purp{G}$, it might seem that $Z\equiv 0$ follows.
%It remains to show that $\cZ=(Z,Z,Z\Sigma)$, as a path controlled by $L$ and $B$, solves the equation
%\begin{equ}
%\cZ_t=\int_0^t\cZ_s\,d\mathbf{L}_s+\int_0^t\cZ_s\Sigma_s\,d\mathbf{B}_s.
%\end{equ}
%This equation is now perfectly well-defined and well-posed, therefore $\cZ$ has to coincide with the trivial solution $0$, 

\emph{Closing the equation}.
Unfortunately, the last sentence of optimism is misguided: from the preceding construction it does \emph{not} follow that there is a controlled path above $Z$ that solves \eqref{eq:Z-intro} as a rough differential equation above $\purp{G}$!
Indeed, %the natural candidate for the Gubinelli derivatives with respect to $B$ and $L$ would be $\Sigma Z$ and $Z$, respectively,
from \eqref{eq:Z-intro} it is clear that $Z$ is controlled by $(B,L)$ with Gubinelli derivatives $(\Sigma Z,Z)$,
so writing $\purp{Z}=(Z,Z,\Sigma Z)$, what the above optimistic sentence is really saying is that
\begin{equ}\label{eq:Z-intro2}
Z_t=\int_0^t\purp{Z_s}\,d\purp{L_s}+\int_0^t\purp{\Sigma_s}\purp{Z_s}\,d\purp{B_s}
\end{equ}
implies $Z=0$ (here $\purp{\Sigma}$ is the trivial extension of $\blue{\Sigma}$ from a controlled path by $B$ to a controlled path by $(B,L)$).
%(i.e. taking its Gubinelli derivative with respect to $L$ to be $0$).
Although \eqref{eq:Z-intro2} indeed implies $Z=0$, it is far less clear that \eqref{eq:Z-intro} implies \eqref{eq:Z-intro2}. Indeed, at this point the role of different colors starts being important, and it is in general not true that two integrals written in blue and purple (or in black and purple) coincide, even if both integrals are meaningful.
%Indeed, as a controlled path, the right-hand side of \eqref{eq:Z-intro} gives $\cZ=(Z,Z,\Sigma Z)$ as an output, while the inputs for the integrands are $\hat \cZ=(Z,0,0)$ and $\tilde \cZ=(Z,0,Z\Sigma)$, respectively (here we write the Gubinelli derivatives with respect to $L$ and $B$ as the second and third coordinate of a controlled path, respectively).\mate{Not sure what I'm saying here is correct, to be double checked in the end}
On the other hand, there is one more resource not used so far: the ``drift'' part of $Z$ does not merely have the regularity of $L$ (i.e. $1+(\alpha-1)H$), but the regularity of the individual drifts of $X$ and $Y$, which is of higher order, $1+\alpha H$. 
The last main task of the proof is to show that this extra regularity can be leveraged to go from \eqref{eq:Z-intro} to \eqref{eq:Z-intro2}.
%We show that this can be leveraged in the construction of $\mathbb{Q}$ above to get a rough path lift of $(L,B)$ such that $\cZ=(Z,Z,Z\Sigma)$ indeed solves the linear rough differential equation \eqref{eq:Z-intro} and hence vanishes.
%this extra information allows us to show that even if $Z$ only solves \eqref{eq:Z-intro} in a ``nonstandard'' way, it is enough to conclude $Z=0$.
%It is \emph{not} a closed equation! Indeed, as a controlled path, the right-hand side gives $\cZ=(Z,Z,Z\Sigma)$, while the integrands under the are $\tilde \cZ=(Z,0,Z\Sigma)$. Therefore, the last step of the proof is to show that this information is sufficient to conclude that $Z=0$.
This implies strong uniqueness, after which strong existence follows from the Gy\"ongy-Krylov lemma \cite{Gyongy_Krylov} and the weak existence result from \cite{KM}.

\subsection{Setup and notation}

\emph{Classical function/distribution spaces}.
For $\alpha\in(0,1]$,  we set $\mathcal{C}^\alpha(\mR^d)$ to be the space of continuous functions $f : \R^d \to \R$ such that
\begin{align*}
  \Vert f\Vert_{\mathcal{C}^\alpha}:=
    [ f]_{\mathcal{C}^\alpha}+\sup_{x\in\mR^d}|f(x)|:=\sup_{x,y\in\mR^d,x\neq y}\frac{|f(x)-f(y)|}{|x-y|^\alpha}+\sup_{x\in\mR^d}|f(x)|<\infty.
\end{align*}
Here, and often below, we write $\mathcal{C}^\alpha$ instead of $\mathcal{C}^\alpha(\mR^d)$ for simplicity.
For $\alpha\in(0,\infty)$, we define $\mathcal{C}^\alpha(\R^d)$ the space of all functions $f$ defined on $\R^d$ having bounded derivatives $\partial^k f$ for multi-indices $k\in\N^d$ with $|k|:=|k_1|+\cdots+|k_d|\leq \alpha$ such that
\begin{align*}
  \Vert f\Vert_{\mathcal{C}^\alpha}:=\sum_{|k|\leq \alpha}\sup_{x\in\mR^d}|\partial^kf(x)|+\sum_{\alpha-1\leq|k|<\alpha} \Vert \partial^k f\Vert_{\mathcal{C}^{\alpha-|k|}}<\infty.
\end{align*}
Note that the $\mathcal{C}^\alpha$-norm always includes the supremum of the function.
 We also denote the space of bounded measurable functions equipped with the supremum norm by $\mathcal{C}^0(\mR^d)$. We denote by $\cC^\infty=\cap_{\alpha>0}\cC^{\alpha}$ the space of smooth functions and by $\cC^{\alpha+}$ the closure of $\cC^{\infty}$ in $\cC^\alpha$. It will often be more convenient to work with the $\cC^{\alpha+}$ spaces, the fact that this is not a loss of generality is a consequence of the well known fact that for any $\eps>0$, $\cC^{\alpha+\eps}\subset\cC^{\alpha+}$ (and that in the relevant assumptions strict inequality will be assumed).
 For $U\subset \R^d$ and a Banach space $V$, the extension of the definition to functions $f:U\to V$ is straightforward and is denoted by $\mathcal{C}^\alpha(U;V)$. If the target space $V$ is obvious from the context, we simply write $\cC^\alpha(U)$ and if the same is true for the domain we write simply $\cC^\alpha$, with some abuse of notation. Further, $\alpha \in (0, 1]$ and $(S, T) \in [0,1]_{\leq}^2$, we denote by $\cC^\alpha_2([S,T]^2; V)$ the collection of all $\mathbb{f}:[S, T]^2 \mapsto V$ such that 
 \begin{equs}
    \,  [\mathbb{f}]_{\cC^\alpha_2}:= \sup_{(s, t) \in [S,T]^2_\leq}\frac{|\mathbb{f}_{s,t}|}{|t-s|^\alpha}.
 \end{equs}

For $\alpha<0$ we denote by $\mathcal{C}^\alpha(\mR^d)$  the space of all Schwarz distributions such that
\begin{equ}\label{def:C-}
  \Vert f\Vert_{\mathcal{C}^\alpha}:=\sup_{\gamma\in(0,1]}\gamma^{-\frac{\alpha}{2}}\Vert P_\gamma f\Vert_{\mathcal{C}^0}<\infty,
\end{equ}
where $P_tf:=p_{t}\ast f$ and $p_t(x):=\frac{1}{\sqrt{2\pi t}^d}e^{-\frac{|x|^2}{2t}}$. Recall the standard heat kernel bound 
\begin{align}
    \label{est:2derivate-heatk}
    \Vert \partial^k p_t\Vert_{L^1(\R^d)}\leq N(d,k) t^{-|k|/2}.
\end{align}

%\subsection{Rough path spaces}  \label{sec:rp_notation}
\emph{Rough path spaces.}
Below $V_1,V_2,V_3$ are all finite dimensional Euclidean spaces and $\cL(V_1;V_2)$ denotes the space of bounded linear operators from $V_1$ to $V_2$.
Recall that there is a canonical isometry between $\cL(V_1;\cL(V_2,V_3))$ and $\cL(V_1\otimes V_2;V_3)$.
For a function $f$ in $1$ variable we denote $f_{s,t}=f_t-f_s$ and for a function $g$ in $2$ variables we denote $\delta g_{s,u,t}=g_{s,t}-g_{s,u}-g_{u,t}$. For an interval $I$ we denote by $\pi_I$ the set of partitions of $I$, that is, collections $\cP$ of closed subintervals of $I$ whose union is $I$ and whose interior is mutually disjoint. For $\cP\in\pi_I$ we denote by $|\cP|$ the length of the largest element of $\cP$.
We loosely follow the convention of \cite{Friz-Hairer} by denoting rough and controlled paths with colors.
This will help distinguish between integrals that have different interpretations (e.g. Young and rough), which, as indicated in the introduction, is crucial in the final part of the proof \cref{sec:closing}.

Let $(S, T) \in [0, 1]^2_{\leq}:=\{(s,t)\in[0,1]^2:\,s\leq t\}$. For $\gamma \in(1/3,1/2]$ we denote by $\mathcal{R}^\gamma([S,T];V_1)$ the set of $\gamma$-H\"older rough paths on $[S,T]$: the subset of $\cC^\gamma([S,T],V_1)\times \cC^{2\gamma}_2([S,T]^2,V_1\otimes V_1)$ 
constrained by the nonlinear relation (\emph{Chen's identity}) postulating that any ${\color{blue}g}=(g,\bbg)\in \mathcal{R}^\gamma$ satisfies
\begin{equ}      \label{eq:Chen_single_regularity}
\bbg_{s,t}-\bbg_{s,u}-\bbg_{u,t}=g_{s, u}\otimes g_{u, t} 
\end{equ}
for all $(s,u,t)\in[S,T]^3_\leq:=\{(s,u,t)\in[S,T]^3:\,s\leq u\leq t\}$.  Such ${\color{blue}g}=(g,\bbg)\in \mathcal{R}^\gamma$ is also called a \emph{lift} of $g$. For $\blue{g}=(g, \mathbb{g}), \blue{h}=(h, \mathbb{h}) \in \mathcal{R}^\gamma([S,T];V_1)$, we set 
\begin{equs}
  \,   [\blue{g}]_{\mathcal{R}^\gamma([S,T])}&:= [g]_{\cC^\gamma([S,T])}+[\mathbb{g}]_{\cC^{2\gamma}_2([S,T]^2)},
  \\
  d_\gamma(\blue{g},\blue{h}) &:=  \|g-h\|_{\cC^\gamma([S,T])}+[\mathbb{g}-\mathbb{h}]_{\cC^{2\gamma}_2([S,T]^2)}. 
\end{equs}
The map $d_\gamma$ is a metric on  $\mathcal{R}^\gamma([S,T];V_1)$.  We denote by $\mathcal{R}^\gamma_{\geo}([S,T];V_1)$ the set of geometric rough paths, that is, the closure in $\mathcal{R}^\gamma([S,T];V_1)$ of the set 
\begin{equs}
    \Big\{ (g, \mathbb{g}) : g \in \cC^1([0,1]; V_1), \mathbb{g}_{s,t}= \int_s^t g_{s,r} \otimes dg_r \Big\}. 
\end{equs}
It is known that $(\mathcal{R}^\gamma_{\geo}([S,T];V_1), d_\gamma)$ is a Polish space (see, e.g., \cite[Excercise 2.8]{Friz-Hairer}). 

Given $\gamma \in (1/3, 1/2]$ and $\blue{g}= (g, \mathbb{g}) \in\cR^{\gamma}([S,T];V_1)$,  we denote by $\cD^{2\gamma}_{g}([S,T];V_2)$ the set of all functions $\blue{f}=(f,f'):[S,T]\to V_2\times\cL(V_1;V_2)$ such that
\begin{equ}
     \, [\blue{f}]_{\mathcal{D}_{g}^{2\gamma}([S, T])}:=[ R^{\blue{f}}]_{\cC^{2\gamma}_2([S, T])} +[f']_{\cC^{\gamma}([S, T])}< \infty, 
\end{equ}
where $R^{\blue{f}}_{s,t}= f_{s,t}-  f'_s g_{s, t}$.
Moreover, let us set 
\begin{equs}
    \|\blue{f}\|_{\mathcal{D}_{g}^{2\gamma}([S, T])} = [\blue{f}]_{\mathcal{D}_{g}^{2\gamma}([S, T])}+ \|f\|_{\cC^0([S, T])}+\|f'\|_{\cC^0([S, T])}. 
\end{equs}
Notice that by the triangle inequality we have 
\begin{equs}  \label{eq:holder_bound_by_rough}
    [f]_{\cC^\gamma([S, T])} \leq  [R^\blue{f}]_{\cC^{2\gamma}_2([S, T])}|T-S|^{\gamma} + \|f'\|_{\cC^0([S, T])}[g]_{\cC^\gamma([S, T])}. 
\end{equs}
Note that for $k\geq1$ and $0 \leq u_0\leq u_1\leq\cdots\leq u_k$ one has
\begin{equs}     
\, [\blue{f}]_{ \cD^{2\gamma}_g([u_0, u_k])}  &  \leq 
\sum_{i=1}^k [\blue{f}]_{ \cD^{2\gamma}_g([u_{i-1},u_i])} +\sum_{i=1}^{k-1}
[ f' ]_{\cC^{\gamma}([u_{i-1},u_i])} [ g]_{\cC^\gamma([u_i,u_k])}
\\
& \leq 2 (1+[g]_{\cC^\gamma([u_0, u_k])} )  \sum_{i=1}^k [\blue{f}]_{ \cD^{2\gamma}_g([u_{i-1}, u_i])}.   \label{eq:sum_partition}
\end{equs}
Let $V_3$ be another finite dimensional Euclidean space. If $F \in \cC^2(V_2; V_3)$, then it holds that $\blue{F(f)}:= (F(f), \nabla F (f) f') \in \cD^{2\gamma}_{g}([S,T]; V_3)$ 
and the following estimates holds (see, e.g., \cite[equation (1.8) and the Appendix therein]{KM})
\begin{equs}
\,\!\!  [R^\blue{F(f)}]_{\cC^{2\gamma}_2([S, T])}  &\leq  N(  [R^\blue{f}]_{\cC^{2\gamma}_2([S, T])}+  \| f'\|_{\cC^0([S, T])}[f]_{\cC^{\gamma}([S, T])}[g]_{\cC^\gamma([S, T])}) ,\label{eq:composition_Holder_R_primitive}
\\
   \!\!\!\!\! [\blue{F(f)}]_{\mathcal{D}_{g}^{2\gamma}([S, T])}  &\leq N ( 1+\|f'\|^2_{\cC^0([S, T])}+ [g]^2_{\cC^\gamma([S, T])} ) ( [\blue{f}]_{\mathcal{D}_{g}^{2\gamma}([S, T])}+ \|f'\|_{\cC^0([S, T])}),
   \label{eq:composition_estimate_controlled}
\end{equs}
where $N$ depends only on $\|F\|_{\cC^2}$.  From \eqref{eq:composition_Holder_R_primitive} and \eqref{eq:holder_bound_by_rough}, we also have 
\begin{equs}
\,  [R^\blue{F(f)}]_{\cC^{2\gamma}_2([S, T])} \leq  N(   &[R^\blue{f}]_{\cC^{2\gamma}_2([S, T])}+ [R^\blue{f}]_{\cC^{2\gamma}_2([S, T])} \| f'\|_{\cC^0_{S,T}}[g]_{\cC^\gamma([S, T])}\\&+  \| f'\|^2_{\cC^0([S, T])}[g]^2_{\cC^\gamma([S, T])} ). \label{eq:composition_Holder_R}
\end{equs}
For $\blue{f}=(f,f') \in \cD^{2\gamma}_{g}([S,T];V_2)$ and $ \blue{h} =(h, h')  \in  \cD^{2\gamma}_{g}([S,T];\mathcal{L}(V_2; V_3))$, we set 
\begin{equs}           \label{eq:def_blue_product}
     \blue{hf}:= \Big(h f,  v_1 \mapsto \big( (h'v_1) f+ h  (f'v_1) \big) \Big).
\end{equs}
When then has that  $\blue{hf}\in \cD^{2\gamma}_{g}([S,T];V_3)$. Indeed, 
one has
\begin{equs}
    R^{\blue{hf}}_{s,t}=  h_s R^{\blue{f}}_{s,t}+R^{\blue{h}}_{s,t}f_s+(h'_sg_{s,t})f_{s,t},
\end{equs}
which implies that 
\begin{equs}
 \,    [R^{\blue{hf}}]_{\cC^{2\gamma}_2([S, T])} & \leq   \|h\|_{\cC^0([S,T])}[R^{\blue{f}}]_{\cC^{2\gamma}_2([S, T])} +[R^{\blue{h}}]_{\cC^{2\gamma}_2([S, T])}\|f\|_{\cC^0([S,T])}
    \\
    & \qquad  +\|h'\|_{\cC^0([S,T])}[f]_{\cC^{\gamma}([S, T])}[g]_{\cC^{\gamma},([S, T])}                    \label{eq:reminder_product}
    \\
    \, [\blue{hf}]_{\cD^{2\gamma}_{g}([S,T])} &  \leq  \|h\|_{\cC^0([S,T])}[\blue{f}]_{\cD^{2\gamma}_{g}([S,T])}+[\blue{h}]_{\cD^{2\gamma}_{g}([S,T])}\|f\|_{\cC^0([S,T])}
    \\
    & \qquad  +\|h'\|_{\cC^0([S,T])}[f]_{\cC^{\gamma}([S, T])}[g]_{\cC^{\gamma},([S, T])}. 
\end{equs}
The latter combined with \eqref{eq:holder_bound_by_rough}  further implies that 
\begin{equs}     \label{eq:controlled_product}
\|\blue{hf}\|_{\cD^{2\gamma}_{g}([S,T])} \leq N \| \blue{h}  \|_{\cD^{2\gamma}_{g}([S,T])}\| \blue{f}  \|_{\cD^{2\gamma}_{g}([S,T])} (1+[g]_{ \mathcal{C}^\gamma([S, T])})^2,
\end{equs}
with a universal constant $N$. 

For  $f , h \in  \cD^{2\gamma}_{g}([S,T];V_2)$ and $F \in \cC^3(V_2, V_3)$, by means of the usual fundamental theorem of calculus (FTC)  and integration by parts,  one has the following rough version of FTC 
\begin{equ}        \label{eq:rough_FTC}
   \blue{F(f)-F(h)}  = \int_0^1 \blue{\nabla F(\theta f+(1-\theta)h)} \, d\theta  \blue{(f-h)},
\end{equ}
where 
 the product is as in \eqref{eq:def_blue_product} and the composition with $\nabla F$ is as before.

For $\blue{f}=(f,f')\in\cD^{2\gamma}_{g}([S,T];\cL(V_1;V_2))$ one can define the $V_2$-valued rough integral of $\blue{f}$ with respect to $\blue{g}$ by setting for $t \in [S, T]$
\begin{equ}
\int_S^t\blue{f_r}\,d\blue{g_r}:=\lim_{\substack{\cP\in\pi_{[S,t]}\\ |\cP|\to0}}\sum_{[s,u]\in\cP} A_{s,u},
\end{equ}
where $A_{s,u}:=f_sg_{s,u}+ f_s'\bbg_{s,u}$. 
The fact that this limit exists is a standard consequence of the sewing lemma and the estimate for $(s,u,t)\in[S,T]_\leq^3$
\begin{equs}
    |\delta A_{s,u,t}|& = | (f_{s, u}- f'_s g_{s, u}) g_{u, t}+  f'_{s,u} \mathbb{g}_{u,t}| 
    \\
    & \leq 2[\blue{f}]_{\mathcal{D}_{g}^{2\gamma}([s, t])} [\blue{g}]_{\mathcal{R}^\gamma([s,t])} |t-s|^{3\gamma},
\end{equs}
where Chen's identity  has been used. Another consequence of the sewing lemma is that the following  estimate holds for all $(s, t) \in [S, T]^2_{\leq}$
\begin{equs}
    \Big| \int_s^t\blue{f_r}\,d\blue{g_r}-f_sg_{s,t}- f_s'\bbg_{s,t} \Big|
 \leq N   [\blue{g}]_{\mathcal{R}^\gamma([s, t])}[\blue{f}]_{\mathcal{D}_{g}^{2\gamma}([s, t])} |t-s|^{3\gamma}, 
\end{equs}
where $N=N(\gamma)$. In addition, the above estimate obviously implies that for all $(s, t) \in [S, T]^2_{\leq}$
\begin{equ}           \label{eq:rough_integral_first_order}
    \Big| \int_s^t\blue{f_r}\,d\blue{g_r}-f_sg_{s,t} \Big| \leq N  [\blue{g}]_{\mathcal{R}^\gamma([s,t])} \big( [\blue{f}]_{\mathcal{D}_{g}^{2\gamma}([s, t])}|t-s|^{3\gamma} + |f'_s|    |t-s|^{2\gamma} \big) ,
\end{equ}
with $N=N(\gamma)$.  From \eqref{eq:rough_integral_first_order} combined with \eqref{eq:holder_bound_by_rough} one can conclude that $ \blue{\int_S^\cdot f_r d g_r}:=(\int_S^\cdot \blue{f_r} d \blue{g_r}, f) \in\cD^{2\gamma}_{g}([S,T];V_2)$, and the following estimate holds
\begin{equs}  \label{eq:boundedness_rough_integration}
\Big\|\blue{\int_S^\cdot f_r d g_r}\Big\|_{\cD^{2\gamma}_{g}([S,T])} \leq N (\|\blue{f}\|_{\cD^{2\gamma}_{g}([S,T])} +1) [\blue{g}]_{\mathcal{R}^\gamma([S,T])},
\end{equs}
with $N=N(\gamma)$. 
For $\rho>0,$ we denote $\mathcal{V}^{\rho}([S,T])$  the collection of  all finite $\rho$-variation paths $g:[S,T]\mapsto W$ so that 
\begin{equ}\label{def:var-1}
    [g]_{\mathcal{V}^{\rho}([S,T])}:=\sup_{\substack{\cP\in\pi_{[S,T]}}}\Big(\sum_{[s,u]\in\cP} |g_{s,u}|^\rho\Big)^\frac{1}{\rho}<\infty. 
\end{equ}
By $\mathcal{V}^{\rho}([S,T]^2)$ we denote the collection of all $g:[S,T]^2\mapsto W$ such that 
\begin{equ}\label{def:var-2}
    [g]_{\mathcal{V}^{\rho}([S,T]^2)}:=\sup_{\substack{\cP_1\in\pi_{[S,T]}\\\cP_2\in\pi_{[S,T]}}}\Big(\sum_{\substack{{[s_1,u_1]\in\cP_1}\\{[s_2,u_2]\in\cP_2}} }\big|g_{s_1,s_2}+g_{u_1,u_2}-g_{s_1,u_2}-g_{u_1,s_2}\big|^\rho\Big)^\frac{1}{\rho}<\infty.
\end{equ}

\emph{Probabilistic setup}.
 We fix a probability space $(\Omega,\cF,\P)$ with a complete filtration $\bF=(\cF_t)_{t\in [0,1]}$ carrying a $d_0$-dimensional $\bF$-Wiener process $W$.
For $H\in(0,1/2)$ we consider a fractional Brownian Motion $B$ given by 
\begin{equ}
B_t=\int_0^t K(t, s) dW_s,
\end{equ}     \label{eq:mandelbrot}
with the kernel $K$ given by 
    \begin{equs}         
    K(t,s)& = c_H \Big[  \left( \frac{t}{s} \right)^{H-1/2} (t-s)^{H-1/2} -(H-1/2)s^{1/2-H} \int_s^t u^{H-3/2}(u-s)^{H-1/2} \, du \Big] \bone_{s<t}
    \\
    \label{eq:def_K}
\end{equs}
where  $c_H= \left(2H/\big((1-2H)\beta(1-2H, H+1/2)\big) \right)^{1/2} >0$, 
and $\beta$ is the Beta function. It is known (see also \cref{rem:same_filtration} below) that $B$ and $W$ generate the same filtration, whose completion will be denoted by $\mathbb{F}^B=(\cF^B_t)_{t \in [0,1]}$.  We clearly have $\cF^B_t \subset \cF_t$ for all $t \in [0,1]$. 
The conditional expectation given $\cF_s$ is denoted by $\E_s$ while conditional expectation given $\cF^B_s$ will be denoted by $\E_s^B$. 

We denote by $\blue{B}:=(B, \mathbb{B})$ the Gaussian rough path lift of $B$ (see, e.g., \cite{FV10}) and recall that  $\blue{B} \in \mathcal{R}^\gamma_{\geo}([0, 1]; \R^{d_0})$ with probability one, for all $\gamma \in (1/3, H)$. From now on, we fix $H_- <H$ and $H_+>H$ sufficiently close to $H$, such that 
\begin{equs}        \label{eq:Choice of H-}
  3H_- >1,  \qquad 1+(\alpha-1)H_+>1/2. 
\end{equs}
The existence of such $H_+$ is guaranteed under the assumption $\alpha>1-1/(2H)$ which is in force (see \cref{asn:b} below). 

\subsection{Formulation}

At this point we can formulate the definition of a solution, the main assumptions, and the main result on \eqref{eq:main}.

\begin{definition}     \label{def_solution}
An $\mathbb{F}$-adapted  stochastic process $(X_t)_{t \in [0, 1]}$ is called a solution of \eqref{eq:main} if the following are satisfied:
\begin{enumerate}[(i)]
\item There exists an $\mathbb{F}$-adapted  stochastic process $(D^X_t)_{t \in [0, 1]}$ such that $\int_0^\cdot b^n(X_s) \, ds \to D^X $ uniformly in time, in probability, whenever $(b^n)_{n\in\N} \subset \cC^\infty$ and $b^n \to b$ in $\cC^\alpha(\R^d)$. \label{item:def_D}

\item With probability one we have $\blue{X}:=(X, \sigma(X)) \in \mathcal{D}^{2H_-}_B([0, 1])$, and  for all $p\geq 1$, we have 
\begin{equs}
        C_{S,p}^X:= \| \blue{X}\|_{L_p(\Omega ; \mathcal{D}^{2H_-}_B([0, 1])) } < \infty.  \label{eq:bound_X_controlled}
     \end{equs}

    \item For all $p \geq 1$, we have 
    \begin{equs}
     C_{D,p}^X:= [D^X]_{\mathcal{C}^{1+\alpha H}([0, 1]; L_p(\Omega))} < \infty.     \label{eq:bound:D^X} 
    \end{equs}

    \item With probability one, for all $t \in [0, 1]$, 
    \begin{equs}
        X_t=D^X_t+ \int_0^t \blue{\sigma(X_s)} \, d \blue{B_s}.
    \end{equs}
\end{enumerate}
If the above holds, we also say \emph{$X$ is a solution with drift component $D^X$}.
\end{definition}

The initial condition is chosen as $0$ purely for convenience and will play no role in the article. As for the coefficients, we impose the following.
\begin{assumption}\label{asn:b}
For some $\alpha>1-1/(2H)$ one has $b\in \cC^{\alpha+}(\R^d;\R^d)$.
\end{assumption}
\begin{assumption}\label{asn:sigma}
One has $\sigma\in\cC^\infty(\R^d;\R^{d\times d_1})$. Furthermore, there exists a constant $\lambda>0$ such that for all $x\in\R^d$, $\sigma(x)\sigma^*(x)\succeq \lambda \mathbb{I}_{d}$.
\end{assumption}
Our main theorem is the following. 
\begin{theorem}\label{thm:main}
Let \cref{asn:b} and \cref{asn:sigma} hold. Then there exists a strong solution $X$ to \eqref{eq:main} and for any other solution $Y$, one has $\P(X=Y)=1$.
\end{theorem}

\section{Preliminaries}
\subsection{Stochastic sewing}
We use the original version of L\^e's stochastic sewing lemma (SSL) from \cite{Khoa}. % and  the extension developed by Matsuda and Perkowski \cite{MR4730255}.
We start by introducing a few common notations for stochastic sewing. For $S<T$ and $i\in\N$ we denote $[S,T]_{\leq}^i:=\{(t_0,\ldots,t_{i-1})\in[S,T]^i:\,t_0\leq \cdots\leq t_{i-1}\}$. %and $\widehat{[S,T]}_{\leq}^i=\{(t_0,\ldots,t_{i-1})\in[S,T]^i_{\leq}:\,|t_{i-1}-t_1|\leq (i-1)|t_1-t_0|\}$
\begin{lemma}\label{lem:SSL-vanila}
Let $p\in[2,\infty)$.   Let $(S,T)\in[0,1]_\leq^2$ and let
$(A_{s,t})_{(s,t)\in[S,T]_\leq^2}$ be a family of $\R^d$-valued random variables 
%satisfying $A_{s,s}=0$, for all $s\in[S,T]$ and 
such that $A_{s,t}$ is $\mathcal{F}_t$-measurable for all $(s,t)\in[S,T]_{\leq}^2$.
Suppose that there exist constants $\Gamma_1,\Gamma_2\in[0,\infty)$, $\beta_1>1/2$, and $\beta_2>1$ such that the following holds:
\begin{enumerate}[(i)]
    \item $\|A_{s,t}\|_{L^p(\Omega)}\leq \Gamma_1|t-s|^{\beta_1}$, $(s,t)\in[S,T]_{\leq}^2$,
    \item $\|\E_s\delta A_{s,u,t}\|_{L^p(\Omega)}\leq \Gamma_2|t-s|^{\beta_2}$, $(s,u,t)\in[S,T]_{\leq}^3$. 
\end{enumerate}
  Then there exists a unique continuous $(\mathcal{F}_t)_{t\in[S,T]}$-adapted process $\mathcal{A}:[S,T]\mapsto L^p(\Omega;\R^d)$  such that $\mathcal{A}_S=0$ and the following bounds hold for some constants $K_1,K_2, K>0$: $(s,t)\in[S,T]_{\leq}^2$,
  \begin{enumerate}[(I)]
      \item $\|\mathcal{A}_t-\mathcal{A}_s- A_{s,t}\|_{L^p(\Omega)}\leq K_1|t-s|^{\beta_1}+ K_2|t-s|^{\beta_2}$, 
      \item $\|\E_s(\mathcal{A}_t-\mathcal{A}_s- A_{s,t})\|_{L^p(\Omega)}\leq K_2|t-s|^{\beta_2}$.
\end{enumerate}
Moreover, there exists a constant $K=K(\beta_1,\beta_2,p,d)$ such that the above bounds hold with $K_1=K\Gamma_1$, $K_2=K\Gamma_2$, and  one has for all $(s,t)\in[S,T]_{\leq}^2$
\begin{enumerate}
    \item[(III)]  $\|\mathcal{A}_t-\mathcal{A}_s\|_{L^p(\Omega)}\leq K\Gamma_1|t-s|^{\beta_1}+K\Gamma_2|t-s|^{\beta_2}$.
    %Here $K$ depends only on $\beta_1,\beta_2,d$.
  \end{enumerate}
\end{lemma}

\subsection{Tools from Malliavin calculus}
We collect some basic results concerning  Malliavin calculus for the fractional Brownian motion (see  \cite{Nu09, CHLT15, CF10, FV10}). In this section, we work on the probability space $(\Omega, \cF^B_1, \mathbb{P})$, that is, we consider random variables that are measurable functions of the underlying fractional Brownian motion $(B_t)_{t \in [0,1]}$. On the basis of the $d_0$-dimensional fractional Brownian motion $B_t=(B^1_t, \dots, B^{d_0}_t)$ we construct an isonormal process on a Hilbert space $\mathcal{H}$ (see below). With some abuse of notation we will use $B$ both for the $d_0$-dimensional process $ [0,1] \ni t \mapsto B_t$ and for the isonormal process $ \mathcal{H} \ni h \mapsto B(h)$. 

Let us denote by $\mathcal{H}$ the completion of the set 
\begin{align*}
    \mathcal{E}:=\left\{ t \mapsto \sum_{k=1}^ma_k\mathbf{1}_{[0,t_k]}(t) : m\in\N_+, a_k\in\R^{d_0}, t_k\in [0, 1]\right\}
\end{align*}
with respect to the norm 
\begin{equs}       \label{eq:norm_H}
    \|\sum_{k=1}^ma_k\mathbf{1}_{[0,t_k]} \|_{\mathcal{H}}^2=\sum_{k,\ell=1}^m\langle a_k\mathbf{1}_{[0,t_k]},a_\ell\mathbf{1}_{[0,t_\ell]}\rangle_{\mathcal{H}}:=\sum_{k,\ell=1}^ma_k\cdot a_\ell Q(t_k,t_\ell),
\end{equs}
where $Q(s,t)= 2^{-1}(s^{2H}+t^{2H}-|t-s|^{2H})$
is the covariance function of the process $t \mapsto B^1_t$. It is evident from definition that \(\mathcal{H}\) is a separable Hilbert space.  Next, we consider the linear map  $B: \mathcal{E} \to L_2(\Omega)$ defined by 
\begin{equs}
    B( a \mathbf{1}_{[0, t]}) = \sum_{i=1}^{d_0}a^i B^i_t,
\end{equs}
for $a=(a^1, \dots, a^{d_0})$ and $t \in [0,1]$. Keeping in mind \eqref{eq:norm_H} and the fact that $\E B^i_sB^j_t= \delta_{ij} Q(s,t)$,  the map $B$ extends to a linear isometry $B : \mathcal{H} \to L_2(\Omega)$.

We denote by $\mathcal{S}$ the collection of real  random variables given by 
\begin{equs}
    \mathcal{S}=\left\{ f(B(h_1),\ldots,B(h_n)): n \in \mathbb{N},  f\in \mathcal{C}_{\pol}^\infty(\R^n), h_i\in\mathcal{H} \right\},
\end{equs}
where $\mathcal{C}_{\pol}^\infty(\R^n)$ denotes the set of smooth real functions on $\R^n$ whose derivatives of any order grow at most polynomialy. 
Similarly, for a  separable Hilbert space $(V, \| \cdot\|_{V})$, let us denote by  $\mathcal{S}(V)$ the collection of smooth  $V$-valued random variables
\begin{equ}  \label{eq:defSV}
    \mathcal{S}(V)=\left\{ \sum_{i=1}^n F_i v_i : n \in \mathbb{N}, F_i \in \mathcal{S}, v_i \in V\right\}.
\end{equ}
The set $\mathcal{S}(V)$ is dense in $L_p(\Omega;V)$ for any $p \geq 1$. 
The \emph{Malliavin derivative} is a linear operator  $D: \mathcal{S}(V) \to \mathcal{S}(V\otimes \mathcal{H})$ given by 
\begin{equ}
DF=\sum_{k=1}^n\partial_kf(B(h_1),\ldots,B(h_n))v\otimes h_k, 
\end{equ}
for $F=f(B(h_1), \dots , B(h_n)) v$, $f \in \cC^\infty_{\text{pol}}(\R^n)$, $v \in V$. 
The operator  $D$ as a mapping from $\mathcal{S}(V)$ to $L_2(\Omega;V \otimes \mathcal{H})$ is closable and its domain, which we will denote by $\mathbb{D}^{1,2}(V)$,   is given by  the completion of $\mathcal{S}(V)$ with respect to the norm
\begin{equ}
    \Vert F\Vert_{\mathbb{D}^{1,2}(V)}:=\big(\E\|F\|_V^2+\E(\Vert DF\Vert_{V \otimes \mathcal{H}}^2)\big)^{1/2}.
\end{equ}
For $F \in \mathbb{D}^{1,2}(V)$ and $h \in \mathcal{H}$, we denote by $D_hF$ the $V$ -valued random variable 
\begin{equs}   
    D_hF= \langle DF, h \rangle_{\mathcal{H}} := \sum_{i,k} \langle DF,  q_k \otimes e_i \rangle_ {V \otimes \mathcal{H}}  \langle e_i, h \rangle_{\mathcal{H}} q_k,
\end{equs}
where $(q_k)_{k=1}^\infty$, $(e_i)_{i=1}^\infty$, are orthonormal bases of $V$, $\mathcal{H}$, respectively. Notice that it does not depend on the choice of the bases. 
The adjoint of $D$ will be denoted by $\delta : \dom(\delta) \subset L_2(\Omega;V \otimes \mathcal{H}) \to L_2(\Omega; V)$. By definition, for $u \in \dom(\delta)$, $F \in \mathbb{D}^{1,2}(V)$, we have the following \emph{integration by parts} formula 
\begin{equ}\label{eq:ini-IBP}
    \E\langle DF,u\rangle_{ V \otimes \mathcal{H}}=\E\langle F, \delta(u)\rangle_V.
\end{equ}
If $u \in L_2(\Omega; V \otimes \mathcal{H})$ is of the form 
\begin{equs}
    u= \sum_{i=1}^n F_i v_i \otimes h_i, \qquad F_i \in \mathcal{S}(\R), \,  v_i \in V,\,  h_i \in \mathcal{H},
\end{equs}
then $u \in \dom(\delta)$ and 
\begin{equs}    \label{eq:explicit_formula_delta}
    \delta(u)= \sum_{i=1}^n (F_i B(h_i)- \langle DF_i, h_i\rangle_{\mathcal{H}}) v_i. 
\end{equs}
For $u$ as above, by virtue of \eqref{eq:explicit_formula_delta}, it is easy to verify that 
\begin{equs}                 \label{eq:Heisenberg}
    D \delta(u) = u+ \delta (\widehat{Du}),
\end{equs}
where the map $\widehat{\cdot} : V \otimes \mathcal{H} \otimes \mathcal{H} \to  V \otimes \mathcal{H} \otimes \mathcal{H}$  is the linear isometry given by 
$\widehat{v \otimes h_1 \otimes h_2}= v \otimes h_2 \otimes h_1$ for $v \in V, h_1, h_2 \in \mathcal{H}$. 
Moreover, it is known (\cite{Nu09}) that 
\begin{equ}\label{eq:domain-inclusion}
    \mathbb{D}^{1,2}(V \otimes \mathcal{H}) \subset \dom(\delta).
\end{equ}

Next, notice that for $m \in \mathbb{N}$, $D^m$ maps $\mathcal{S}(V)$ into $\mathcal{S}(V \otimes \mathcal{H}^{\otimes m})$. For $m \in \mathbb{N}$, $p\geq 2$, we denote by  $\mathbb{D}^{m,p}(V)$ the completion of $\mathcal{S}(V)$ with respect to the norm
\begin{equ}
    \Vert F\Vert_{\mathbb{D}^{m,p}(V)}:=\big(\E\|F\|_V^p+\sum_{k=1}^m\E(\Vert D^kF\Vert_{V \otimes \mathcal{H}^{\otimes k}}^p)\big)^{1/p}.
\end{equ}
The space $\mathbb{D}^{m,p}(V)$ is reflexive (since it can be identified with a closed subspace of $\prod_{k=0}^nL_p(\Omega; V\otimes \cH^{\otimes k})$. 
The following lemma, whose proof is in the Appendix, will be used often. 
\begin{lemma}   \label{lem:orthogonal_density}
    Let $(e_i)_{i=1}^\infty, (q_k)_{k=1}^\infty$ be  orthonormal basis of $\mathcal{H}, V$, respectively.  The collection  
    \begin{equs}
        \mathcal{S}_{ON}(V \otimes \mathcal{H})= \left\{ \sum_{i,k= 1}^n f_{ik}(B(e_1), \dots, B(e_n)) q_k \otimes e_i: n \in \mathbb{N}, f_{ik} \in \cC_{\pol}^\infty(\R^n) \right\} 
    \end{equs}
    is dense in $\mathbb{D}^{k, p}(V \otimes \mathcal{H})$ for any $k \geq 0$, $p \geq 2$. 
\end{lemma}

Clearly,  $\mathbb{D}^{m',p'}(V) \subset \mathbb{D}^{m,p}(V)$ if $m' \geq m$ and $p' \geq p$. It is immediate that  for all $F \in \mathbb{D}^{m,p}(V)$ we have 
\begin{equs}
    \| D F\|_{\mathbb{D}^{m-1,p}(V\otimes \mathcal{H})} \leq \| F\|_{\mathbb{D}^{m,p}(V)}. 
\end{equs}
In addition, it is known that there exists $C=C(p,m)$ such that for all $u \in \mathbb{D}^{m,p}(V\otimes \cH)$
we have 
\begin{equ} 
\label{est:delta}
    \| \delta(u)\|_{\mathbb{D}^{m-1,p}(V)}\leq C\|u\|_{\mathbb{D}^{m,p}(V \otimes \mathcal{H})}.
\end{equ}
If $F\in\mathbb{D}^{1,p}(\R^m)$  and $g\in C^{1}(\R^m)$, then $g(F)\in\mathbb{D}^{1,p}(\R)$ and the following \emph{chain rule} holds
\begin{equ}
    \label{eq:chain}
    Dg(F)=\sum_{k=1}^m\D_kg(F)DF^k.
\end{equ}
\section{Partial Malliavin calculus and conditional estimates}\label{sec:partial-malliavin} 

In this section, we continue to work on the probability space $(\Omega, \cF^B_1, \mathbb{P})$ (recall that the terminal time is $1$). The concept of partial Malliavin calculus was first introduced in \cite{stroock} and later further studied in \cite{Nualart_PMC}. Its development was aimed at studying the conditional densities of diffusion processes appearing in filtering problems.

It turns out that this concept is very useful for our purposes as well. Recall from the introduction that, roughly speaking, for our sewing arguments, we need to estimate expressions of the form $\E^B_s \nabla f (\phi^{s,x}_t)$ in terms of the function $f$ but in relatively weak norms, say $\|f\|_{L_\infty}$ for now. In the case $H=1/2$, one can use the Markov property to reduce the conditional expectation to a usual expectation, and then use tools from Malliavin calculus, such as the Malliavin integration by parts formula. Such an argument does not work in the non-Markovian setting. Nevertheless, it turns out that an integration by parts formula conditionally on $\cF^B_s$ is true, provided that one considers Malliavin derivatives with respect to $B_\cdot-\E^B_sB_\cdot$, which is independent of $\cF^B_s$, rather than Malliavin derivatives with respect to $B$ (see \eqref{eq:conditional_IBP} below). In this context, derivatives with respect to $B_\cdot-\E^B_sB_\cdot$  are referred to as \emph{partial Malliavin derivatives} since $B_\cdot-\E^B_sB_\cdot$ can be seen as a component of $B$ via the decomposition $B_\cdot=(B_\cdot-\E^B_sB_\cdot)+\E^B_s B_\cdot$.  In this section, we study the properties of these partial Malliavin derivatives and their adjoint operator.

\begin{remark}\label{rem:cate}
    In \cite{cate} a different way to estimate conditional densities is proposed.
    To our best judgment this approach has flaws that seem to be nontrivial to fix. For instance, one of the basic statements, Lemma 4.7 therein is not true and we provide a counterexample: $F=1$, $G=W_a$, $h=1_{[a,b]}$. Indeed, using notation therein (note that $W$ in their notation is $B$ in our notation), one then has $D_{[a,b]} F=0$,  $D G=1_{[0,a]}$ and so $ D_{[a,b]}G=1_{[a,b]}D G=0$, which means that the first two expectations in \cite[Lemma~4.7]{cate} are $0$, while the last is $\E(W_a(W_b-W_a))\neq 0$.
    Although this lemma is not directly used in the paper, in the unconditioned setting it is an important ingredient in showing the continuity of the $\delta$ operator in the Malliavin norm. Since the conditional version of the continuity of $\delta$ is only stated but not proved in \cite[Theorem~4.11]{cate}, it is not clear how one can overcome the failure of \cite[Lemma~4.7]{cate} in its proof. 
\end{remark}

For $s \in [0,1]$, we define the subspace \(\mathcal{H}_s\) of \(\mathcal{H}\) by  
\begin{equs}
    \mathcal{H}_s= \overline{\text{span}\{ \bone_{[0, q]} : \, q \leq s\}}^{\|\cdot\|_{\mathcal{H}}}.
\end{equs}
We denote by $\mathcal{H}_s^\perp$ the orthogonal complement of $\mathcal{H}_s$, so that one has the orthogonal decomposition $   \mathcal{H}=  \mathcal{H}_s \oplus  \mathcal{H}_s^\perp$.
We denote by $\Pi_{ \mathcal{H}_s^\perp}: \mathcal{H} \to  \mathcal{H}_s^\perp$ the orthogonal projection operator and  with a slight abuse of notation, we also use the same notation for the projection operator from $V \otimes\mathcal{H}$ onto $V \otimes \mathcal{H}_s^\perp$. We have the following continuity property. 
\begin{lemma}     \label{lem:continuity_projection_Sobolev}
    For $k \in\N, p \geq 2$, there exists a constant $C=C(k,p)$ such that for all $F \in \mathbb{D}^{k,p}(V)$ we have 
    \begin{equs}
        \| \Pi_{\mathcal{H}_s^\perp}F \|_{\mathbb{D}^{k,p}(V\otimes \mathcal{H})} \leq C \|F \|_{\mathbb{D}^{k,p}(V\otimes \mathcal{H})}.
    \end{equs}
\end{lemma}
\begin{proof}
    Let $(e_i)_{i=1}^\infty$ be an orthonormal basis of $\mathcal{H}$ consisting of eigenvectors of $\Pi_{\mathcal{H}_s^\perp}: \mathcal{H} \to \mathcal{H}_s^\perp$, so that $(\Pi_{\mathcal{H}_s^\perp}e_i)_{i=1}^\infty$ are orthogonal and let $(q_k)_{k=1}^\infty$ be an orthonormal basis of $V$. By \cref{lem:orthogonal_density}, it suffices to show the desired inequality only for $F$ of the form 
    \begin{equs}
        F= \sum_{i,k= 1}^n f_{ik}(B(e_1), \dots, B(e_n)) q_k \otimes e_i, 
    \end{equs}
    with $n \in \mathbb{N}$ and $f_{ik} \in \cC_{\pol}^\infty(\R^n)$. By using orthogonality and that the norm of the projection operator is bounded by $1$, we see that for $l \in \{0, \dots, k\}$, we have 
    \begin{equs}
    \,  &  \quad  \| D^l \Pi_{\mathcal{H}_s^\perp}F\|_{V \otimes \mathcal{H}^{\otimes (l+1)}}^2
      \\
      & = \Big\|\sum_{i,k=1}^n \sum_{j_1,...,j_l} \D_{j_1,...,j_l}f_{ik}(B(e_1), \dots, B(e_n)) q_k \otimes \Pi_{\mathcal{H}_s^\perp}e_i \otimes e_{j_1} \dots \otimes e_{j_1}\Big\|_{V \otimes \mathcal{H}^{\otimes (l+1)}}^2
      \\
      & = \sum_{i,k=1}^n \sum_{j_1,...,j_l} \big \|\D_{j_1,...,j_l}f_{ik}(B(e_1), \dots, B(e_n))  q_k \otimes \Pi_{\mathcal{H}_s^\perp}e_i \otimes e_{j_1} \dots \otimes e_{j_1} \big\|_{V \otimes \mathcal{H}^{\otimes (l+1)}}^2
      \\
      & \leq  \sum_{i,k=1}^n \sum_{j_1,...,j_l} \big | \D_{j_1,...,j_l}f_{ik}(B(e_1), \dots, B(e_n))| ^2= \| D^l F\|_{V \otimes \mathcal{H}^{\otimes (l+1)}}^2.
    \end{equs}
    From this, the claim follows. 
\end{proof}
Next, we define the operator $D_{\mathcal{H}_s^\perp}: \mathbb{D}^{1,2}(V) \to L_2(\Omega; V \otimes \mathcal{H})$ by the formula
\begin{equs}
    D_{\mathcal{H}_s^\perp}F= \Pi_{\mathcal{H}_s^\perp} DF. 
\end{equs}
We denote by $\delta_{\mathcal{H}_s^\perp}: \mathrm{dom}(\delta_{\mathcal{H}_s^\perp}) \to L_2(\Omega; V)$ the adjoint of  $D_{\mathcal{H}_s^\perp}$, where 
\begin{equs}
    \dom(\delta_{\mathcal{H}_s^\perp}):= \{ u \in  V \otimes \mathcal{H}: \Pi_{\mathcal{H}_s^\perp} u \in \dom(\delta)\}. 
\end{equs}
For $u \in  \dom(\delta_{\mathcal{H}_s^\perp})$ we have  $\delta_{\mathcal{H}_s^\perp}(u)= \delta(\Pi_{\mathcal{H}_s^\perp} u)$. Notice that by  \cref{lem:continuity_projection_Sobolev} and \eqref{eq:domain-inclusion},
for any $m \in\N$, $p \geq 2$, we have that  $\mathbb{D}^{m,p}(V \otimes \mathcal{H})$ is contained in $\dom(\delta_{\mathcal{H}_s^\perp})$. The following \emph{conditional integration by parts formula}  holds.

\begin{proposition}
    For $F \in \mathbb{D}^{1,2}(V)$ and $u \in \mathbb{D}^{1,2}(V \otimes \mathcal{H})$ we have 
\begin{equs}           \label{eq:conditional_IBP}
    \E_s^B \langle F, \delta_{\mathcal{H}_s^\perp}(u)\rangle_V= \E_s^B \langle  D_{\mathcal{H}_s^\perp}F, u \rangle_{V \otimes \mathcal{H}}, \qquad a.s. 
\end{equs}
\end{proposition}
\begin{proof}
    For $G$ of the form $g(B_{r_1}, ..., B_{r_n})$ with a smooth $g$ and $0 \leq r_1 \leq ... \leq r_n \leq s$, we have by \eqref{eq:ini-IBP}
\begin{equs}
    \E G \langle F, \delta_{\mathcal{H}_s^\perp}(u)\rangle_V & =  \E \langle G F, \delta_{\mathcal{H}_s^\perp}(u)\rangle_V
    \\
   &  =\E \langle D(G F), \Pi_{\mathcal{H}_s^\perp}u \rangle_{V\otimes \mathcal{H}}
   \\
   & = \E \langle F \otimes DG+ G DF, \Pi_{\mathcal{H}_s^\perp}u \rangle_{V\otimes \mathcal{H}}
   \\
   & = \E G \langle D_{\mathcal{H}_s^\perp}F, u \rangle_{V\otimes \mathcal{H}},
\end{equs}
where for the last equality we have used the fact that $DG \in \mathcal{H}_s$, hence, $ \Pi_{\mathcal{H}_s^\perp}(F \otimes DG)=0$. 
\end{proof}
\begin{remark}
    For $u \in \mathscr{S}_{ON}(V \otimes \mathcal{H})$, by virtue of \eqref{eq:explicit_formula_delta}, it is straightforward to verify that 
    \begin{equs}      \label{eq:Pi_delta}
        \Pi_{\mathcal{H}_s^\perp}\delta ( \widehat{D\Pi_{\mathcal{H}_s^\perp}u}) = \delta ( \widehat{D_{\mathcal{H}_s^\perp}\Pi_{\mathcal{H}_s^\perp}u})= \delta_{\mathcal{H}_s^\perp}( \widehat{D_{\mathcal{H}_s^\perp}\Pi_{\mathcal{H}_s^\perp}u}).
    \end{equs}
    Indeed, let us verify the first equality. If $u$ is of the form 
    \begin{equs}
        u = \sum_{i,k= 1}^n f_{ik}(B(e_1), \dots, B(e_n)) q_k \otimes e_i, 
    \end{equs}
    then we have that 
    \begin{equs}
       \widehat{D\Pi_{\mathcal{H}_s^\perp}u}= \sum_{i,j,k=1}^n \D_jf_{ik}(e_1,\dots, e_n) q_k \otimes e_j \otimes \Pi_{\mathcal{H}_s^\perp}e_i.
    \end{equs}
    Then, according to \eqref{eq:explicit_formula_delta}, we have 
    \begin{equs}      \label{eq:deltaDPi}
          \delta (\widehat{D\Pi_{\mathcal{H}_s^\perp}u})&= \sum_{i,k,j=1}^n \D_jf_{ik}(e_1,\dots, e_n) B(\Pi_{\mathcal{H}_s^\perp}e_i) q_k \otimes e_j
          \\
          & \qquad - \sum_{i,k,j, l=1}^n \D^2_{lj}f_{ik}(e_1,\dots, e_n) \langle e_l, \Pi_{\mathcal{H}_s^\perp}e_i \rangle_{\mathcal{H}} q_k \otimes e_j.  
    \end{equs}
    On the other hand, we have 
    \begin{equ}
\widehat{D_{\mathcal{H}_s^\perp}\Pi_{\mathcal{H}_s^\perp}u}= \sum_{i,j,k=1}^n \D_jf_{ik}(e_1,\dots, e_n) q_k \otimes \Pi_{\mathcal{H}_s^\perp}e_j \otimes \Pi_{\mathcal{H}_s^\perp}e_i,
    \end{equ}
    which implies that  
    \begin{equs}       \label{eq:deltaD_HPi}
        \delta (\widehat{D_{\mathcal{H}_s^\perp}\Pi_{\mathcal{H}_s^\perp}u})&= \sum_{i,k,j=1}^n \D_jf_{ik}(e_1,\dots, e_n) B(\Pi_{\mathcal{H}_s^\perp}e_i) q_k \otimes {\Pi_{\mathcal{H}_s^\perp}}e_j
          \\
          & \qquad - \sum_{i,k,j, l=1}^n \D^2_{lj}f_{ik}(e_1,\dots, e_n) \langle e_l, \Pi_{\mathcal{H}_s^\perp}e_i \rangle_{\mathcal{H}} q_k \otimes \Pi_{\mathcal{H}_s^\perp}e_j.
    \end{equs}
    By comparing \eqref{eq:deltaDPi} and \eqref{eq:deltaD_HPi}, the first equality in \eqref{eq:Pi_delta} is obvious. The second equality holds since $\widehat{D_{\mathcal{H}_s^\perp}\Pi_{\mathcal{H}_s^\perp}u}$ takes values in $V \otimes {\mathcal{H}_s^\perp} \otimes {\mathcal{H}_s^\perp}$, and on the latter,   $\delta$ and $\delta_{\mathcal{H}_s^\perp}$ coincide. 

    Finally, by \eqref{eq:Pi_delta} and \eqref{eq:Heisenberg}, it follows that 
\begin{equs}                 \label{eq:H_Heisenberg}
    D_{\mathcal{H}_s^\perp} \delta_{\mathcal{H}_s^\perp}(u) = \Pi_{\mathcal{H}_s^\perp}u+ \delta_{\mathcal{H}_s^\perp} (\widehat{D_{\mathcal{H}_s^\perp} \Pi_{\mathcal{H}_s^\perp}u}).
\end{equs}
\end{remark}

\begin{lemma}          \label{lem:conditional_boundedness_L2}
    For $u \in \mathbb{D}^{1,2}(V \otimes \mathcal{H})$ we have with probability one
    \begin{equs}
        \E^B_s \|\delta_{\mathcal{H}_s^\perp}(u)\|^2_V \leq \E^B_s \left( \|\Pi_{\mathcal{H}_s^\perp} u\|^2_{V \otimes \mathcal{H}}+ \|\Pi_{\mathcal{H}_s^\perp \otimes \mathcal{H}_s^\perp} Du\|^2_{V \otimes \mathcal{H}\otimes \mathcal{H}} \right). 
    \end{equs}
\end{lemma}
\begin{proof}
     Let $(e_i)_{i=1}^\infty$ be an orthonormal basis of $\mathcal{H}$ consisting of eigenvectors of $\Pi_{\mathcal{H}_s^\perp}$, let $(q_k)_{k=1}^\infty$ be an orthonormal basis of $V$ and let  $u$ be of the form 
    \begin{equs}   \label{eq:simple_u}
        u= \sum_{i,k= 1}^n F_{ik} q_k \otimes  e_i, \qquad F_{ik} \in \mathcal{S}(\mR). 
    \end{equs}
    By using \eqref{eq:conditional_IBP} (with $F=\delta_{\mathcal{H}_s^\perp}(u)$), and \eqref{eq:Heisenberg}, we get 
    \begin{equs}
         \E^B_s \|\delta_{\mathcal{H}_s^\perp}(u)\|^2_V& = \E^B_s \big\langle \Pi_{\mathcal{H}_s^\perp} u,  D \delta( \Pi_{\mathcal{H}_s^\perp} u) \big\rangle_{V \otimes \mathcal{H}}
         \\
         &=  \E^B_s \left( \|\Pi_{\mathcal{H}_s^\perp} u\|_{V \otimes \mathcal{H}}^2 +  \big\langle \Pi_{\mathcal{H}_s^\perp} u, \delta( \widehat{ D\Pi_{\mathcal{H}_s^\perp} u} ) \big\rangle_{V \otimes \mathcal{H}}\right). 
    \end{equs}
    Notice that $\widehat{ D\Pi_{\mathcal{H}_s^\perp} u} \in V \otimes \mathcal{H} \otimes \mathcal{H}_s^{\perp}$, so that we can use \eqref{eq:conditional_IBP} to get 
     \begin{equs}
         \E^B_s \|\delta_{\mathcal{H}_s^\perp}(u)\|^2_V& =  \E^B_s \left( \|\Pi_{\mathcal{H}_s^\perp} u\|_{V \otimes \mathcal{H}}^2 +  \big\langle D_{\mathcal{H}_s^\perp}\Pi_{\mathcal{H}_s^\perp} u, \widehat{ D\Pi_{\mathcal{H}_s^\perp} u}  \big\rangle_{V \otimes \mathcal{H} \otimes \mathcal{H} }\right) 
         \\
         &= \E^B_s \left( \|\Pi_{\mathcal{H}_s^\perp} u\|_{V \otimes \mathcal{H}}^2 +  \big\langle \Pi_{\mathcal{H}_s^\perp \otimes \mathcal{H}_s^\perp} Du, \widehat{ D\Pi_{\mathcal{H}_s^\perp} u}  \big\rangle_{V \otimes \mathcal{H} \otimes \mathcal{H} }\right)
         \\
         & \leq \E^B_s \left( \|\Pi_{\mathcal{H}_s^\perp} u\|^2_{V \otimes \mathcal{H}}+ \|\Pi_{\mathcal{H}_s^\perp \otimes \mathcal{H}_s^\perp} Du\|^2_{V \otimes \mathcal{H}\otimes \mathcal{H}} \right). 
    \end{equs}
    This shows the desired inequality for $u$ of the from \eqref{eq:simple_u}. For general $u \in \mathbb{D}^{1,2}(V \otimes \mathcal{H})$ the result follows by a standard approximation argument upon using  \cref{lem:orthogonal_density}, \cref{lem:continuity_projection_Sobolev} and the continuity of conditional expectation on $L_1(\Omega)$. 
\end{proof}
The proof of the next lemma is essentially a repetition of the proof of \cite[Proposition 3.9]{Malliavin_Hairer}.
\begin{lemma}       \label{lem:conditional_continuity_delta_1}
    For every $p \in [1, \infty)$,  there exists $k=k(p) \in \mathbb{N}$ and a constant $C=C(p)$ such that for all separable Hilbert spaces $V$ and all  $u \in \mathbb{D}^{k,p}(V \otimes \mathcal{H})$ we have with probability one
    \begin{equs}
        \E^B_s \| \delta_{\mathcal{H}_s^\perp}(u)\|^p_V \leq C \sum_{l=0}^{k}  \left(\E^B_s \| D^l_{\mathcal{H}_s^\perp} \Pi_{\mathcal{H}_s^\perp}u \|_{V \otimes \mathcal{H}^{ \otimes (l+1)}}^{2p} \right)^{1/2}.
    \end{equs}
\end{lemma}

\begin{proof}
     The statement is clearly true for $p \in [1,2]$ by virtue of by \cref{lem:conditional_boundedness_L2}. Suppose now that it is true for all $p \in [1, 2 (3/2)^n]$ for some $n \in \{0,1,2,\dots\}$ and let us show that it is true also for $n+1$. Clearly, it suffices to show the result for $u \in \mathcal{S}_{ON}(V\otimes \mathcal{H})$. So let $p \in [2 (3/2)^n, 2 (3/2)^{n+1}]$.  By the chain rule, \eqref{eq:conditional_IBP}, and \eqref{eq:Heisenberg}, we have 
    \begin{equs}
        \E^B_s \| \delta_{\mathcal{H}_s^\perp}(u)\|_V^p & = (p-1) \E^B_s \big( \| \delta_{\mathcal{H}_s^\perp}(u)\|_V^{p-2}\langle D \delta_{\mathcal{H}_s^\perp}(u), \Pi_{\mathcal{H}_s^\perp}u \rangle_{V \otimes \mathcal{H}} \big) 
        \\
        & = (p-1) \E^B_s \Big(  \| \delta_{\mathcal{H}_s^\perp}(u)\|_V^{p-2} ( \|\Pi_{\mathcal{H}_s^\perp}u\|^2_{V \otimes \mathcal{H}}+ \langle \delta ( \widehat{D\Pi_{\mathcal{H}_s^\perp}u}), \Pi_{\mathcal{H}_s^\perp}u \rangle_{V \otimes \mathcal{H}} \Big). 
    \end{equs}
    By \eqref{eq:explicit_formula_delta}, one can easily see that 
    \begin{equs}
        \Pi_{\mathcal{H}_s^\perp}\delta ( \widehat{D\Pi_{\mathcal{H}_s^\perp}u}) = \delta ( \widehat{D_{\mathcal{H}_s^\perp}\Pi_{\mathcal{H}_s^\perp}u}).
    \end{equs}
    Using the above together with the Young and the Cauchy-Schwartz inequalities and rearranging, gives
    \begin{equs}
        \E^B_s \| \delta_{\mathcal{H}_s^\perp}(u)\|_V^p &  \lesssim  \E^B_s \left( \|\Pi_{\mathcal{H}_s^\perp}u\|^p_{V \otimes \mathcal{H}}+ 
        \|\delta ( \widehat{D_{\mathcal{H}_s^\perp}\Pi_{\mathcal{H}_s^\perp}u}) \|_{V \otimes \mathcal{H}} ^{p/2}\|\Pi_{\mathcal{H}_s^\perp}u \|_{V \otimes \mathcal{H}}^{p/2}
         \right).    \label{eq:delta_before_Y_H}
    \end{equs}
    By H\"older's and Young's inequalities we have 
    \begin{equs}
        \E^B_s \big( \|\delta ( \widehat{D_{\mathcal{H}_s^\perp}\Pi_{\mathcal{H}_s^\perp}u}) &\|_{V \otimes \mathcal{H}} ^{p/2}\|\Pi_{\mathcal{H}_s^\perp}u \|_{V \otimes \mathcal{H}}^{p/2}\big) \\& \leq  \big(\E^B_s \|\delta ( \widehat{D_{\mathcal{H}_s^\perp}\Pi_{\mathcal{H}_s^\perp}u}) \|_{V \otimes \mathcal{H}} ^{2p/3}\big)^{3/4} \big(\E^B_s\|\Pi_{\mathcal{H}_s^\perp}u \|_{V \otimes \mathcal{H}}^{2p} \big)^{1/4}
        \\
        & \lesssim  \big(\E^B_s \|\delta ( \widehat{D_{\mathcal{H}_s^\perp}\Pi_{\mathcal{H}_s^\perp}u}) \|_{V \otimes \mathcal{H}} ^{2p/3}\big)^{3/2}+  \big(\E^B_s\|\Pi_{\mathcal{H}_s^\perp}u \|_{V \otimes \mathcal{H}}^{2p} \big)^{1/2}.
    \end{equs}
     Combining this with \eqref{eq:delta_before_Y_H}, we get 
    \begin{equs}     \label{eq:delt_comb}
         \E^B_s \| \delta_{\mathcal{H}_s^\perp}(u)\|_V^p\lesssim  \big(\E^B_s \|\delta ( \widehat{D_{\mathcal{H}_s^\perp}\Pi_{\mathcal{H}_s^\perp}u}) \|_{V \otimes \mathcal{H}} ^{2p/3}\big)^{3/2}+   \big(\E^B_s\|\Pi_{\mathcal{H}_s^\perp}u \|_{V \otimes \mathcal{H}}^{2p} \big)^{1/2}.
    \end{equs}
    Next, notice that $ \widehat{D_{\mathcal{H}_s^\perp}\Pi_{\mathcal{H}_s^\perp}u} \in V \otimes \mathcal{H}_s^{\perp} \otimes \mathcal{H}_s^{\perp}$, which means that $\delta ( \widehat{D_{\mathcal{H}_s^\perp}\Pi_{\mathcal{H}_s^\perp}u}) =\delta_{\mathcal{H}_s^\perp} ( \widehat{D_{\mathcal{H}_s^\perp}\Pi_{\mathcal{H}_s^\perp}u})$. Therefore,  by the induction hypothesis, there exists $k' \in \mathbb{N}$  such that 
    \begin{equs}
        \E^B_s \|\delta ( \widehat{D_{\mathcal{H}_s^\perp}\Pi_{\mathcal{H}_s^\perp}u}) \|_{V \otimes \mathcal{H}} ^{2p/3}  & \lesssim  \sum_{l=0}^{k'}  \left(\E^B_s \| D^l_{\mathcal{H}_s^\perp} \widehat{D_{\mathcal{H}_s^\perp}\Pi_{\mathcal{H}_s^\perp}u}\|_{V \otimes \mathcal{H}^{ \otimes (l+2)}}^{4p/3} \right)^{1/2}
        \\
        &=  \sum_{l=0}^{k'+1}  \left(\E^B_s \| D^l_{\mathcal{H}_s^\perp} \Pi_{\mathcal{H}_s^\perp}u\|_{V \otimes \mathcal{H}^{ \otimes (l+1)}}^{4p/3} \right)^{1/2}.
    \end{equs}
    Combining this with \eqref{eq:delt_comb}, we get 
    \begin{equs}
         \E^B_s \| \delta_{\mathcal{H}_s^\perp}(u)\|_V^p& \lesssim  \sum_{l=0}^{k'+1}  \left(\E^B_s \| D^l_{\mathcal{H}_s^\perp} \Pi_{\mathcal{H}_s^\perp}u\|_{V \otimes \mathcal{H}^{ \otimes (l+1)}}^{4p/3} \right)^{3/4}+ \big(\E^B_s\|\Pi_{\mathcal{H}_s^\perp}u \|_{V \otimes \mathcal{H}}^{2p} \big)^{1/2}
         \\
         & \lesssim \sum_{l=0}^{k'+1}  \left(\E^B_s \| D^l_{\mathcal{H}_s^\perp} \Pi_{\mathcal{H}_s^\perp}u\|_{V \otimes \mathcal{H}^{ \otimes (l+1)}}^{2p} \right)^{1/2},
    \end{equs}
which finishes the proof. 
\end{proof}
Next, for $u \in \mathbb{D}^{m,p}(V)$, let us define 
\begin{equs}
    \| u\|_{\mathbb{D}^{m,p}_s(V)} = \left( \sum_{l=0}^m \E^B_s \| D^l_{\mathcal{H}_s^\perp} u \|_{V \otimes \mathcal{H}^{ \otimes l}}^p \right)^{1/p}.
\end{equs}
We have the following continuity property for the operator $\delta_{\cH_s^\perp}$. 
\begin{lemma}\label{lem:continuity-delta-high}
    For any $p \in [1, \infty)$ and $m \in \{0, 1, 2, \dots\}$, there exists $n=n(p,m) \in \mathbb{N}$, $n\geq m$, and a constant $C=(p,m)$, such that for any separable Hilbert space $V$, any $u \in \mathbb{D}^{m,p}(V \otimes \mathcal{H})$, and any $s \in [0,1]$,  we have with probability one
    \begin{equ}
         \| \delta_{\mathcal{H}_s^\perp}(u) \|_{\mathbb{D}^{m,p}_s(V)} \leq C \| \Pi_{\mathcal{H}_s^\perp}u\|_{\mathbb{D}^{n,2p}_s(V\otimes \mathcal{H})}.
    \end{equ}
\end{lemma}
\begin{proof}
Clearly it suffices to show the inequality for $u \in \mathcal{S}_{ON}(V \otimes \mathcal{H}). $
    The statement is true for $m=0$, by virtue of  \cref{lem:conditional_continuity_delta_1}. Suppose now that it holds for some $m \in \{0, 1, 2,  \dots\}$ and let us show that it also holds for $m+1$. 
    By using \eqref{eq:H_Heisenberg}, we see that  
    \begin{equs}
        \E_s \| D^{m+1}_{\mathcal{H}_s^\perp}\delta_{\mathcal{H}_s^\perp}(u) \|_{V \otimes \mathcal{H}^{ \otimes (m+1)}}^p \lesssim \E_s \|D^m_{\mathcal{H}_s^\perp} \Pi_ {\mathcal{H}_s^\perp}u \|_{V \otimes \mathcal{H}^{ \otimes (m+1)}}^p + \E_s \|D^m_{\mathcal{H}_s^\perp} \delta_{{\mathcal{H}_s^\perp}}(\widehat{D_{\mathcal{H}_s^\perp}\Pi_{\mathcal{H}_s^\perp}u} ) \|_{V \otimes \mathcal{H}^{ \otimes (m+1)}}^p,
    \end{equs}
    and we only need to estimate the second term. By the induction hypothesis, we have for some $n' \in \mathbb{N}$
    \begin{equs}
      \E_s \|D^m_{\mathcal{H}_s^\perp} \delta_{{\mathcal{H}_s^\perp}}(\widehat{D_{\mathcal{H}_s^\perp}\Pi_{\mathcal{H}_s^\perp}u} ) \|_{V \otimes \mathcal{H}^{ \otimes (m+1)}}^p & \lesssim \| \Pi_{\mathcal{H}_s^\perp} \widehat{D_{\mathcal{H}_s^\perp}\Pi_{\mathcal{H}_s^\perp}u} \|_{\mathbb{D}^{n',2p}_s(V\otimes \mathcal{H} \otimes \mathcal{H})}
      \\
      & = \|  \widehat{D_{\mathcal{H}_s^\perp}\Pi_{\mathcal{H}_s^\perp}u} \|_{\mathbb{D}^{n',2p}_s(V\otimes \mathcal{H} \otimes \mathcal{H})}
      \\
      & \lesssim  \|  \Pi_{\mathcal{H}_s^\perp}u \|_{\mathbb{D}^{n'+1,2p}_s(V\otimes \mathcal{H})},
    \end{equs}
    which finishes the proof. 
\end{proof}
Our next task is the following. We want to obtain a lower bound for $\|\Pi_{\cH_s^\perp} f\|_{\cH}$ for $f \in \mathcal{H}$. To achieve this, we first show  
that $\|\Pi_{\cH_s^\perp} f\|_{\cH}$ can be written as a 2D Young integral against the covariance function of the process $t \mapsto B_t-\E^B_sB_t$. Then, given that the covariance function of the latter satisfies certain properties, the desired bound will follow from \cite[Corollary~6.10]{CHLT15}. Let us now be more precise. 

Recall that the inner product on $\cH$ is given by means of the covariance function $Q(\cdot, \cdot)$ of the process $B^1$. Let us set 
\begin{equs}    \label{eq:def_tilde_bar}
    \widetilde{B}^s_t:= B_t-\E^B_s B_t= \int_s^t K(t,r) \, dW_r, \qquad \overline{B}^s_t := \E_s^B B_t = \int_0^s K(t, r) \, dW_r 
\end{equs}
and let us denote by $\widetilde{B}^{s,1}$ and $\overline{B}^{s,1}$ the first of the $d_0$ components of $\widetilde{B}^s$ and $\overline{B}^s$, respectively. Further, let us denote by $Q_s(\cdot, \cdot)$ the  covariance function of $\widetilde{B}^{s,1}$. To study the covariance $Q_s$, let us recall some basic properties of the kernel $K$. The next lemma is a direct consequence of the definition of $K$ (see, \eqref{eq:mandelbrot}). 
\begin{lemma}                                           \label{lem:some_properties_of_K}
    The following hold:
    \begin{enumerate}[(i)]
        \item For all $(s,t) \in [0, 1]_{<}^2$, we have $K(t,s) \geq 0$.                                  \label{item:positivity_of_K}          
        \item For $s \in [0, 1)$,  the map $(s, 1] \ni t \to K(t,s)$ is continuously differentiable and      \label{item:derivative_K}
\begin{equs}     \label{eq:derivative_K}
    \D_t K(t, s) =c_H(H-\frac{1}{2}) \left( \frac{t}{s}\right)^{H-1/2}(t-s)^{H-3/2}.
\end{equs}
In particular,  $(s, 1] \ni t \to K(t,s)$ is decreasing.
\end{enumerate}
\end{lemma}
Let us continue with some properties of $K$. It is well known (see \cite[pp.281-284]{NUalart_Malliavin}) 
\begin{equs}
    \sup_{t \in [0,1]} \| K(t, \cdot)\|_{L_2([0,1])}< \infty,
\end{equs}
 and that for all  $s, t \in [0,1]$ we have
\begin{equs}       \label{eq:representation_covariance_fBM}
    Q(s,t) = \int_0^{s\wedge t} K(t,r) K(s,r) \, dr. 
\end{equs}
Further, the operator $K^* : \mathcal{E} \to L_2([0,1])$ given by 
\begin{equs}
    (K^*h)(s)= K(1, s)h(s)+ \int_s^1( h(t)-h(s) )\D_t K(t,s) \, dt 
\end{equs}
satisfies 
\begin{equ}     \label{eq:K_on_characteristic}
    (K^*\bone_{[0,t]})(r)= K(t, r)\bone_{[0, t]}(r), \qquad r \in [0,1],
\end{equ}
for all $t \in [0,1]$. Consequently, $K^* : \mathcal{E} \to L_2([0,1])$
is an linear isometry, and therefore it extends to an isometry  between $\mathcal{H}$ and a closed subspace of $L_2([0,1])$. 
It turns out that the map $K^* : \cH \to L_2([0,1])$ is in fact surjective. Indeed, if $f$ belongs to the orthogonal complement of $K^*(\cH)$, then for all $t \in [0,1]$ we have
\begin{equs}
    0= \langle K^*\bone_{[0,t]}, f \rangle_{L_2([0,1])}= \int_0^t K(t,r)f(r) \, dr. 
\end{equs}
The right hand side of the above relation admits the following representation in terms of Liouville integrals (see \cite[Theorem 2.1]{ustunel} and references therein):
\begin{equs}
   0=  \int_0^t K(t, r) f(r) \, dr = I^{2H}_{0^+}\Big(  (\cdot)^{1/2-H} I^{1/2-H}_{0^+}\big( (\Box)^{H-1/2}f(\Box) \big) (\cdot) \Big) (t).
\end{equs}
Since for any $\alpha>0$, the Liouville integrals $I^\alpha_{0^+}: L_1([0,1]) \to L_1([0,1])$ are injective, we conclude that $f \equiv 0$, which shows that $K^*$ is 
surjective.  

\begin{remark}     \label{rem:same_filtration}
From this discussion, it follows that   
 $(B((K^*)^{-1}\bone_{[0,t]}))_{t \in [0,1]}$ is a standard Brownian motion, it generates the same filtration as $(B_t)_{t\in[0,1]}$, and it coincides with the underlying Wiener process  $(W_t)_{t \in [0,1]}$. 
\end{remark}

\begin{lemma}
  The process $(\widetilde{B}^{s,1}_t)_{t \in [s, 1]}$ has negatively correlated increments. Moreover, for any $S \in [s, 1]$, and any partition $D= \{t_i\}_{i=1}^n$ of $[s, S]$, the $n\times n$ matrix $(Q^D_{ij})_{i,j=1}^n$ with entries 
  \begin{equs}
      Q^D_{ij}= \E (\widetilde{B}^{s,1}_{t_{i-1}, t_i}\widetilde{B}^{s,1}_{t_{j-1}, t_j})
  \end{equs}
  is diagonally dominant. 
\end{lemma}
\begin{proof}
We start by showing that $(\widetilde{B}^{s,1}_t)_{t \in [s, 1]}$ has negatively correlated increments. A sufficient condition is that $\D^2_{t t'}Q_s(t, t') \leq 0 $ for $t < t'$ (see \cite[p. 205]{CHLT15}). To see that this is the case, notice that 
\begin{equs}
    Q_s(t,t') &= 
\E(\widetilde{B}^{s,1}_t \widetilde{B}^{s,1}_{t'}) 
\\
&= \E \left(\int_s^t K(t,r) dW^1_r\int_s^{t'}K(t',r) \, dW^1_r \right)
\\
&= Q(t, t') - \int_0^s K(t,r) K(t',r) \, dr. \label{eq:rep_Q_s}
\end{equs}
It is straightforward that $\D^2_{t t'}Q(t, t') \leq 0 $ for $t < t'$. Moreover, from \eqref{eq:derivative_K}, it follows that 
\begin{equs}
    \int_0^s \D_tK(t,r) \D_{t'}K(t',r) \, dr > 0. 
\end{equs}
Consequently, we have that $\D^2_{t t'}Q_s(t, t') \leq 0 $ for $t < t'$. 

Next, since $(\widetilde{B}^{s,1}_t)_{t \in [s, 1]}$ is a  centered Gaussian process with negatively correlated increments,  starting at zero, a sufficient condition for the matrix $(Q^D_{ij})_{i,j=1}^n$ to be diagonally dominant is that $\D_t Q_s(t, t') \geq 0$ for $t < t'$ (see \cite[p. 205]{CHLT15}). The latter follows again by \eqref{eq:rep_Q_s} combined with the fact that  $\D_t Q(t, t') \geq 0$ for $t < t'$ and that  $\D_tK(t,r)<0$, $K(t',r)>0$ for $r \in [0, s]$. This finishes the proof.
\end{proof}

\begin{lemma}
    There exists a constant $C=C(H)$  such that for all $(s,t, u, v, T)\in [0,1]^5_{\leq}$ we have 
    \begin{equs}\label{eq:FV-condition}
        \E \big( \widetilde{B}^{s,1}_{t, T}\widetilde{B}^{s,1}_{u, v} \big) \leq C |v-u|^{2H}. 
    \end{equs}
\end{lemma}
\begin{proof}
By orthogonality, we have 
\begin{equs}
     \E \big( \widetilde{B}^{s,1}_{t, T}\widetilde{B}^{s,1}_{u, v} \big)=  \E \big(B^1_{t, T}B^1_{u, v} \big)-  \E \big( \widebar{B}^{s,1}_{t, T} \widebar{B}^{s,1}_{u, v} \big).
\end{equs}
It follows from the covariance of $B$ (see also \cite[p. 407]{friz2006}) that $|\E \big( B^1_{t, T}B^1_{u, v} \big)| \lesssim |v-u|^{2H}$ so we only focus on bounding the second term on the right hand side of the above equality. 
    We have 
    \begin{equs}
        | \E \big( \widebar{B}^{s,1}_{t, T} \widebar{B}^{s,1}_{u, v} \big)| & = \Big|\int_0^s \big(K(T, r)-K(t,r)\big)\big( K(v,r)-K(u,r) \big) \, dr \Big| 
        \\
        & \leq 2 \int_0^s K(s, r) \big( K(u,r)-K(v,r)\big) \, dr,
    \end{equs}
    where we have used Lemma \ref{lem:some_properties_of_K}, that is,  the fact that $K$ is non-negative and decreasing in the first variable. Consequently, by the above and \eqref{eq:representation_covariance_fBM} we get 
   \begin{equs}
        |\E \big( \widebar{B}^{s,1}_{t, T} \widebar{B}^{s,1}_{u, v} \big)| & \leq 2( Q(s,u)-Q(s,v))
        \\
        &= |u|^{2H}-|u-s|^{2H}- |v|^{2H}+|v-s|^{2H} \leq 2 |u-v|^{2H},
    \end{equs}
    which finishes the proof.

\end{proof}

\begin{corollary}\label{cor:holder-controlled2}
    For any $s\in[0,1)$, the covariance function $Q_s$ is  uniformly of H\"older controlled $2$-dimensional $1/(2H)$-variation, that is, there exists a constant $C=C(H)$  such that for any $(s,u,t)\in[0,1]_\leq^3$ one has
    \begin{equ}\label{eq:holder-controlled2}
        \,[Q_s]_{\cV^{1/(2H)}([u,t]^2)}\leq C (t-u)^{1/(2H)}.
    \end{equ}
\end{corollary}
\begin{proof}
    Following the proof of \cite[Proposition~15.5]{FV10}, \eqref{eq:holder-controlled2} is a consequence of negatively correlated increments and \eqref{eq:FV-condition}.
\end{proof}

\begin{lemma}\label{lem:lower-var}
Let $s \in [0,1]$  and for $(u, v) \in [s, 1]^2_{\leq}$  set $\widetilde{\mathcal{G}}_{u,v}= \sigma( \widetilde{B}^{s,1}_{u, r} , r \in [u, v])$. Then, there exists $\vartheta= \vartheta(H) >0$ such that 
  \begin{equs}
      \inf_{ \substack{ s \in [0,1] \\ 
    s \leq  u < v \leq 1} }\frac{1}{(v-u)^\vartheta} \Var( \widetilde{B}^{s,1}_{u,v} | \widetilde{\mathcal{G}}_{s,u} \vee \widetilde{\mathcal{G}}_{v,1}) >0.
  \end{equs}
\end{lemma}
\begin{proof}
    For $(t,t') \in [0, 1]_{<}^2$, let us set $\mathcal{G}_{t, t'}= \sigma ( B^1_{t, r}, r \in [t, t'] )$. One can easily see that $\mathcal{G}_{0,s}= \sigma(W^1_r, r \leq s)$, from which it easily follows that  for $(u,v) \in [0,1]^2_{\leq} $, we have 
    \begin{equs}
        \mathcal{G}_{0, u} \vee \mathcal{G}_{v, 1} =  \widetilde{\mathcal{G}}_{s,u} \vee \widetilde{\mathcal{G}}_{v,1} \vee \mathcal{G}_{0,s}.
    \end{equs}
    To ease the notation, let us set 
$\mathcal{G}= \mathcal{G}_{0, u} \vee \mathcal{G}_{v, 1}$, $\widetilde{\mathcal{G}}= \widetilde{\mathcal{G}}_{s,u} \vee \widetilde{\mathcal{G}}_{v,1}$. Hence, with this notation we have 
$\mathcal{G}= \widetilde{\mathcal{G}} \vee \mathcal{G}_{0,s}$.  Moreover, we have that  $\widetilde{\mathcal{G}}$ is independent of $\mathcal{G}_{0,s}$. 

Using the Gaussianity, we have 
\begin{equs}   \label{eq:variance_Z}
    \Var(\widetilde{B}^{s,1}_{u,v} | \widetilde{\mathcal{G}}_{s,u} \vee \widetilde{\mathcal{G}}_{v,1})& =  \Var(\widetilde{B}^{s,1}_{u,v} | \widetilde{\mathcal{G}})
    \\
    &= \E \Big( \widetilde{B}^{s,1}_{u,v}- \E[ \widetilde{B}^{s,1}_{u,v}|\widetilde{\mathcal{G}} ] \Big)^2. 
\end{equs}
Next, notice that 
\begin{equs}
    \E[ \widetilde{B}^{s,1}_{u,v}|\widetilde{\mathcal{G}} ]= \E[ B^1_{u,v}- \widebar{B}^{s,1}_{u,v}  |\widetilde{\mathcal{G}} ]= \E[ B^1_{u,v} |\widetilde{\mathcal{G}} ],
\end{equs}
where we have used that $ \widebar{B}^{s,1}_{u,v}$ is $\mathcal{G}_{0,s}$-measurable, hence, independent of $\widetilde{\mathcal{G}}$ and has zero mean. Hence, 
\begin{equs}          \label{eq:whatever1}
     \widetilde{B}^{s,1}_{u,v}- \E[ \widetilde{B}^{s,1}_{u,v}|\widetilde{\mathcal{G}}]= B_{u,v}- \widebar{B}^{s,1}_{u,v} -  \E[ B^1_{u,v} |\widetilde{\mathcal{G}}].
\end{equs}
Next, we claim that 
\begin{equs}
    \widebar{B}^{s,1}_{u,v} + \E[ B^1_{u,v} |\widetilde{\mathcal{G}}]= \E[ B^1_{u,v} |\widetilde{\mathcal{G}} \vee \mathcal{G}_{0,s}].          
\end{equs}
Indeed, we have 
\begin{equs}
    \E[ B^1_{u,v} |\widetilde{\mathcal{G}} \vee \mathcal{G}_{0,s}]&= \E[ B^1_{u,v} - \widebar{B}^{s,1}_{u,v}|\widetilde{\mathcal{G}} \vee \mathcal{G}_{0,s}]+ \widebar{B}^{s,1}_{u,v}
    \\
    &= \E[ B^1_{u,v} - \widebar{B}^{s,1}_{u,v}|\widetilde{\mathcal{G}} ]+ \widebar{B}^{s,1}_{u,v}
    \\
    &= \E[ B^1_{u,v} |\widetilde{\mathcal{G}} ]+ \widebar{B}^{s,1}_{u,v},
\end{equs}
where for the second equality we have used that $B^1_{u,v} - \widebar{B}^{s,1}_{u,v}$ is independent of $\mathcal{G}_{0,s}$, and for the last equality we have used that $\widebar{B}^{s,1}_{u,v}$ is independent of $\widetilde{\mathcal{G}}$ and has zero mean. By this combined with \eqref{eq:whatever1}, we get 
\begin{equs}
     \widetilde{B}^{s,1}_{u,v}- \E[ \widetilde{B}^{s,1}_{u,v}|\widetilde{\mathcal{G}}] &=  B^1_{u,v}- \E[ B^1_{u,v} |\widetilde{\mathcal{G}}\vee \mathcal{G}_{0,s} ]
     \\
      &= B^1_{u,v}- \E[ B^1_{u,v} | \mathcal{G}_{0, u} \vee \mathcal{G}_{v, 1}].
\end{equs}
Plugging this in \eqref{eq:variance_Z}, we conclude that 
\begin{equs}
     \Var( \widetilde{B}^{s,1}_{u,v} | \widetilde{\mathcal{G}}_{s,u} \vee \widetilde{\mathcal{G}}_{v,1})=  \Var(B^1_{u,v} | \mathcal{G}_{0,u} \vee \mathcal{G}_{v,1}).
\end{equs}
Finally, it is known that there exists constants $
\vartheta>0$
and $c >0 $ such that 
\begin{equs}
\inf_{0 \leq  u \leq v \leq 1} \frac{1}{(v-u)^\vartheta}\Var(B^1_{u,v} | \mathcal{G}_{0,u} \vee \mathcal{G}_{v,1}) = c >0 
\end{equs}
(see \cite[Lemmata 4.1 and 4.2]{CHLT15}), from which the conclusion follows.
\end{proof}
The following is from \cite[Proposition 5.2.1, p. 288]{NUalart_Malliavin}. In the statement, $D^W$ denotes the usual Malliavin derivative with respect to the underlying Wiener process $W$, and $\mathbb{D}^{1,2}_W$ denotes the corresponding Sobolev of all random variables  $X \in L_2(\Omega)$ for which $D^WX \in L_2(\Omega; L_2([0,1]))$. 
\begin{proposition}%\cite[Proposition 5.2.1]{NUalart_Malliavin}    
\label{prop:connection_two_Malliavin}
   It holds that $\mathbb{D}^{1,2}_W=\mathbb{D}^{1,2}$ and that for any $F \in \mathbb{D}^{1,2}$ we have 
\begin{equs}      
    K^*DF= D^WF. 
\end{equs}
\end{proposition}

Next, we see how the operator $\Pi_{\mathcal{H}_s^\perp}$ can be characterised by means of $K^*$. 
\begin{lemma}            \label{lem:representation_projection}
    For any $h \in \cH$ and $s \in [0,1]$ we have
    \begin{equs}
        K^* \Pi_{\mathcal{H}_s^\perp} h =  \bone_{[s,1]} K^*h . 
    \end{equs}
\end{lemma}

\begin{proof}
    Suppose that $h=\bone_{[0,t]}$ and notice that $h=DB_t$. Recall that we have the decomposition $B_t= \widetilde{B}^s_t+ \widebar{B}^s_t$ (see \eqref{eq:def_tilde_bar}).  By \cref{prop:connection_two_Malliavin}, it follows that both $\widetilde{B}^s_t$ and $\widebar{B}^s_t$ belong to  $\mathbb{D}^{1,2}$ and their Malliavin derivatives are deterministic elements of $\cH$. Moreover,  we have that $D\widebar{B}^s_t \in \mathcal{H}_s$, while $D\widetilde{B}^s_t \in \cH_s^{\perp}$. Indeed, the first claim follows from the fact that  $\widebar{B}^s_t$ is $\cF^B_s$-measurable, while the second follows from the fact that for any  $g=\bone_{[0,q]}$, with $q \leq s$, we have 
    \begin{equs}
        \langle D\widetilde{B}^s_t, g \rangle_{\cH}& =  \E \langle D\widetilde{B}^s_t, g \rangle_{\cH}
        \\
        & = \E (\widetilde{B}^s_t  B_q)  = 0,
    \end{equs}
    where we have used \eqref{eq:ini-IBP} and the fact that $\widetilde{B}^s_t$ is independent of $\cF^B_s$. Consequently, we have 
    \begin{equs}
        \Pi_{\mathcal{H}_s^\perp} h=  \Pi_{\mathcal{H}_s^\perp} D B_t= D\widetilde{B}^s_t. 
    \end{equs}
    By  \cref{prop:connection_two_Malliavin} again, we have 
    \begin{equs}
         K^*\Pi_{\mathcal{H}_s^\perp} h= K^* \widetilde{B}^s_t= D^W\widetilde{B}^s_t= K(t, \cdot)\bone_{[s,t]}(\cdot)= \bone_{[s,1]} K^*h,
    \end{equs}
    where for the last equality we have used \ref{eq:K_on_characteristic}. By linearity the claim is true for all $h \in \mathcal{E}$ and by a standard density and continuity argument it also holds for all $h \in \cH$. 
\end{proof}

\begin{lemma}\label{lem:H-norm-integral}
 Let $(s, t) \in [0,1]^2_{\leq}$ and let $f,g \in \cH$ such that they both have support in $[s,t]$ and $f,g \in \mathcal{C}^\gamma([s,t])$ for some $\gamma>1-2H$. Then we have 
    \begin{equs}
        \langle  \Pi_{\mathcal{H}_s^\perp} f,  \Pi_{\mathcal{H}_s^\perp} g \rangle_{\mathcal{H}} =\int_s^t \int_s^t f(u) g(r) Q_s(du,dr).
    \end{equs}
    
\end{lemma}

\begin{proof}
    For $u, v \in [s, t]$ and  $f=\bone_{[s,u]}, g = \bone_{[s,v]}$, by using the fact that $K^*$ is a linear isometry and  \cref{lem:representation_projection},  we have 
    \begin{equs}
       \langle  \Pi_{\mathcal{H}_s^\perp} f,  \Pi_{\mathcal{H}_s^\perp} g \rangle_{\mathcal{H}} & = \langle  \Pi_{\mathcal{H}_s^\perp} \bone_{[s,u]},  \Pi_{\mathcal{H}_s^\perp} \bone_{[s,v]} \rangle_{\mathcal{H}} 
     %  \\
    %   &=\langle  K^*(\Pi_{\mathcal{H}_s^\perp} \bone_{[0,u]}),  K^*(\Pi_{\mathcal{H}_s^\perp} \bone_{[0,v]}) \rangle_{L_2([0,1])}
       \\
       &= \langle  K^*(\Pi_{\mathcal{H}_s^\perp} \bone_{[s,u]}),  K^*(\Pi_{\mathcal{H}_s^\perp} \bone_{[s,v]}) \rangle_{L_2([0,1])}
       \\
       &= \int_s^{u \wedge v} K(u, r)K(v,r) \, dr 
       \\
       & = Q_s(u, v)
       \\
       &= \int_s^t \int_s^t f(r) g(r') Q_s(dr,dr').      \label{eq:H_Q}
    \end{equs}
    By linearity, the above relation holds for all $f,g \in \mathcal{E}$ which are supported in $[s,t]$. For a general $f$ as in the statement one can argue as in \cite[Remark 2.16]{GOT}, let $f^n$ be given by $f^n(x)=f(s+i(t-s)/n)$ for $x \in (s+i(t-s)/n, s+(i+1)(t-s)/n]$, $i =0, ..., n-1$. and let $f^n(x) =0$ for $x \notin [s,t]$, and define $g^n$ similarly.
    Obviously, we have 
    \begin{equs}                           \label{eq:to_pass_to_the_limit_soon...}
        \langle  \Pi_{\mathcal{H}_s^\perp} f^n,  \Pi_{\mathcal{H}_s^\perp} g^n \rangle_{\mathcal{H}} & =\int_s^t \int_s^t f^n(r) g^n(r') Q_s(dr,dr'). 
    \end{equs}
    Moreover,  it follows (see, \cite{2D_Young}) that 
    \begin{equs}
      \lim_{n, m \to \infty }  \int_s^t \int_s^t f^n(r) g^m(r') Q_s(dr,dr') & =
    \int_s^t \int_s^t f(r) g(r') Q_s(dr,dr')      \label{eq:after_limit_Q}
    \\
    \lim_{n, m \to \infty }  \int_0^t \int_0^t f^n(r) g^m(r') Q(dr,dr') & = \int_0^t \int_0^t f(r) g (r') Q(dr,dr')
\end{equs}
By the second equality above (with $f^m$ in place of $g^m$) and \eqref{eq:H_Q} with $s=0$ (in which case $\mathcal{H}_s^\perp=\mathcal{H}$), it follows that $f^n$ is Cauchy in $\mathcal{H}$. Similarly for $g^n$.  On the other hand $f^n \to f$ and $g^n \to g$ pointwise and by virtue of the embedding $\mathcal{H} \hookrightarrow L_1([0,1]) $ (see, \cite[Lemma 4.2]{CHLT15}), it follows that $f^n \to f, g^n \to g $ in $\mathcal{H}$. Consequently, the left hand side of \eqref{eq:to_pass_to_the_limit_soon...} converges to  $\langle  \Pi_{\mathcal{H}_s^\perp} f,  \Pi_{\mathcal{H}_s^\perp} g \rangle_{\mathcal{H}}$, which combined with \eqref{eq:after_limit_Q} finishes the proof. 
\end{proof}

By \cref{lem:lower-var} and \cref{cor:holder-controlled2}, the conditions of \cite[Corollary~6.10]{CHLT15} are satisfied by the covariance function $Q_s$. That and \cref{lem:H-norm-integral} give the following (recall that $\vartheta$ is introduced in \cref{lem:lower-var}). 
\begin{corollary}\label{cor:interpolation}
    Let $0\leq s<t\leq 1$. Let $\gamma>1-2H$. Let $f\in\cH$ such that   $\supp f\subset[s,t]$ and the restriction of $f$ to $[s,t]$ belongs to $\cC^\gamma$.
    Then there exists a constant $c=c(H,\gamma)>0$ such that the following bound holds:
    \begin{equ}
        \big\|\Pi_{\cH^\perp_s}f\|_{\cH}\geq c(t-s)^H\|f\|_{L^\infty([s,t])}\min\Big(1,\frac{\|f\|_{L^\infty([s,t])}}{\|f\|_{\cC^\gamma([s,t])}}\Big)^{\vartheta/2\gamma}.
    \end{equ}
\end{corollary}

\section{Properties of the flow of the driftless equation}
In this section we derive properties of the flow of a driftless analogue to \eqref{eq:main}, that is, we consider 
\begin{equ}\label{eq:flow-def}
\phi_{t}^{s,x}=x+ \int_{s}^t \blue{\sigma(\phi_r^{s,x})}\,d\blue{B_r},\quad t \in [s, 1].
\end{equ} 
Recall that the flow $\phi^{s,x}_t$ exists and it  is smooth and invertible in $x$ (see \cite[Theorem 8.15]{Friz-Hairer}). We will be denoting its inverse by $\overleftarrow{\phi}^{s,x}_t$. More precisely,  building upon ideas introduced in \cite{Lyons-Cass-Litterer} and \cite{INAHAMA}, we obtain estimates for $\phi^{s,x}$, its Jacobian, the inverse of the Jacobian, and their Malliavin derivatives (see, \cref{lem:bounds_for_J_sigma(phi)}, \cref{lem:flow-D-with-power}, and \cref{lem:Jacobi-Malliavin-Derivatives}). One important feature is that these are almost sure estimates which hold uniformly in $(s,x)$. 
% In this section we build on ideas introduced in \cite{Lyons-Cass-Litterer} and \cite{INAHAMA}. 

\begin{lemma}      \label{lem:bound_controlled_flow}
There exist constants $C=C(H_-,\| \sigma\|_{\cC^2})$ and $q=q(H_-)$  such that for all $s \in [0, 1]$, $x \in \R^d$,  we have 
    \begin{equs}  \label{eq:controlled_phi_as}
      \,   [\blue{\phi^{s, x}}]_{\mathcal{D}^{2H_-} _B([s,1])}  \leq C (1+[\blue{B}]_{\cR^{H_-}})^q.
    \end{equs}
\end{lemma}

\begin{proof}
Let $s \in [0, 1]$, $x \in \R^d$. For $(u,v) \in [s,1]^2_{\leq}$, we have    
    \begin{equs}
      \phi_{u,v}^{s,x}= \int_u^v \blue{\sigma( \phi^{s,x}_r)} \, d\blue{B_r}, 
    \end{equs}
    which gives 
    \begin{equs}
      | \phi_{u,v}^{s,x}-\sigma( \phi^{s,x}_u) B_{u, v}|  & = \Big| \int_u^v \blue{\sigma( \phi^{s,x}_r)} \, d\blue{B_r}-\sigma( \phi^{s,x}_u) B_{u, v} %B_{u, t}
      \Big| 
      \\
      &  \lesssim  [\blue{B}]_{\cR^{H_-}}\big( [\blue{\sigma( \phi^{s,x})}]_{\mathcal{D}^{2H_-}_B([u,v])} |u-v|^{3H_-} + |\nabla \sigma \sigma (\phi_u^{s,x})| |u-v|^{2H_-} \big),
    \end{equs}
    where we have used \eqref{eq:rough_integral_first_order}. This now implies that 
    \begin{equs}         \label{eq:bound_R^Y}
        \,  [R^\blue{ \phi^{s,x}}]_{\cC^{2H_-}_2([u,v])}  
      &  \lesssim [\blue{B}]_{\cR^{H_-}}\big( [\blue{\sigma( \phi^{s,x})}]_{\mathcal{D}^{2H_-}_B([u,v])} |u-v|^{H_-} + 1 \big),
    \end{equs}
    where we have used the fact that $\nabla \sigma \sigma (\phi_u^{s,x})$ is bounded by $\|\sigma\|_{\cC^1}^2$.  
    Next by using  \eqref{eq:composition_estimate_controlled}, we get 
    \begin{equs}        
      \,   [\blue{\sigma( \phi^{s,x})}]_{\mathcal{D}^{2H_-}_B([u,v])} &  \lesssim ( 1+\|\sigma( \phi^{s,x})\|^2_{\cC^0([s,1])}+ [B]^2_{\cC^{H_-}} ) ( [\blue{ \phi^{s,x}}]_{\mathcal{D}^{2H_-}_B([u,v])}+ \|\sigma( \phi^{s,x})\|_{\cC^0([s,1])})
      \\
     &  \lesssim ( 1+ [B]^2_{\cC^{H_-}} ) ( [\blue{ \phi^{s,x}}]_{\mathcal{D}^{2H_-}_B([u,v])}+1), 
    \end{equs}
    where we have used the boundedness of $\sigma$.  By replacing this in \eqref{eq:bound_R^Y}, we get 
    \begin{equs}               \label{eq:whatever}
        \,  [R^\blue{ \phi^{s,x}}]_{\cC^{2H_-}_2([u,v])}  
      &  \lesssim (1+[\blue{B}]^3_{\cR^{H_-}}) \big( [\blue{ \phi^{s,x}}]_{\mathcal{D}^{2H_-}_B([u,v])} |u-v|^{H_-} + 1 \big).
    \end{equs}
    Next, by using the regularity of $\sigma$ and   \eqref{eq:holder_bound_by_rough}  we have 
\begin{equs}         
   \, [\sigma( \phi^{s,x})]_{\cC^{H_-}([u,v])} & \lesssim [\phi^{s,x}]_{\cC^{H_-}([s,t])} \\&  \lesssim ([R^\blue{\phi^{s,x}}]_{\cC^{2H_-}_2([u,v])}|u-v|^{H_-}+ \|\sigma( \phi^{s,x})\|_{\cC^0([s,1])} [B]_{\cC^{H_-}})
   \\
   &  \lesssim ([R^\blue{\phi^{s,x}}]_{\cC^{2H_-}_2([u,v])}|u-v|^{H^-}+ [B]_{\cC^{H_-}}).
\end{equs}
 Adding the above inequality to 
\eqref{eq:whatever}, gives 
\begin{equs}              
        \,  [\blue{ \phi^{s,x}}]_{\mathcal{D}^{2H_-}_B([u,v])}
      &  \leq C (1+[\blue{B}]^3_{\cR^{H_-}}) \big( [\blue{ \phi^{s,x}}]_{\mathcal{D}^{2H_-}_B([u,v])} |u-v|^{H_-} + 1 \big),
    \end{equs}
    where $C$ depends only on $H_-$ and $\| \sigma\|_{\cC^2}$. The result now follows from \cref{lem:buckling_D_norm}. 
\end{proof}
The next result follows immediately from \cref{lem:bound_controlled_flow} and \eqref{eq:holder_bound_by_rough}. 
\begin{corollary}            \label{cor:phi-x}
    There exist constants $C=C(H_-,\| \sigma\|_{\cC^2})$ and $q=q(H_-)$ such that for all $(s,t) \in [0, 1]^2_{\leq}$, $x \in \R^d$,  we have 
    \begin{equs}  \label{eq:phi^x-x}
      |\phi^{s,x}_t-x| \leq C (1+[\blue{B}]_{\cR^{H_-}})^q|t-s|^{H_-}.
    \end{equs}
\end{corollary}

\begin{proposition}\label{Prop:sta-ini}
    Let $(\phi_{t}^{s,x})_{t\in[s,1]}$ be given by \eqref{eq:flow-def}, let X be a solution to the SDE \eqref{eq:main}. Then, for any $p \geq 1$ there exists a constant $C=C( H_-, p, C_{D,p}^X, C_{S,2p}^X, \|\sigma\|_{\cC^2}) $, such that for all $(s,t) \in [0, 1]^2_{\leq}$ we have 
    \begin{equs}
        \|X_t-\phi^{s, X_s}_t\|_{L_p(\Omega)} \leq C|t-s|^{1+\alpha H}.
    \end{equs}
\end{proposition}

\begin{proof}
    Let us fix $(s, t) \in [0, 1]^2_{\leq}$ and for $v \in [s, t]$ let us introduce the notation
\begin{equs}
    Y_v:=  X_v- \phi_v^{s,X_s}, \qquad Y'_v:= \sigma(X_v)- \sigma(\phi_v^{s,X_s}), \qquad Y''_v:= \nabla \sigma \sigma (X_v)- \nabla \sigma \sigma (\phi_v^{s,X_s}),
\end{equs}
and 
\begin{equs}
    \blue{Y}=(Y, Y'),  \qquad  \blue{Y'}=(Y', Y'').
\end{equs}
Notice that
    \begin{equs}
      Y_{s, t}=D^X_{s,t}- \int_s^t \blue{Y_r'} \, d\blue{B_r}, 
    \end{equs}
    which gives 
    \begin{equs}
      |Y_{s, t}-Y'_s B_{s, t}|  & =|D^X_{s, t}|+ \Big|  \int_s^t \blue{Y_r'} \, d\blue{B_r} -Y'_s B_{s, t}\Big| 
      \\
      & \lesssim  |D^X_{s,t}|+  [\blue{B}]_{\cR^{H_-}}\big( [\blue{Y'}]_{\mathcal{D}_B^{2H_-}([s,t])} |t-s|^{3H_-} + |Y''_s| |t-s|^{2H_-} \big),
    \end{equs}
    where we have used \eqref{eq:rough_integral_first_order}. 
    Next, notice that  $Y'_s=Y''_s=0$, which combined with the inequality above gives  
    \begin{equs}         \label{eq:estimate_Y_s_t}
        |Y_{s,t}| \lesssim  |D^X_{s,t}|+  [\blue{B}]_{\cR^{H_-}}[\blue{Y'}]_{\mathcal{D}^{2H_-}_B([u,v])} |t-s|^{3H_-}. 
    \end{equs}
    By  \eqref{eq:composition_estimate_controlled}, we get 
    \begin{equs}
      \,  [\blue{Y'}]_{\mathcal{D}^{2H_-}_B([s,t])} & \leq [ \blue{\sigma (X)}]_{\mathcal{D}^{2H_-}_B([s,t])} + [\blue{\sigma ( \phi^{s,X_s})}]_{\mathcal{D}^{2H_-}_B([s,t])} 
      \\
      & + \lesssim  \big( 1+  [B]^2_{\cC^{H_-}} \big) \big( 1+ [ \blue{X}]_{\mathcal{D}^{2H_-}_B([s,t])} + [\blue{\phi^{s,X_s}}]_{\mathcal{D}^{2H_-}_B([s,t])} \big)
    \end{equs}
    where we have used the fact that the Gubinelli derivatives of $X$ and $\phi^{s,x}$ are bounded by $\| \sigma\|_{\cC^0}$. Replacing the above in \eqref{eq:estimate_Y_s_t}, gives
    \begin{equs}
        |Y_{s,t}| \lesssim |D^X_{s,t}|+  \big(1+[\blue{B}]^3_{\cR^{H_-}}\big)\big( 1+ [ \blue{X}]_{\mathcal{D}^{2H_-}_B([s,t])} + [\blue{\phi^{s,X_s}}]_{\mathcal{D}^{2H_-}_B([s,t])} \big)  |t-s|^{3H_-}.
    \end{equs}
   By H\"older's inequality we obtain  
    \begin{equs}
        \|Y_{s,t}\|_{L_p(\Omega)} \lesssim \|D^X_{s,t}\|_{L_p(\Omega)}+  \big( \| [\blue{X}]_{\mathcal{D}^{2H_-} _B([s,t])} \|_{L_{2p}(\Omega)}+\| [\blue{\phi^{s,X_s}}]_{\mathcal{D}^{2H_-} _B([s,t])} \|_{L_{2p}(\Omega)} +1 \big)  |t-s|^{3H_-}.
    \end{equs}
    By \eqref{eq:bound:D^X}, \eqref{eq:bound_X_controlled}, and \cref{lem:bound_controlled_flow},  we get 
   \begin{equs}
       \|Y_{s,t}\|_{L_p(\Omega)} \lesssim  |t-s|^{1+\alpha H} +|t-s|^{3H_-}.
   \end{equs}
   The claim now follows by the observations that $Y_{s,t}= X_t- \phi_t^{s,X_s}$ and  $3H_->1> 1+\alpha H$. 
    \end{proof}

Next, we want to derive estimates for the jacobian matrix of the flow, that is, for  $J^{s,x}_t := \nabla_x \phi^{s,x}_t$. It is known that with probability one, for all $s \in [0,1], x \in \R^d$, we have that $\blue{J^{s,x}}:=(J^{s,x}, \nabla\sigma( \phi^{s,x}) J^{s,x} ), \blue{(J^{s,x})^{-1}}:=((J^{s,x})^{-1}, -\nabla \sigma( \phi^{s,x}) (J^{s,x})^{-1} ) \in \mathcal{D}_B^{2H_-}([s, 1])$, and satisfy the following equations (see, e.g., \cite[Theorem 8.15, p. 148]{Friz-Hairer})
\begin{equs}
    J^{s,x}_t & =  \mathbb{I}_{d}+ \int_s^t \blue{ \nabla\sigma(\phi^{s,x}_r) J^{s,x}_r} \, d \blue{B_r}, \qquad  t \in [s, 1],
    \\
    (J^{s,x}_t)^{-1} &= \mathbb{I}_{d} -\int_s^t \blue{ \nabla\sigma(\phi^{s,x}_r) (J^{s,x}_r)^{-1}} \, d \blue{B_r}, \qquad  t \in [s, 1],
\end{equs}
where $\mathbb{I}_d$ is the identity matrix.
It follows from the above (or simply by using the chain rule and the flow property $\phi^{s,x}_t=\phi_t^{0,\overleftarrow{\phi}^{0,x}_s}$) that
\begin{equ}\label{eq:Jacobian"flow"}
    J^{s,x}_t=J^{0,y}_t (J^{0,y}_s)^{-1}|_{y=\overleftarrow{\phi}^{0,x}_s}.
\end{equ} 
To obtain estimates simultaneously for both $J^{s,x}$ and $(J^{s,x})^{-1}$, let $E$ be a finite dimensional normed linear space, $F\in \cC^2(\R^d;  \mathcal{L}(E; \mathcal{L}(\R^{d_0} ; E))$ and $s \in [0,t]$, $x \in \R^d$,  $\blue{M}= (M, M') \in \mathcal{D}^{2H_-}_B([s,1]; \mathcal{L}(\R^{d_0}; E ))$ and consider the equation
\begin{equs} \label{eq:general_equation_Jacobian}
    V_t=V_0+ \int_s^t (\blue{F(\phi^{s,x}_r) V_r+M_r} )\, d\blue{B_r}, \qquad t \in [s, 1].
\end{equs}
Let us set $1/\rho=H_-$. We will first obtain estimates in $\rho$-variation norms by using some results by \cite{Lyons-Cass-Litterer} and those estimates will be upgraded to H\"older estimates at a second step. For this, let us recall some facts and some let us introduce some notation related to $\rho$-variation norms.  For $\blue{X}=(X,X') \in \mathcal{D}_B^{2H_-}$, 
recall that one has the estimate 
\begin{equs}
\big| \int_s^t \blue{X_r} \, d\blue{B_r} -X_sB_{s,t}-X'_s\mathbb{B}_{s,t}\big| 
    &  \leq C \big( [R^\blue{X}]_{\mathcal{V}^{\rho/2}_2([s,t])}[B]_{\mathcal{V}^{\rho}([s,t])}+ [X']_{\mathcal{V}^{\rho}([s,t])}[\mathbb{B}]_{\mathcal{V}^{\rho/2}_2([s,t])} \big).
\end{equs}
which implies 
\begin{equs}
  \big| \int_s^t \blue{X_r} \, d\blue{B_r} -X_sB_{s,t}\big| 
    & \leq C \big( [R^X]_{\mathcal{V}^{\rho/2}_2([s,t])}[B]_{\mathcal{V}^{\rho}([s,t])}+ ([X']_{\mathcal{V}^{\rho}([s,t])}+|X'_s|)[\mathbb{B}]_{\mathcal{V}^{\rho/2}_2([s,t])} \big) ,    
    \\
    \label{est:rough_integral_p_variation}
\end{equs}
for a constant $C=C(H_-)$. Further, let us set 
\begin{equs}   \label{eq:def_bold_B_p_variation}
   \,  [\blue{B}]_{\mathcal{V}^{\rho}([s,t])}:= \big( [B]^\rho_{\mathcal{V}^{\rho}([s,t])}+ [\mathbb{B}]^{\rho/2}_{\mathcal{V}^{\rho/2}_2([s,t])} \big)^{1/\rho},
\end{equs}
and for $\blue{X}=(X, X') \in \mathcal{D}^{2H_-}_B$, let us set 
\begin{equs}
   \,  [\blue{X}]_{\mathbf{V}^{\rho/2}_B([s, t])} & : = [R^\blue{X}]_{\mathcal{V}^{\rho/2}_2([s,t])} + [X']_{\mathcal{V}^{\rho}([s,t])}
   \\
    \|\blue{X}\|_{\mathbf{V}^{\rho/2}_B([s, t])} & : =   [\blue{X}]_{\mathcal{V}^{\rho/2}_B([s, t])}+ \|X\|_{\cC^0([s,t])}+\|X'\|_{\cC^0([s,t])}
\end{equs}
The following lemma is quite standard and its proof is straightforward. 
\begin{lemma}      
  Set $\rho=1/H_- $. There exists a constant $C=C(H_-)$ such that  for all $(s,t) \in [0,1]^2_{\leq}$, $\blue{X}=(X,X') \in \mathcal{D}^{2H_-}_B([s,t])$, and $F \in \mathcal{C}^2$,  we have 
    \begin{equs}
  \,    [X]_{\mathcal{V}^\rho([s,t])} & \leq C([R^{\blue{X}}]_{\mathcal{V}^{\rho/2}_2([s,t])} + \|X'\|_{\cC^0([s,t])}[B]_{\mathcal{V}^\rho([s,t])}),   \label{eq:p_var_bounded_remainder}
\\         
\,   [R^{\blue{F(X)}}]_{\mathcal{V}^{\rho/2}_2([s,t])}  &  \leq C\|F\|_{\mathcal{C}^2} \big(  [R^{\blue{X}}]_{\mathcal{V}^{\rho/2}_2([s,t])}+ \|X'\|_{\mathcal{C}^0([s,t])}[R^\blue{X}]_{\mathcal{V}^{\rho/2}_2([s,t])} [B]_{\mathcal{V}^\rho([s,t])}
  \\ 
  & \qquad +\|X'\|_{\mathcal{C}^0([s,t])}^2 [B]^2_{\mathcal{V}^\rho([s,t])} \big). 
    \label{eq:composition_bound_p_variation}
\end{equs}
\end{lemma}

Next, we  prove a variant of 
\cref{lem:bound_controlled_flow}. 
\begin{lemma}             \label{lem:R^phi_order_one}
 Set $\rho=1/H_-$.  There exists a constant $\delta=\delta(H_-,\|\sigma||_{\cC^2})  \in (0, 1)$ such that the following holds: for all $x \in \R^d$, $(s, u, v) \in [0, 1]^3_{\leq}$, if $[\blue{B}]_{\mathcal{V}^{\rho}([u,v])} \leq \delta$, then $[R^{\blue{\phi^{s,x}_{\cdot}}}]_{\mathcal{V}^{\rho/2}_2([u,v])} \leq 1$. 
\end{lemma}

\begin{proof}
To ease the notation, let set 
\begin{equs}
    Y_r:= \phi_r^{s,X_s}, \qquad Y'_r:= \sigma(\phi_r^{s,X_s}), \qquad  Y''_r:= \nabla \sigma \sigma (\phi_r^{s,X_s}), \qquad  \blue{Y}= (Y, Y'), \qquad \blue{Y'}= (Y', Y''). 
\end{equs}
    With the same arguments as in  %Lemma 
    \cref{lem:bound_controlled_flow}, this time using \eqref{est:rough_integral_p_variation}, we get 
    \begin{equs}
       \,  [R^\blue{Y}&]_{\mathcal{V}^{\rho/2}_2([u,v])} \\&\lesssim  [R^{\blue{Y'}}]_{\mathcal{V}^{\rho/2}_2([u,v])} [B]_{\mathcal{V}^{\rho}([u,v])}+([Y']_{\mathcal{V}^{\rho}([u,v])}+\|Y'\|_{\cC^0([u,v])})[\mathbb{B}]_{\mathcal{V}^{\rho/2}_2([u,v])}.  \label{eq:buckle_R^Y_p_var}
    \end{equs}
   Moreover, by 
    \eqref{eq:composition_bound_p_variation} we have
    \begin{equs}
  \,   [R^{\blue{Y'}}]_{\mathcal{V}^{\rho/2}_2([u,v])}  \lesssim  [R^\blue{Y}]_{\mathcal{V}^{\rho/2}_2([u,v])} + [R^\blue{Y}]_{\mathcal{V}^{\rho/2}_2([u,v])} [B]_{\mathcal{V}^\rho([u,v])} +[B]^2_{\mathcal{V}^\rho([u,v])}.   \label{eq:bound_1_Y'_p_variation}
\end{equs}
Using the regularity and the boundedness of $\sigma$ and its derivatives, we see that 
\begin{equs}
  \,   [Y']_{\mathcal{V}^{\rho}([u,v])} \lesssim [Y]_{\mathcal{V}^{\rho}([u,v])} \lesssim  [R^Y]_{\mathcal{V}^{\rho/2}_2([u,v])}+ [B]_{\mathcal{V}^{\rho}([u,v])}.
\end{equs}
Plugging the above and \eqref{eq:bound_1_Y'_p_variation} in \eqref{eq:buckle_R^Y_p_var}, and keeping in mind that $\|Y'\|_{\cC^0([u,v])} \leq \|\sigma\|_{\cC^0}$, we get 
   \begin{equs}
       \,  [&R^\blue{Y}]_{\mathcal{V}^{\rho/2}_2([u,v])} \\& \leq C  \big([R^{\blue{Y}}]_{\mathcal{V}^{\rho/2}_2([u,v])} + [R^{\blue{Y}}]_{\mathcal{V}^{\rho/2}_2([u,v])} [B]_{\mathcal{V}^{\rho}([u,v])}+[B]^2_{\mathcal{V}^{\rho}([u,v])} \big) [B]_{\mathcal{V}^{\rho}([u,v])}
       \\
       & \quad+ C\big( 1+ [R^{\blue{Y}}]_{\mathcal{V}^{\rho/2}_2([u,v])}+ [B]_{\mathcal{V}^{\rho}([u,v])} \big) [\mathbb{B}]_{\mathcal{V}^{\rho/2}_2([u,v])},
    \end{equs} 
    where $C$ is a constant depending only on $H_-$ and $\|\sigma\|_{\cC^2}$ and can be assumed to be greater than $1$.
If $[\blue{B}]_{\mathcal{V}^{\rho}([u,v])} \leq \delta:= (10 C)^{-1}$, then we can rearrange the above to get 
that $[R^\blue{Y}]_{\mathcal{V}^{\rho/2}_2([u,v])} \leq 1$. This finishes the proof. 
\end{proof}

\begin{proposition}  \label{pro:bound_V}
   Let $F \in \cC^2$ and set $\rho=1/H_-$. There exists a random variable $\Theta$ such that for all $x \in \R^d$, $s\in [0,1]$ and $\blue{M}=(M, M') \in \mathcal{D}^{2H_-}_B([s,1])$, if $V$ is a solution of \eqref{eq:general_equation_Jacobian}, then 
   \begin{equs}                      
       \|V\|_{\cC^0([s,1])}  \leq \Theta (|V_s|+ \|\blue{M}\|_{\mathbf{V}^{\rho/2}_B([s, 1])}).
       \label{eq:sup estimate V}
   \end{equs}
   Moreover, for any $p \geq 1$, there exists a constant $C=C(p, \|F\|_{\cC^2}, \|\sigma\|_{\cC^2}, H_-)$, such that  $\|\Theta\|_{L_p(\Omega)}\leq C$.
\end{proposition}

To prove the proposition above we will need a collection of estimates. 

\begin{lemma}    \label{lem:inequalities_FV}
   Let $F \in \cC^2$ and set $\rho=1/H_-$. There exists a constant $C=C(\|F\|_{\cC^2}$, $\|\sigma\|_{\cC^2},H_-)$ such that for all $x \in \R^d$, $s\in [0,1]$ and $\blue{M}=(M, M') \in \mathcal{D}^{2H_-}_B([s,1])$, if $V$ is a solution of \eqref{eq:general_equation_Jacobian}, then for all $(u,v) \in [s,1]_{\leq}$ we have 
    \begin{equs}
     \,    [R^{\blue{F(\phi^{s,x}) V}}]_{\mathcal{V}^{\rho/2}_2([u,v])}  \leq& C \big(  [R^\blue{V}]_{\mathcal{V}^{\rho/2}_2([u,v])} + \|V\|_{\cC^0([u,v])} [R^{\blue{F(\phi^{s,x})}}]_{\mathcal{V}^{\rho/2}_2([u,v])} \big) 
  \\  & +C[B]_{\mathcal{V}^{\rho}([u,v])}\Big([R^\blue{V}]_{\mathcal{V}^{\rho/2}_2([u,v])}+(\|V\|_{\cC^0([u,v])}\\  &\qquad\qquad\qquad\qquad+ \|M\|_{\cC^0([u,v])})[B]_{\mathcal{V}^{\rho}([u,v])} \Big),
    \\
   \,  [(F(\phi^{s,x}) V)']_{\mathcal{V}^{\rho}([u,v])}\leq &C \Big( [R^\blue{V}]_{\mathcal{V}^{\rho/2}_2([u,v])} + (\|V\|_{\cC^0([u,v])}+\|M\|_{\cC^0([u,v])})
   \\
   &\quad \times ([R^{\blue{\phi^{s,x}}}]_{\mathcal{V}^{\rho/2}_2([u,v])}+[B]_{\mathcal{V}^{\rho}([u,v])}) \Big),
    \\
 \|(F(\phi^{s,x}) V)'\|_{\cC^0([u,v])}  \leq   & C ( \|V\|_{\cC^0([u,v])}+\|M\|_{\cC^0([u,v])}).
     \end{equs}
\end{lemma}

\begin{proof}
We start with the first inequality. Similarly to \eqref{eq:reminder_product}, we have 
\begin{equs}   
   \,  [R^{\blue{F(\phi^{s,x}) V}}]_{\mathcal{V}^{\rho/2}_2([u,v])} & \lesssim \| F(\phi^{s,x})\|_{\cC^0([u,v])}[R^\blue{V}]_{\mathcal{V}^{\rho/2}_2([u,v])}
   \\
   & \qquad +\|V\|_{\cC^0([u,v])}[R^{\blue{F(\phi^{s,x})}}]_{\mathcal{V}^{\rho/2}_2([u,v])}
   \\
   & \qquad +\| \nabla F(\phi^{s,x})\sigma(\phi^{s,x}) \|_{\cC^0([u, v])}[B]_{\mathcal{V}^{\rho}_{u,v}}[V]_{\mathcal{V}^{\rho}([u,v])}
   \\
   & \lesssim [R^\blue{V}]_{\mathcal{V}^{\rho/2}_2([u,v])}
  +\|V\|_{\cC^0([u,v])}[R^{\blue{F(\phi^{s,x})}}]_{\mathcal{V}^{\rho/2}_2([u,v])}
 \\
 & \qquad +[B]_{\mathcal{V}^{\rho}_{u,v}}[V]_{\mathcal{V}^{\rho}([u,v])} \label{eq:bound_R^F}
\end{equs}
By \eqref{eq:p_var_bounded_remainder}we have 
\begin{equs}
 \,    [V]_{\mathcal{V}^{\rho}([u,v])} &\lesssim  [R^\blue{V}]_{\mathcal{V}^{\rho/2}_2([u,v])} + \| F(\phi^{s,x})V+M\|_{\cC^0([u,v])} [B]_{\mathcal{V}^{\rho}([u,v])}
 \\
 &\lesssim [R^\blue{V}]_{\mathcal{V}^{\rho/2}_2([u,v])} + (\|V\|_{\cC^0([u,v])}+\|M\|_{\cC^0([u,v])} )[B]_{\mathcal{V}^{\rho}([u,v])}. 
\end{equs}
Hence, combining this with \eqref{eq:bound_R^F} gives the first desired inequality. 
For the second inequality,  we have that 
\begin{equs}
    \big(F(\phi^{s,x})V\big)'_u z= (\nabla F(\phi^{s,x}_u) \sigma (\phi^{s,x}_u) z)V_u +  F(\phi^{s,x}_u)\big( (F(\phi^{s,x}_u)V_u+M_u)\big)z,
    \label{eq:Leibniz_derivative_FV}
\end{equs}
for $z \in \R^{d_0}$, from which, by using the boundedness of $F$ and $\sigma$ and their derivatives, we get 
\begin{equs}
   \,   [(F(\phi^{s,x}) V)']_{\mathcal{V}^{\rho}([u,v])} \lesssim [V]_{\mathcal{V}^{\rho}([u,v])}  + (\|V\|_{\cC^0([u,v])}+\|M\|_{\cC^0([u,v])})[\phi^{s,x}]_{\mathcal{V}^{\rho}([u,v])}.
\end{equs}
The above inequality, combined with \eqref{eq:p_var_bounded_remainder} gives 
\begin{equs}
   \,   &[(F(\phi^{s,x}) V)']_{\mathcal{V}^{\rho}([u,v])}\\& \lesssim [R^\blue{V}]_{\mathcal{V}^{\rho/2}_2([u,v])} + (\|V\|_{\cC^0([u,v])}+\|M\|_{\cC^0_{u,v}})([R^{\blue{\phi^{s,x}}}]_{\mathcal{V}^{\rho/2}_2([u,v])}+[B]_{\mathcal{V}^{\rho}([u,v])}). 
\end{equs}
Lastly, the third inequality follows trivially from \eqref{eq:Leibniz_derivative_FV}. 
\end{proof}
%%%%%%%%%%%%%%%%%%%%%%%%%%%%%%%%%%%%%%%%%%%%%29.07.5:30 pm
Similarly, one can get the corresponding estimates in the H\"older scale. 
\begin{lemma}    \label{lem:inequalities_FV_Holder}
    Let $F \in \cC^2$.   There exists a constant $C=C(\|F\|_{\cC^2}, \|\sigma\|_{\cC^2}, H_-)$ such that for all $x \in \R^d$, $s\in [0,1]$ and $\blue{M}=(M, M') \in \mathcal{D}^{2H_-}_B([s,1])$, if $V$ is a solution of \eqref{eq:general_equation_Jacobian}, then for all $(u,v) \in [s,1]_{\leq}$ we have 
    \begin{equs}
     \,    [R^{\blue{F(\phi^{s,x}) V}}]_{\cC^{2H_-}_2([u,v])} 
    & \leq C \big(  [R^\blue{V}]_{\cC^{2H_-}_2([u,v])} + \|V\|_{\cC^0([u,v])} [R^{\blue{F(\phi^{s,x})}}]_{\cC^{2H_-}_2([u,v])} \big)
     \\
    & \qquad +C[B]_{\cC^{H_-}([u,v])}\Big([R^\blue{V}]_{\cC^{2H_-}_2([u,v])}+(\|V\|_{\cC^0([u,v])} \\
    & \qquad\qquad\qquad\qquad\qquad+ \|M\|_{\cC^0([u,v])})[B]_{\cC^{H_-}([u,v])} \Big),
    \\
   \,  [(F(\phi^{s,x}) V)']_{\cC^{H_-}([u,v])} 
    & \leq C \big( [R^\blue{V}]_{\cC^{2H_-}_2([u,v])} + (\|V\|_{\cC^0([u,v])}+\|M\|_{\cC^0([u,v])})\\
    &\qquad\qquad\qquad\qquad\qquad([R^{\blue{\phi^{s,x}}}]_{\cC^{2H_-}_2([u,v])}+[B]_{\cC^{H_-}([u,v])}) \big),
    \\
    \|(F(\phi^{s,x}) V)'\|_{\cC^0([u,v])}  &\leq C ( \|V\|_{\cC^0([u,v])}+\|M\|_{\cC^0([u,v])}).
     \end{equs}
\end{lemma}

\begin{remark}
From the above lemma,  \eqref{eq:composition_Holder_R}, and the fact that the Gubinelli derivative of $\phi^{s,x}$ is bounded by $\| \sigma\|_{\cC^0}$, we get that 
\begin{equs}
   \, & [\blue{F(\phi^{s,x}) V}]_{\mathcal{D}^{2H_-}_B([u,v])} 
\\& \leq C(1+[B]_{\cC^{H_-}([u,v])}) [R^\blue{V}]_{\cC^{2H_-}_2([u,v])}
\\
& \qquad +C( \|V\|_{\cC^0([u,v])}+\|M\|_{\cC^0([u,v])})(1+[R^{\blue{\phi^{s,x}}}]_{\cC^{2H_-}_2([u,v])}+ [B]_{\cC^{H_-}([u,v])})^2,      \label{eq:bound_F(phi)V_cleaner}
\end{equs}
where $C$ is as in %Lemma 
    \cref{lem:inequalities_FV_Holder}. 
\end{remark}

\begin{proof}[Proof of  \cref{pro:bound_V}]
    By using the equation and  \eqref{est:rough_integral_p_variation} we have that any  $(u,v) \in [s, 1]_{\leq}^2 $
    \begin{equs}
      \,   [R^\blue{V}]_{\mathcal{V}^{\rho/2}_2([u,v])} &\leq C\big( [R^{\blue{F(\phi^{s,x}) V+M}}]_{\mathcal{V}^{\rho/2}_2([u,v])}[B]_{\mathcal{V}^{\rho}([u,v])} \big) 
        \\
       & \qquad  + C\big([(F(\phi^{s,x}) V+M)']_{\mathcal{V}^{\rho}([u,v])}+|(F(\phi^{s,x}) V+M)'_u|\big)[\mathbb{B}]_{\mathcal{V}^{\rho/2}_2([u,v])}
       \\
       & \leq  C\big( [R^{\blue{F(\phi^{s,x}) V}}]_{\mathcal{V}^{\rho/2}_2([u,v])}[B]_{\mathcal{V}^{\rho}([u,v])} \big) 
        \\
       & \qquad  + C\big([(F(\phi^{s,x}) V)']_{\mathcal{V}^{\rho}([u,v])}+\|(F(\phi^{s,x}) V)'\|_{\cC^0([u,v])}\big)[\mathbb{B}]_{\mathcal{V}^{\rho/2}_2([u,v])}
       \\
       & \qquad +C [R^\blue{M}]_{\mathcal{V}^{\frac{\rho}{2}}_2([u,v])}[B]_{\mathcal{V}^{\rho}([u,v])} + \|\blue{M}\|_{\mathbf{V}^{\rho/2}_B([u, v])}[\mathbb{B}]_{\mathcal{V}^{\rho/2}_2([u,v])}.
    \end{equs}
    Next, let us assume that $u,v$ are close enough so that $[\blue{B}]_{\mathcal{V}^{\rho}([u,v])} \leq \delta \wedge \tilde{\delta}=: \xi \in (0,1)$, where $\tilde{\delta} \in (0,1)$ will be determined later and $\delta$ is provided by %Lemma 
    \cref{lem:R^phi_order_one}. Then, the above simplifies to 
    \begin{equs}
      \,   [R^\blue{V}]_{\mathcal{V}^{\rho/2}_2([u,v])} &\leq \tilde{\delta} C\big( [R^{\blue{F(\phi^{s,x}) V}}]_{\mathcal{V}^{\rho/2}_2([u,v])}+
       [(F(\phi^{s,x}) V)']_{\mathcal{V}^{\rho}([u,v])}+\|(F(\phi^{s,x}) V)'\|_{\cC^0([u,v])})
       \\
       & \qquad +C \|\blue{M}\|_{\mathbf{V}^{\rho/2}_B([u, v])}.
    \end{equs}
    Next, by %Lemma 
    \cref{lem:inequalities_FV} combined with the fact that $[\blue{B}]_{\mathcal{V}^{\rho}_{s,t}} \leq \delta \wedge \tilde{\delta}$ and %Lemma 
    \cref{lem:R^phi_order_one}, we get 
 \begin{equs}                 
      \,   [R^\blue{V}]_{\mathcal{V}^{\rho/2}_2([u,v])} &\leq \tilde{\delta} C\big( [R^\blue{V}]_{\mathcal{V}^{\rho/2}_2([u,v])}+ 
       \|V\|_{\cC^0([u,v])}(1+[R^{\blue{F(\phi^{s,x})}}]_{\mathcal{V}^{\rho/2}_2([u,v])})
       \\
       & \qquad +C\|\blue{M}\|_{\mathbf{V}^{\rho/2}_B([s, 1])}.
       \\
       \label{eq:R^V_delta_tilde}
       \end{equs}
       Moreover, again  the fact that $[\blue{B}]_{\mathcal{V}^{\rho}([s,t])} \leq \delta \wedge \tilde{\delta}$ combined with  
    \eqref{eq:composition_bound_p_variation} and %Lemma 
    \cref{lem:R^phi_order_one} implies that we can drop the term $[R^{\blue{F(\phi^{s,x})}}]_{\mathcal{V}^{\rho/2}_2([u,v])}$
       in \eqref{eq:R^V_delta_tilde}, to arrive at 
      \begin{equs}                   
      \,   [R^\blue{V}]_{\mathcal{V}^{\rho/2}_2([u,v])} &\leq \tilde{\delta} C\big( [R^\blue{V}]_{\mathcal{V}^{\rho/2}_2([u,v])}+ 
       \|V\|_{\cC^0([u,v])}\big)+C \|\blue{M}\|_{\mathbf{V}^{\rho/2}_B([s, 1])}.
       \end{equs}
       Provided that $\tilde{\delta}$ is sufficiently small, this gives 
      \begin{equs}
      \,   [R^\blue{V}]_{\mathcal{V}^{\rho/2}_2([u,v])} &\leq \tilde{\delta} C 
       \|V\|_{\cC^0([u,v])}+C\|\blue{M}\|_{\mathbf{V}^{\rho/2}_B([s, 1])}.
       \end{equs}
       In addition, by using \eqref{eq:p_var_bounded_remainder} and the above inequality, we see that 
       \begin{equs}
            \|V\|_{\cC^0([u,v])} &\leq |V_u|+  [V]_{\mathcal{V}^{\rho}([u,v])} 
            \\
            & \leq |V_u|+ C\big([R^\blue{V}]_{\mathcal{V}^{\rho/2}_2([u,v])}+ (\|V\|_{\cC^0([u,v])}+\|M\|_{\cC^0([u,v])}) [B]_{\mathcal{V}^{\rho}([u,v])} \big)
            \\
            & \leq |V_u|+ C\big(\tilde{\delta} C 
       \|V\|_{\cC^0([u,v])}+C \|\blue{M}\|_{\mathbf{V}^{\rho/2}_B([s, 1])}+ \tilde{\delta}(\|V\|_{\cC^0([u,v])}+\|M\|_{\cC^0([u,v])})  \big),
       \end{equs}
       which, provided again that $\tilde{\delta}$ is sufficiently small, gives 
       \begin{equs}       \label{eq:to be itereted}
            \|V\|_{\cC^0([u,v])} &\leq 2 |V_u|+  C\|\blue{M}\|_{\mathbf{V}^{\rho/2}_B([s, 1])}.
       \end{equs}
       We can now define inductively
          $\tau_0=0$, $\tau_{i+1}:= \inf\{ t\geq \tau_i : [\blue{B}]_{\mathcal{V}^{\rho}_{[\tau_i,t]}} \geq \delta \wedge \tilde{\delta} \}$
       and set $N= \sup \{ i : \tau_i \geq 1\}$ ($N=N(\omega)$ is an almost surely finite random variable, see, \cite[Lemma 4.9]{Lyons-Cass-Litterer}). Further, let us set $u_0=s$ and $u_i= \tau_{k+i}$, for $i \geq 1$, where $k= \inf\{i : \tau_i >s\}$. Then, by applying \eqref{eq:to be itereted} with $u=u_i$ and $v=u_{i+1}$ and iterating, it follows that 
       \begin{equs}
         \|V\|_{\cC^0([s,1])} \leq C 2^N(|V_s|+\|\blue{M}\|_{\mathbf{V}^{\rho/2}_B([s, 1])}).
       \end{equs}
       This shows \eqref{eq:sup estimate V} with $\Theta= C2^N$. Finally, from 
       \cite[Theorem 6.3]{Lyons-Cass-Litterer} (by choosing $q=(1/p+1/2)^{-1}$ in that theorem, see also Corollary 5.5 therein), it follows that there exist positive constants $C_1$, $c_2$ such that $\mathbb{P}(N >n) \leq C_1 \exp( -c_2 n^{1+2/p})$, which in particular implies that $\Theta$ has finite moments of any order. This finishes the proof. 
\end{proof}
Next, we show that the supremum estimates for $V$ can be upgraded to estimates in the H\"older scale. 

\begin{proposition}\label{pro:bound_V_Holder}
  Let $F \in \cC^2$. There exists a constants $C=C(\|F\|_{\cC^2}, \|\sigma\|_{\cC^2}, H_-)$ and $q=q(H_-)$ such that for all $x \in \R^d$, $s\in [0,1]$ and $\blue{M}= (M, M') \in \mathcal{D}^{2H_-}_B([s,1])$, if $V$ is a solution of \eqref{eq:general_equation_Jacobian}, then 
       \begin{equs}
\|\blue{V}\|_{\mathcal{D}^{2H_-}_B([s,1])} \leq C (1+[\blue{B}]_{\cR^{H_-}})^q ( \|V\|_{\cC^0([s,1])}+\|\blue{M}\|_{\mathcal{D}^{2H_-}_B([s,1])}).
   \end{equs}
\end{proposition}
\begin{proof}
As usual, by using \eqref{eq:rough_integral_first_order} we get 
\begin{equs}
    \, &  [R^\blue{V}]_{\mathcal{C}^{2H_-}_2([u,v])} 
    \\
     \lesssim &  [\blue{B}]_{\cR^{H_-}}\big([\blue{F(\phi^{s,x}_r)V+M}]_{\mathcal{D}^{2H_-}_B([u,v])}|u-v|^{H_-} +\|V\|_{\cC^0([u,v])}+\|M\|_{\cC^0([u,v])}\big)
     \\
      \lesssim  & [\blue{B}]_{\cR^{H_-}}\big([\blue{F(\phi^{s,x}_r)V}]_{\mathcal{D}^{2H_-}_B([u,v])}|u-v|^{H_-}  + \|\blue{M}\|_{\mathcal{D}^{2H_-}_B([u,v])}+\|V\|_{\cC^0([u,v])}\big).
    \end{equs}
Further, by using \eqref{eq:bound_F(phi)V_cleaner} we have 
\begin{equs}
    \, [R^\blue{V}]_{\mathcal{C}^{2H_-}_2([u,v])}\lesssim
     &  (1+ [\blue{B}]_{\cR^{H_-}})^2  [R^\blue{V}]_{\mathcal{C}^{2H_-}_2([u,v])}|u-v|^{H_-} +  K_0, 
    \end{equs}
    where 
    \begin{equs}
     K_0 = ( \|V\|_{\cC^0([s,1])}+\|\blue{M}\|_{\mathcal{D}^{2H_-}_B([s,1])})(1+[R^{\blue{\phi^{s,x}}}]_{\cC^{2H_-}_2([s,1])}+ [B]_{\cC^{H_-}([s,1])})^3.
    \end{equs}
    Further, by \cref{lem:bound_controlled_flow} we see that 
    \begin{equs}
    K_0 \lesssim  K:=   ( \|V\|_{\cC^0([s,1])}+\|\blue{M}\|_{\mathcal{D}^{2H_-}_B([s,1])})(1+ [\blue{B}]_{\cR^{H_-}})^{(12+9/H_-)},  
    \end{equs}
    hence we conclude that 
   \begin{equs}
    \, [R^\blue{V}]_{\mathcal{C}^{2H_-}_2([u,v])}\lesssim
     &  (1+ [\blue{B}]_{\cR^{H_-}})^2  [R^\blue{V}]_{\mathcal{C}^{2H_-}_2([u,v])}|u-v|^{H_-} +  K. 
\label{eq:towards_gronwal_holder_jacobian}
    \end{equs}

    In addition, we see that 
    \begin{equs}
     \,    [V']_{\cC^{H_-}([u,v])}& = [F(\phi^{s,x})V+M]_{\cC^{H_-}([u,v])}
     \\ 
     & \leq \|F\|_{\cC^0}[V]_{\cC^{H_-}([u,v])}+ \|V\|_{\cC^0([u,v])}\|F\|_{\cC^1}[\phi^{s,x}]_{\cC^{H_-}([u,v])} + [M]_{\cC^{H_-}([u,v])}
     \\
     & \lesssim [V]_{\cC^{H_-}([u,v])}+ \|V\|_{\cC^0([u,v])}[\phi^{s,x}]_{\cC^{H_-}([u,v])}+ [M]_{\cC^{H_-}([u,v])}
     \\
     & \lesssim [R^\blue{V}]_{\mathcal{C}^{2H_-}_2([u,v])} |u-v|^{H_-}+ \| F(\phi^{s,x})V+M\|_{\cC^0([u,v])}
     \\
     & \qquad +\|V\|_{\cC^0([u,v])}( [R^{\blue{\phi^{s,x}}}]_{\mathcal{C}^{2H_-}_2([s,1])}+ [B]_{\cC^{H_-}})
     \\
     & \qquad + [R^\blue{M}]_{\mathcal{C}^{2H_-}_2([s,1])}+ \|M'\|_{\cC^0([s,1])}[B]_{\cC^{H_-}}
     \\
     & \lesssim [R^\blue{V}]_{\mathcal{C}^{2H_-}_2([u,v])} |u-v|^{H_-}+K, 
    \end{equs}
    where for the third inequality we have used \eqref{eq:holder_bound_by_rough}. From this and \eqref{eq:towards_gronwal_holder_jacobian}, we get that 
    \begin{equs}
      \,   [\blue{V}]_{\mathcal{D}^{2H_-}_B([u,v])} \leq C  (1+ [\blue{B}]_{\cR^{H_-}})^2[\blue{V}]_{\mathcal{D}^{2H_-}_B([u,v])} |u-v|^{H-}+CK,
    \end{equs}
    for some constant depending only on $\|F\|_{\cC^2}$, $\|\sigma\|_{\cC^2}$, and $H_-$. By \ref{lem:buckling_D_norm} we get some other constant $C$ (with the same dependence) that 
   \begin{equs}
      \,   [\blue{V}]_{\mathcal{D}^{2H_-}_B([s,1])} \leq C  (1+ [\blue{B}]_{\cR^{H_-}})^{1+(2/H_-)} K.       \label{eq:bound_V_D_byC_0}
    \end{equs}
   In addition, we clearly have that $\|V\|_{\cC^0([s,1])}$ and $\|F(\phi^{s,x}) V+M\|_{\cC^0([s,1])}$ can be bounded by the right hand side of \eqref{eq:bound_V_D_byC_0}. Consequently, the claim follows. 
\end{proof}   
The following is a direct consequence of  \cref{pro:bound_V}, \cref{pro:bound_V_Holder}, and the fact that $[\cdot]_{\mathcal{V}^{\rho}} \leq [\cdot]_{\mathcal{C}^{H_-}}$ and $[\cdot]_{\mathcal{V}^{\rho/2}_2} \leq [\cdot]_{\mathcal{C}^{2H_-}_2}$. 
\begin{corollary}     \label{cor:V_by_M}
    Let $F \in \cC^2$. There exists a random variable $\Theta$ such that for all $x \in \R^d$, $s\in [0,1]$ and $\blue{M}=(M, M') \in \mathcal{D}^{2H_-}_B([s,1])$, if $V$ is a solution of \eqref{eq:general_equation_Jacobian}, then 
   \begin{equs}                      
       \|\blue{V}\|_{\cD^{2H_-}_B([s, 1])} \leq \Theta (|V_s|+ \|\blue{M}\|_{\cD^{2H_-}_B([s, 1])}). 
   \end{equs}
    Moreover, for any $p \geq 1$, there exists a constant $C=C(p, \|F\|_{\cC^2}, \|\sigma\|_{\cC^2}, H_-)$, such that  $\|\Theta\|_{L_p(\Omega)}\leq C$.
\end{corollary}

\begin{lemma}  \label{lem:bounds_for_J_sigma(phi)}
    Let $F \in \cC^2$. There exists a random variable $\Psi$ such that for all $s \in [0, 1], x \in \R^d$, we have 
    \begin{equs}
        \| \blue{F (\phi^{s,x})}\|_{\cD^{2H_-}_B([s, 1])}& + \| \blue{J^{s,x}}\|_{\cD^{2H_-}_B([s, 1])} + \| \blue{(J^{s,x})^{-1}}\|_{\cD^{2H_-}_B([s, 1])}
        \\
        & + \| \blue{\nabla J^{s,x}}\|_{\cD^{2H_-}_B([s, 1])} + \|  \blue{\nabla(J^{s,x})^{-1}}\|_{\cD^{2H_-}_B([s, 1])}\leq \Psi.
    \end{equs}
    Moreover, for any $p \geq 1$, there exists a constant $C=C( H_-, p, \| \sigma\|_{\cC^4}, \| F\|_{\cC^2})$ such that $\| \Psi \|_{L_p(\Omega)} \leq C$. 
\end{lemma}

\begin{proof}
    By \eqref{eq:composition_estimate_controlled} and by the boundedness of $\sigma$,  there exists  $N=N(\|F\|_{\cC^2}, \| \sigma\|_{\cC^0})$ such that for all $s \in [0, 1], x \in \R^d$,  we have 
    \begin{equs}
    \, [\blue{F(\phi^{s,x})}]_{\mathcal{D}_B^{2H_-}([s, 1])}  &\leq N ( 1+ [B]^2_{\cC^{H_-}} ) ( [\blue{\phi^{s,x}}]_{\mathcal{D}_{B}^{2H_-}([s,1])}+ 1). 
\end{equs}
Moreover, it is straight forward that 
$\| F(\phi^{s,x})\|_{\cC^0} \leq \| F\|_{\cC^0}$ and $\| \nabla F (\phi^{s,x}) \sigma(\phi^{s,x})\|_{\cC^0} \leq \| F\|_{\cC^1} \| \sigma\|_{\cC^0}$, which combined with the above and  \cref{lem:bound_controlled_flow}  gives 
\begin{equs}
\|\blue{F(\phi^{s,x})}\|_{\mathcal{D}_B^{2H_-}([s, 1])}  &\leq N ( 1+ [B]^2_{\cC^{H_-}} ) ( [\blue{\phi^{s,x}}]_{\mathcal{D}_{B}^{2H_-}([s,1])}+ 1)
\\
& \leq N  (1+[\blue{B}]_{\cR^{H_-}})^{6+(3/H_-)}=:\Psi_0,     \label{eq:bound_F(phi)}
\end{equs}
for some $N=N(H_-,\|F\|_{\cC^2}, \| \sigma\|_{\cC^2})$. 
Next, notice that $J^{s,x}$ satisfies \eqref{eq:general_equation_Jacobian} with $F= \nabla\sigma$ and $\blue{M}=0$ and recall that $J^{s,x}_s= \mathbb{I}_d$. By  \cref{cor:V_by_M}, there exists a random variable  $\Psi_1$ whose $L_p$-norm depends only on $p, \|\sigma\|_{\cC^3}$, and $H_-$, such that for all $s \in [0,1], x \in \R^d$ we have 
     \begin{equs}
\|\blue{J^{s,x}}\|_{\mathcal{D}^{2H_-}_B([s,1])} & \leq \Psi_1.    \label{eq:bound_J_Psi1}
\end{equs}
   
   Next, recall that  $\nabla J^{s,x}$ satisfies \eqref{eq:general_equation_Jacobian} with $F= \nabla \sigma$ and $\blue{M}=\blue{\nabla^2 \sigma (\phi^{s,x}) J^{s,x}}$ (\cite[Theorem 8.15, p. 148]{Friz-Hairer}) and notice that $\nabla J^{s,x}_s=0$. Hence, by  \cref{cor:V_by_M} again, there exists a random variable  $\tilde{\Psi}_2$ whose $L_p$-norm depends only on $p, \|\sigma\|_{\cC^3}$, and $H_-$, such that for all $s \in [0,1], x \in \R^d$ we have 
     \begin{equs}
\|\blue{ \nabla J^{s,x}}\|_{\mathcal{D}^{2H_-}_B([s,1])} & \leq \tilde{\Psi}_2 \|\blue{\nabla^2 \sigma (\phi^{s,x}) J^{s,x}}\|_{\mathcal{D}^{2H_-}_B([s,1])}. 
\end{equs}
Further, by using \eqref{eq:controlled_product}, the result obtained in \eqref{eq:bound_F(phi)} with $F=\nabla^2\sigma$, and \eqref{eq:bound_J_Psi1}, for some random variable $\bar{\Psi}_0$  whose $L_p$-norm depends only on $p, \|\sigma\|_{\cC^4}$, and $H_-$, we get that 
 \begin{equs}
\|\blue{ \nabla J^{s,x}}\|_{\mathcal{D}^{2H_-}_B([s,1])} & \leq \tilde{\Psi}_2 \|\blue{\nabla^2 \sigma (\phi^{s,x}) J^{s,x}}\|_{\mathcal{D}^{2H_-}_B([s,1])}
\\
& \leq \tilde{\Psi}_2 (1+[\blue{B}]_{\cR^{H_-}})^2\|\blue{\nabla^2 \sigma (\phi^{s,x}) }\|_{\mathcal{D}^{2H_-}_B([s,1])}\| \blue{J^{s,x}}\|_{\mathcal{D}^{2H_-}_B([s,1])}
\\
& \leq \tilde{\Psi}_2 (1+[\blue{B}]_{\cR^{H_-}})^2\|\blue{\nabla^2 \sigma (\phi^{s,x}) }\|_{\mathcal{D}^{2H_-}_B([s,1])}\| \blue{J^{s,x}}\|_{\mathcal{D}^{2H_-}_B([s,1])}
\\
& \leq \tilde{\Psi}_2 (1+[\blue{B}]_{\cR^{H_-}})^2\bar{\Psi}_0 \Psi_1=: \Psi_2. 
\end{equs}
Finally, in exactly the same manner one shows the existence of random variables $\Psi_3, \Psi_4$ whose 
whose $L_p$-norm depends only on $p, \|\sigma\|_{\cC^4}$, and $H_-$, and for all $s \in [0,1], x \in \R^d$ we have  $\| \blue{ (J^{s,x})^{-1}}\|_{\cD^{2H_-}_B([s, 1])} \leq \Psi_3$ and  $\|  \blue{\nabla(J^{s,x})^{-1}}\|_{\cD^{2H_-}_B([s, 1])}\leq \Psi_4$. Consequently, setting $\Psi= \sum_{i=0}^4\Psi_i$ finishes the proof. 
\end{proof}

    \begin{lemma}\label{lem:flow-D-with-power}
 For all $(s, t) \in [0, 1]^2_{\leq}$ and $x \in \R^d$, we have that $\phi^{s, x}_t \in \mathbb{D}^\infty$. In addition, for any $ n \geq 1$, there exists a random variable $\Lambda_n$ such that for all  $(s, t) \in [0, 1]^2_{\leq}$, $x \in \R^d$, with probability one,  we have 
\begin{equs}
    \|D^n \phi^{s, x}_t\|_{\R^d \otimes \mathcal{H}^{\otimes n}} \leq \Lambda_n |t-s|^{H_-},
\end{equs}
 Moreover,  for any $p \geq 1$, there exists a  constant $C=C(n, p, \|\sigma\|_{\cC^{n+4}},H_-)$, such that  $\|\Lambda_n\|_{L_p(\Omega)}\leq C$.       
 Finally, one has the identity
 \begin{equ}\label{eq:malliavin-and-jacobi}
    D_r\phi^{s,x}_t=J^{r,x}_t\sigma(\phi^{s,x}_r)\bone_{r\in[s,t]}.
\end{equ}
    \end{lemma}

    \begin{proof}
         The fact that $\phi^{s, x}_t \in \mathbb{D}^\infty$ would follow from \cite[Theorem 1.2]{INAHAMA} and \eqref{eq:malliavin-and-jacobi} would follow from \cite[Proposition 2.26]{GOT} if $s$ would be zero.
         To see the general case, one can shift the whole problem as follows.
         One can easily see that the process $\tilde{\phi}_t^{u,x} =\phi^{s+u, x}_{s+t} $ , $t \geq u\geq 0$, is the flow of the rough equation with coefficient $\sigma$ starting at time zero, driven by the geometric rough path $\blue{\tilde{B}}$, given by the natural lift of the process $\tilde{B}_t:= B_{s+t}-B_s$, $t \geq 0$.
         For brevity we denote $\tilde \phi_t=\tilde \phi^{0,x}_t$.
         Then, by \cite[Theorem 1.2]{INAHAMA} one has that for any $l \in \mathbb{N}$ and 
$p \geq 1$, $\| \tilde{D}^l \tilde{\phi}_t \|_{L_p(\Omega; \tilde{\mathcal{H}}^{\otimes l})}< \infty $, where $\tilde{\mathcal{H}}$ is the analogue of the space $\mathcal{H}$ for the process $\tilde{B}$.
%i.e., the closure in $\mathcal{H}$ of the linear span of the indicator functions $r \mapsto \bone_{[s,u]}(s+r)$, for $u \in [s, 1]$.
It is not hard to see that the map $\Theta_s:\tilde\cH\to\cH$ given by 
\begin{equ}
    (\Theta_s\tilde h)(r)
    =\begin{cases}
        \tilde h(r-s)&\text{ if } r\geq s\\0 &\text{ if }r<s
    \end{cases}
\end{equ}
is an isometry. Indeed, denote by $\hat Q(u,v)=\E(B_{u+s}B_{v+s})$ and note that $\hat Q(du,dv)=\tilde Q(du,dv)$. Therefore if $\tilde h$ is a step function with support in $[0,T]$ for some $T>0$ then
by a simple change of variables we have 
\begin{equs}
    \|\Theta_s\tilde h\|_{\cH}^2&=\int_s^{s+T}\int_s^{s+T}\tilde h(u-s)\tilde h(v-s)Q(du,dv)
    \\&=\int_0^T\int_0^T\tilde h(u)\tilde h(v)\hat Q(du,dv)=\int_0^T\int_0^T\tilde h(u)\tilde h(v)\tilde Q(du,dv)=\|\tilde h\|_{\tilde\cH}^2.
\end{equs}
It is then clear that if a random variable $X$ is measurable with respect to the $\sigma$-algebra generated by the increments of $\tilde{B}$  it is also measurable with respect to the $\sigma$-algebra generated by the increments of $B$ and for any $l \in \mathbb{N}$ we have that 
 $D^l X=\Theta_s^{\otimes l} \tilde{D}^l X$, whenever the right hand side is meaningful.  Hence, indeed we have $\phi^{s, x}_t \in \mathbb{D}^\infty$. Similarly,  from $D\phi^{s,x}_t=\Theta_s\tilde D\phi^{s,x}_t$ we immediately get the $r<s$ case of \eqref{eq:malliavin-and-jacobi}, the $r>t$ case is trivial, while for $r\in[t,s]$ we invoke \cite[Proposition 2.26]{GOT} for the process $\tilde B$ to get
 \begin{equs}
    D_r\phi^{s,x}_t=\tilde D_{r-s}\tilde{\phi}^{0,x}_{t-s}=\tilde J^{r-s,x}_{t-s}\sigma(\tilde\phi^{0,x}_{r-s})\bone_{r-s\leq t-s}=J^{r,x}_t\sigma(\phi^{s,x}_r)\bone_{r\in[s,t]},
\end{equs}
where $\tilde J$ is the Jacobian of the flow $\tilde \phi$.

Next, the desired estimate can be shown again by  following the arguments of \cite{INAHAMA}: Let us consider an independent copy of $B$, denoted by $\hat B$, and (with a slight abuse of notation) consider $\blue{W}= (W, \mathbb{W})$, the natural lift of $W=(B, \hat{B})$ to a geometric rough path (here $W$ should not be confused with the underlying Wiener process). Then, if $\blue{X}=(X, X') \in \mathcal{D}^{2H_- }_{B}([0, 1])$ is controlled by $B$, it can also be seen as a path controlled by $W$, namely, with a slight abuse of notation, $\blue{X}= (X, (X',0)) \in  \mathcal{D}^{2H_- }_{W}([0, 1])$, and similarly for $\blue{Y}= (Y, Y') \in \mathcal{D}^{2H_- }_{\hat{B}}([0, 1])$.  By this consideration, if $\blue{Z} \in  \mathcal{D}^{2H_- }_{W}([0, 1])$, $\blue{X} \in \mathcal{D}^{2H_- }_{B}([0, 1])$, $ Y  \in \mathcal{D}^{2H_- }_{\hat{B}}([0, 1])$, then integrals of the form $\int \blue{f(X,Y, Z) }\, d\blue{B}$  or $\int \blue{f(X,Y, Z)} \, d\blue{\hat{B}}$ can be defined as rough integrals against $\blue{W}$.  We then set 
\begin{equs}
    \Xi^{(1),s ,x }_t &=   J^{0, x}_t \int_s^t  \blue{(J^{0, x}_r)^{-1} \sigma (\phi^{s, x}_r )} \, d\blue{\hat{B}_r},
   \\
    \Xi^{(2), s, x}_t &=  J^{0, x}_t\int_s^t  \blue{(J^{0, x}_r)^{-1} \nabla^2\sigma (\phi^{s, x}_r) (\Xi^{(1), s, x}_r,\Xi^{(1), s, x}_r, \, \cdot )} d\blue{\hat{B}_r} 
    \\
    & + 2 J^{0, x}_t\int_s^t \blue{(J^{0, x}_r)^{-1} \nabla \sigma (\phi^{s, x}_r)(\Xi^{(1), s, x}_r, \cdot ) }d\blue{\hat{B}_r} 
\end{equs}
and inductively, for $n \geq 3$
\begin{equs}
 \Xi^{(n), s, x}_t &=  J^{0, x}_t\int_s^t  \blue{(J^{0, x}_r)^{-1}\sum_{l=2}^n  \sum_{i_1+...+i_l=n} C_{i_1, ..., i_l}\nabla^l\sigma (\phi^{s, x}_r) (\Xi^{(i_1), s, x}_r,..., \Xi^{(i_l), s, x}_r, \cdot )}  d\blue{B_r}
 \\
 & + J^{0, x}_t\int_s^t \blue{(J^{0, x}_r)^{-1}\sum_{l=1}^{n-1} \sum_{i_1+...+i_l=n-1} \nabla^l\sigma (\phi^{s, x}_r)(\Xi^{(i_1), s, x}_r, ...,\Xi^{(i_l), s, x}_r, \cdot   )} \, d \blue{\hat{B}_r}.
\end{equs}
Then, it is shown in \cite{INAHAMA} (Proposition 3.3, Lemma 4.8, and Section 4.3)   that for any $n \geq 1$, there exists a constant $C_0=C_0(n)$, such that almost surely, we have 
\begin{equs}        \label{eq:bound_Malliavin_Xi}
    \|D^n \phi^{s,x}_t\|_{\R^d \otimes \cH^{\otimes n }} \leq  C_0 \big( \hat{\E} |\Xi^{(n), s, x}_t|^2 \big)^{1/2},
\end{equs}
where $\hat{\E}$ denotes the expectation with respect to $\hat{B}$. 
We derive estimates for the right hand side. We start for $n=1$. 
By 
\eqref{eq:controlled_product} and \eqref{eq:boundedness_rough_integration}, we get 
\begin{equs}
    \|  \blue{\Xi^{(1),s ,x }}\|_{\cD^{2H_-}_W([s, 1])} &\lesssim   \| \blue{J^{0, x}}\|_{\cD^{2H_-}_W([s, 1])} \Big\| \int_s^\cdot  \blue{ (J^{0, x}_r)^{-1} \sigma (\phi^{s, x}_r )} \, d\blue{\hat{B}_r} \Big\| _{\cD^{2H_-}_W([s, 1])} (1+[ \blue{W}]_{\cR^{H_-}})^2
    \\
& \lesssim   \| \blue{J^{0, x}}\|_{\cD^{2H_-}_W([s, 1])}  (\|\blue{ (J^{0, x})^{-1} \sigma( \phi^{s,x})}\|_{\cD^{2H_-}_B([s, 1])} +1) (1+[ \blue{W}]_{\cR^{H_-}})^3. 
\end{equs}
Therefore, by \eqref{eq:controlled_product} and \cref{lem:bounds_for_J_sigma(phi)} we get 
\begin{equs}
   \|  \blue{ \Xi^{(1),s ,x }}\|_{\cD^{2H_-}_W([s, 1])}  \leq P_1( \Psi_0, [ \blue{W}]_{\cR^{H_-}}),   \label{eq:estimate_Xi1}
\end{equs}
for some polynomial $P_1$, where $\Psi_0$ is the random variable from the conclusion of \cref{lem:bounds_for_J_sigma(phi)} corresponding to the choice $F= \sigma$. By this bound, a repetition of the above argument gives 
\begin{equs}
   \|   \blue{\Xi^{(2),s ,x }}\|_{\cD^{2H_-}_W([s, 1])}  \leq P_2( \Psi_0, \Psi_1, \Psi_2,  [ \blue{W}]_{\cR^{H_-}}), 
\end{equs}
for some polynomial $P_2$, where $\Psi_1$, $\Psi_2$ are the random variable from the conclusion of \cref{lem:bounds_for_J_sigma(phi)} corresponding to the choices $F= \nabla \sigma$ and $F=\nabla^2 \sigma$, respectively. By induction, one gets 
\begin{equs}
   \|   \blue{\Xi^{(n),s ,x }}\|_{\cD^{2H_-}_W([s, 1])}  \leq P_n( \Psi_0, \Psi_1, \Psi_2, ..., \Psi_n,  [ \blue{W}]_{\cR^{H_-}}),      \label{eq:estimate_Xi_n}
\end{equs}
for some polynomial $P_n$,  where $\Psi_l$ is the random variable from the conclusion of \cref{lem:bounds_for_J_sigma(phi)} corresponding to the choice $F= \nabla^l \sigma$, $l=0, ..., n$.
Then, by \eqref{eq:holder_bound_by_rough} it follows that 
\begin{equs}
    \,   [\Xi^{(n),s ,x } ]_{\cC^{H_-}([s, 1])} &  \leq  \|  \blue{ \Xi^{(n),s ,x }}\|_{\cD^{2H_-}_W([s, 1])}  (1+ [ \blue{W}]_{\cR^{H_-}})
    \\
   &  \leq P_n( \Psi_1, \Psi_2, ..., \Psi_n,  [ \blue{W}]_{\cR^{H_-}})  (1+ [ \blue{W}]_{\cR^{H_-}}) =: \Lambda_n. 
\end{equs}
The claim follows from the above inequality combined with \eqref{eq:bound_Malliavin_Xi} and  the fact that $\Xi^{(n), s, x}_s=0$. 
    \end{proof}

    \begin{lemma}\label{lem:Jacobi-Malliavin-Derivatives}
For all $(s, t) \in [0, 1]^2_{\leq}$, $x \in \R^d$, we have that $J^{s, x}_t \in \mathbb{D}^{m,p}$ for all $m \in \mathbb{N}$ and $p \geq 1$.   In addition, for any $ n   \in  \mathbb{N}$, there exists a random variable $\Theta_n$ such that for all  $(s, t) \in [0, 1]^2_{\leq}$,  $x \in \R^d$, with probability one,  we have 
\begin{equs}
    \|D^n J^{s, x}_t\|_{\R^{d\times d} \otimes \mathcal{H}^{\otimes n}} \leq \Theta_n. 
\end{equs}
 Moreover, for any $p \geq 1$, there exists a constant $C=C(d, n, p, \|\sigma\|_{\cC^{n+5}}, H_-)$ such that  $\|\Theta_n\|_{L_p(\Omega)}\leq C$.
    \end{lemma}

    \begin{proof}
    Let us fix $m \in \mathbb{N}$ and $p \geq 1$. It follows again by \cite{INAHAMA} (Proposition 3.3, Lemma 4.8, and Section 4.3)  that for $n \in \mathbb{N}$ there exists a constant $C_0(n)$ such that for all   $x, y \in \R^d$,  with probability one, we have
    \begin{equs}      \label{eq:inahama_Xi}
    \|D^n \phi^{s,x}_t-D^n \phi^{s,y}_t \|_{\R^d \otimes \cH^{\otimes n }}  \leq C_0 \big( \hat{\E} |  \Xi^{(n), s, x}_t - \Xi^{(n), s, y}_t|^2 \big)^{1/2},
    \end{equs}
    where $\Xi^{(n)}$ are as in the proof of  \cref{lem:flow-D-with-power}. 
    Next, we claim that there exists a polynomial $P_n$ such that for all $(s,t) \in [0,1]^2_{\leq}$, $x,y \in \R^d$, with probability one, we have 
    \begin{equs}
      |  \Xi^{(n), s, x}_t - \Xi^{(n), s, y}_t|  \leq  |x-y| P_n(\Psi_0,  \Psi_1, \Psi_2, ..., \Psi_{n+1},   [ \blue{W}]_{\cR^{H_-}}),   \label{eq:Lip_Xi_initial}
    \end{equs}
      where $\Psi_l$, $l \in \{0, 1, ..., n+1\}$, are the random variables from the conclusion of \cref{lem:bounds_for_J_sigma(phi)} corresponding to the choice $F= \nabla^l \sigma$. Indeed, for $n=1$, similarly to the derivation of \eqref{eq:estimate_Xi1}, we have 
   \begin{equs}
& \,  \| \blue{\Xi^{(1), s, x}- \Xi^{(1), s, y}} \|_{\cD^{2H_-}_W([s, 1])}
   \\
   = & \, \Big\| \blue{ J^{0, x} \int_s^\cdot  (J^{0, x}_r)^{-1} \sigma (\phi^{s, x}_r ) \, d\hat{B}_r -J^{0, y} \int_s^\cdot  (J^{0, y}_r)^{-1} \sigma (\phi^{s, y}_r ) \, d\hat{B}_r }\Big\|_{\cD^{2H_-}_W([s, 1])}
   \\
    \lesssim  \, &   (1+[ \blue{W}]_{\cR^{H_-}})^2 \Bigg( \| \blue{ J^{0, x}-J^{0, y}}\|_{\cD^{2H_-}_W([s, 1])}  \Big\| \blue{\int_s^\cdot  (J^{0, y}_r)^{-1} \sigma (\phi^{s, y}_r ) \, d\hat{B}_r} \Big\|_{\cD^{2H_-}_W([s, 1])} 
    \\
    & \, + \| \blue{J^{0, x} }\|_{\cD^{2H_-}_W([s, 1])}\Big\| \blue{\int_s^\cdot  \Big( (J^{0, x}_r)^{-1} \sigma (\phi^{s, x}_r )- (J^{0, y}_r)^{-1} \sigma (\phi^{s, y}_r )\Big)  \, d\hat{B}_r }\Big\|_{\cD^{2H_-}_W([s, 1])} \Bigg).   \label{eq:towrads_Lip_Xi_1}
   \end{equs} 
   Then, by the fundamental theorem of calculus, it follows that 
   \begin{equs}
\| \blue{ J^{0, x}-J^{0, y}}\|_{\cD^{2H_-}_W([s, 1])} \leq   |x-y| \int_0^1 \|\blue{ \nabla J^{0, \theta x+(1-\theta y)}} \|_{\cD^{2H_-}_W([s, 1])} \, d \theta \leq |x-y| \Psi_0,
   \end{equs}
   where for the last inequality we have used  \cref{lem:bounds_for_J_sigma(phi)}.  By \eqref{eq:boundedness_rough_integration}, \eqref{eq:controlled_product}, and  \cref{lem:bounds_for_J_sigma(phi)}, we get 
  \begin{equs}
      \Big\| \blue{\int_s^\cdot  (J^{0, y}_r)^{-1} \sigma (\phi^{s, y}_r ) \, d\hat{B}_r} \Big\|_{\cD^{2H_-}_W([s, 1])} \leq G_1(\Psi_0, [ \blue{W}]_{\cR^{H_-}}),
  \end{equs}
  for some polynomial $G_1$. Again by \eqref{eq:boundedness_rough_integration}, the fundamental theorem of calculus,  \eqref{eq:controlled_product}, and  \cref{lem:bounds_for_J_sigma(phi)}, one obtains 
  \begin{equs}
      \Big\| \blue{\int_s^\cdot  \Big( (J^{0, x}_r)^{-1} \sigma (\phi^{s, x}_r )- (J^{0, y}_r)^{-1} \sigma (\phi^{s, y}_r )\Big)  \, d\hat{B}_r }\Big\|_{\cD^{2H_-}_W([s, 1])} \leq |x-y| G_2(\Psi_0, \Psi_1, [ \blue{W}]_{\cR^{H_-}}),
  \end{equs}
  for some polynomial $G_2$. Hence, by plugging the three estimates above back into \eqref{eq:towrads_Lip_Xi_1} and using again  \cref{lem:bounds_for_J_sigma(phi)}, we get 
  \begin{equs}
      \| \blue{\Xi^{(1), s, x}- \Xi^{(1), s, y}} \|_{\cD^{2H_-}_W([s, 1])} \leq  |x-y| P_1( \Psi_0, \Psi_1,  [ \blue{W}]_{\cR^{H_-}}), 
  \end{equs}
  for some polynomial $P_1$. We can now go back to the definition of $\blue{\Xi^{(2), s, x}}$, and by the above bound combined  with \eqref{eq:estimate_Xi1}, a repetition of the above argument one gets 
  \begin{equ}
       \| \blue{\Xi^{(2), s, x}- \Xi^{(2), s, y}} \|_{\cD^{2H_-}_W([s, 1])} \leq  |x-y| P_2( \Psi_0, \Psi_1, \Psi_2, \Psi_3  [ \blue{W}]_{\cR^{H_-}}). 
  \end{equ}
  From here, with the help of  \cref{lem:bounds_for_J_sigma(phi)} and estimates \eqref{eq:estimate_Xi_n}, by using induction one obtains 
  \begin{equs}
       \| \blue{\Xi^{(n), s, x}- \Xi^{(n), s, y}} \|_{\cD^{2H_-}_W([s, 1])} \leq |x-y| P_n(\Psi_0,  \Psi_1, \Psi_2, ..., \Psi_{n+1},   [ \blue{W}]_{\cR^{H_-}}),
  \end{equs}
  which in particular implies \eqref{eq:Lip_Xi_initial}. 
  From this and \eqref{eq:inahama_Xi}, we get that for any $m  \in \mathbb{N}$ and $p \geq 1$,  there exists a constant  there exists a constant $C_1=C_1(m, p, \| \sigma\|_{\mathcal{C}^{m+1}}, H_- )$ such that for all $(s, t) \in [0, 1]^2_{\leq}$, $x, y \in \R^d$, we have 
  \begin{equs}\label{est:stability-phi}
     \| \phi^{s,x}_t- \phi^{s,y}_t\|_{\mathbb{D}^{m,p}} \leq C |x-y|. 
  \end{equs}
  By choosing $y=x+\ell^{-1}e_i$, with $e_i$ being the standard unit vector in the $i$th direction in $\R^d$, we see that the sequence  $\psi_\ell := \ell(\phi^{s,x}_t- \phi^{s,x+\ell^{-1}e_i}_t)$ is bounded in $\mathbb{D}^{m,p}$ by $C$. Since $\mathbb{D}^{m,p}$ is reflexive, there is  subsequence $\psi_{\ell_k}$ which converges weakly to some $\psi\in\mathbb{D}^{m,p}$.  On the other hand, by definition we know that $\psi_{\ell_k} \to \D_i\phi^{s,x}_t$ almost surely,  and by \cref{lem:bounds_for_J_sigma(phi)}, also in $L_p$. Hence, we conclude that $\D_i\phi^{s,x}_t = \psi  \in \mathbb{D}^{m,p}$. Moreover, there exists a sequence $\bar{\psi_k}$ consisting of convex combinations of $\psi_{n_k}$ such that $\bar{\psi}_k \to \D_i\phi^{s,x}_t$ strongly in $\mathbb{D}^{m,p}$. In particular, for any $n \in \{1, ..., m\}$ we have 
  $ D^{n} \bar{\psi}_k\to  D^n \D_i\phi^{s,x}_t$ in probability. By \eqref{eq:inahama_Xi} and \eqref{eq:Lip_Xi_initial}, it follows that with probability one, for  all $k \in \mathbb{N}$, $(s,t) \in [0,1]^2_{\leq}$, $x \in \R^d$,  we have %\chengcheng{Here $\hat \E$ again} \KD{Yes, it is corrrect as it is. One integrates with respect to $\hat{B}$ only. Agree}
  \begin{equs}
      \| D^{n} \bar{\psi}_k \|_{\R^d \otimes \cH^{\otimes n }}  \leq  C_0 \big( \hat{\E} |P_n(\Psi_0,  \Psi_1, \Psi_2, ..., \Psi_{n+1},   [ \blue{W}]_{\cR^{H_-}})|^2 \big)^{1/2},
  \end{equs}
  and upon passing to a further subsequence which converges almost surely and using Fatou's lemma, with probability one, we have 
  \begin{equs}
      \|D^n \D_i\phi^{s,x}_t \|_{\R^d \otimes \cH^{\otimes n }}  \leq  C_0 \big( \hat{\E} |P_n(\Psi_0,  \Psi_1, \Psi_2, ..., \Psi_{n+1},   [ \blue{W}]_{\cR^{H_-}})|^2 \big)^{1/2}.
  \end{equs}
  We can now define $\Theta_n : = d C_0 \big( \hat{\E} |P_n(\Psi_0,  \Psi_1, \Psi_2, ..., \Psi_{n+1},   [ \blue{W}]_{\cR^{H_-}})|^2 \big)^{1/2}$, and the claim follows. 
\end{proof}

\begin{lemma}\label{lem:continuous-version}
    Let $s \in [0,1]$, $k\in\N$ and $f \in \cC^1(\R^{kd})$. Then, for each $x_1,\ldots,x_k \in \R^d, t_1,\ldots,t_k \in [s,1]$, with probability one, we have
    \begin{equs}                       \label{eq:swap_filtrations}
        \E_s f(\phi^{s,x_1}_{t_1},\ldots,\phi^{s,x_k}_{t_k})= \E_s^Bf(\phi^{s,x_1}_{t_1},\ldots,\phi^{s,x_k}_{t_k}). 
    \end{equs}
    Moreover, the random fields in \eqref{eq:swap_filtrations} have continuous modifications in the variables $x_1,\ldots,x_k$ and $t_1,\ldots,t_k$.
    %, the map $(\omega, x_1,\ldots,x_k, t_1,\ldots,t_k) \mapsto  \E_s f(\phi^{s,x_1}_{t_1},\ldots,\phi^{s,x_k}_{t_k})$ has a continuous modification in the variables $x_1,\ldots,x_k,t_1,\ldots,t_k$.
    %and $(\omega, x, t) \mapsto  \E_s^B f(\phi^{s,x}_t)$ have versions continuous in $(x,t)\in \R^d \times [s,1]$. 
\end{lemma}

\begin{proof}
  Let $\beta \in (0,1)$ be sufficiently close to $1$ so that $\beta-1+H+H_-  > 2H_-$. Let us consider the space $\mathcal{C}^\beta_{H-1}([s, 1])$ of all functions $w:[s,1] \to \R^{d_0}$ such that 
  \begin{equs}\label{eq:weighted-Holder}
    \,   \|w\|_{\mathcal{C}^\beta_{H-1}([s, 1])}:=|w(s)|+ \sup_{(u, t) \in (s, 1]^2_{\leq}} \frac{|w_t-w_u|}{|u-s|^{H-1}|t-u|^{\beta}} < \infty. 
  \end{equs}
  Also,  let $\mathcal{C}^{\beta+}_{H-1}([s, 1])$
 be the closure of $\cC^1([s,1])$ in $\mathcal{C}^\beta_{H-1}([s, 1])$. 
 Then, let us consider the map $T:\mathcal{R}^{H_-}_{\geo}([s,1])\times\mathcal{C}^{\beta+}_{H-1}([s, 1]) \to  \mathcal{R}^{H_-}([s,1])$ defined as follows: for $\blue{g}=(g, \mathbb{g}) \in \mathcal{R}^{H_-}_{\geo}([s,1])$ and $w \in \mathcal{C}^{\beta+}_{H-1}([s, 1])$, set 
 \begin{equs}
     T(\blue{g}, w)=\Big( g+w, \mathbb{g}+ \int g \otimes dw +\int w \otimes dg+ \int w\otimes \, dw \Big),
 \end{equs}
 where in the above, by virtue of \cref{lem:Young_with_blow_up}, the integrals  $\int_s^t g_{s,r} \otimes dw_r$ and  $\int_s^t w_{s,r} \otimes dw_r$ are canonically defined as (singular) Young integrals and as functions of $(s,t)$ belong to $\cC^{2H_-}_2$, while the ``integral" $\int w \otimes dg$ is defined as follows 
 \begin{equs}
     \int_s^t w_{s, r} \otimes dg_r:=w_{s,t}\otimes g_{s,t}- \Big(\int_s^t g_{s, r}\otimes dw_r \Big)^*.
 \end{equs}
 Moreover, again by \cref{lem:Young_with_blow_up}, it is easy to see that $T$ is continuous. 

Next, recall that we have the decomposition 
\begin{equs}                
    B_t = \int_0^s K(t, r) \, dW_r+ \int_s^t K(t,r) \, dW_r =:\overline{B}_t+ \widetilde{B}_t,
\end{equs}
where we dropped the $s$ dependence on the right hand side for notational convenience. 
Then clearly we have that the process $\overline{B}_t$ is $\cF_s$-measurable, while $\widetilde{B}_t$ is independent of $\cF_s$. Moreover, the same properties hold for $\cF^B_s$ in place of $\cF_s$, since the filtration generated by $B$ coincides with the filtration generated by $W$. 

It is known that $\overline{B} \in \mathcal{C}^{\beta+}_{H-1}([s, 1])$ with probability one (it follows, for example, by \cite[(5.46) p. 295]{NUalart_Malliavin} and Kolmogorov's continuity criterion). Moreover, since $\overline{B}_t$ is $\cF_s$-measurable for each $t \in [s, 1]$ and the Borel $\sigma$-algebra on $C([s,1])$ coincides with the cylindrical one, it follows that $\overline{B}$ is an $\cF_s$-measurable $C([s,1])$-valued random variable. By the Lusin-Suslin theorem, since the embedding $\mathcal{C}^{\beta+}_{H-1}([s, 1]) \subset C([s,1])$ is continuous, it follows that $\overline{B}$ is an $\cF_s$-measurable $\mathcal{C}^{\beta+}_{H-1}([s, 1])$-valued random variable.

Next, we lift $\widetilde{B}$ to a rough path $ \blue{\widetilde{B}}=(\widetilde{B}, \widetilde{\mathbb{B}})$ by setting 
\begin{equs}                  \label{eq:tilde_second_order}
     \widetilde{\mathbb{B}}_{u,t}=\mathbb{B}_{u,t}-\int_u^t\overline{B}_{u,r} \otimes d \overline{B}_r- \int_u^t \overline{B}_{u,r} \otimes d \widetilde{B}_r-\int_u^t \widetilde{B}_{u,r} \otimes d\overline{B}_r.
\end{equs}
Then, we have that $\blue{\widetilde{B}} \in  \mathcal{R}^{H_-}_{\geo}([s,1])$ with probability one (see, e.g., \cite[Lemma 2.12]{Avi_Toyomu}). We claim that $\blue{\widetilde{B}}$ is a $\mathcal{R}^{H_-}_{\geo}([s,1])$-valued random variable that is independent of $\mathcal{F}_s$.  To see this, let us consider $\widetilde{B}^n$ and $\overline{B}^n$, given by linear interpolation of values $\widetilde{B}^n_t$ and $\overline{B}^n_t$, respectively, for $t\in \{s+k(1-s)/n : k=1,\ldots ,n\}$. Further, let us set $B^n=\overline{B}^n+\widetilde{B}^n$. Then, for each $(u, t) \in [s,1]^2$, we have 
\begin{equs}
 \,    & \int_u^t\widetilde{B}^n_{u,r} \otimes d \widetilde{B}^n_r
    \\
    = \, & \int_u^t B^n_{u,r} \, dB^n_r -\int_u^t\overline{B}
^n_{u,r} \otimes d \overline{B}^n_r- \int_u^t \overline{B}^n_{u,r} \otimes d \widetilde{B}^n_r-\int_u^t \widetilde{B}^n_{u,r} \otimes d\overline{B}^n_r. 
\end{equs}
The left hand side is clearly independent of $\cF_s$. On the other hand, each term on the right hand side converges in probability as $n \to \infty$ to the corresponding term at the right hand side of \eqref{eq:tilde_second_order}. This shows that for $(u, t) \in [s,1]^2$, $\widetilde{\mathbb{B}}_{u,t}$ is independent of $\cF_s$. Moreover, since the map $(u,t) \mapsto \widetilde{\mathbb{B}}_{u,t}$ is continuous (can be shown by Chen's relation combined with the fact that $\widetilde{B} \in \cC^{H_-}$, $\widetilde{\mathbb{B}}\in \cC^{2H_-}_2$), it follows   that the pair $(\widetilde{B}, \widetilde{\mathbb{B}})$ is a $C([s,1]) \times C([s,1]^2)$ valued random variable, independent of $\cF_s$. Moreover, the identity map from  $\mathcal{R}^{H_-}_{\geo}([s,1])$ to $C([s,1]) \times C([s,1]^2)$ is continuous, and since these are Polish spaces, it follows by the Lusin-Suslin theorem that $\blue{\widetilde{B}}$ is a $ \mathcal{R}^{H_-}_{\geo}([s,1])$-valued random variable, independent of $\cF_s$.

Next notice that $\blue{B}=T(\blue{\widetilde{B}}, \overline{B})$. Hence, we have that $(\phi^{s,x_1}_{t_1},\ldots,\phi^{s,x_k}_{t_k})=\mathcal{I}(T(\blue{\widetilde{B}}, \overline{B}))$ with a map $\mathcal{I}$, where by the continuity of the It\^o-Lyons map $\mathcal{I}$ is also continuous.  Consequently, by using \cite[Proposition 2.2.2]{Rockner_Prevot} we have 
\begin{equs}
    \E_s f(\phi^{s,x_1}_{t_1},\ldots,\phi^{s,x_k}_{t_k})= \E_s f\big(\mathcal{I}(T(\blue{\widetilde{B}}, \overline{B}))\big)= \Phi(\overline{B}), 
\end{equs}
where $\Phi(\zeta):= \E f\big(\mathcal{I}(T(\blue{\widetilde{B}}, \zeta))\big)$  for $\zeta \in \mathcal{C}^{\beta+}_{H-1}([s, 1])$. Notice that the filtration $(\cF^B)_{t\in [0,1]}$ could have been chosen as $(\cF_t)_{t \in [0,1]}$ from the beginning, so that we also get 
\begin{equs}
    \E^B_sf(\phi^{s,x_1}_{t_1},\ldots,\phi^{s,x_k}_{t_k})=  \Phi(\overline{B}),
\end{equs}
which shows \eqref{eq:swap_filtrations}. The existence of continuous modifications follows easily by Kolmogorov's continuity criterion. 
\end{proof}

\section{Emerging driving signal of the linearised equation}
The goal of this section is to construct (a truncated version of) the object $L$ anticipated in the introduction \eqref{eq:L-early}. Note that in \eqref{eq:L-early} the distribution $\nabla b$ is integrated along the process $\theta X+(1-\theta)Y$, and a priori we do not know that such a convex combination of the nondegenerate processes (i.e. its Malliavin matrix is invertible) $X$ and $Y$ is also nondegenerate. One can only guarantee this as long as the two solutions do not separate too much. This will be formulated by an appropriate cutoff function.

Given $\sigma$ satisfying \cref{asn:sigma}, let
\begin{equ}\label{eq:rho def}
\rho=\frac{\sqrt{\lambda}}{10^5\|\sigma\|_{\cC^1}}.
\end{equ}
One then has the property that if $x_1,x_2,x_3,x_4\in\R^d$ are such that $|x_1-x_2|,|x_1-x_3|,|x_1-x_4|\leq 24\rho$, then for any $\theta_1,\ldots,\theta_4\in[0,1]$ such that $\theta_1+\cdots+\theta_4=1$, one has 
\begin{equs}
\Big(&\sum_{i=1}^4 \theta_i\sigma(x_i)\Big)\Big(\sum_{i=1}^4 \theta_i\sigma(x_i)\Big)^*
\\
&=\Big(\sigma(x_1)-\sum_{i=2}^4\theta_i\big(\sigma(x_1)-\sigma(x_i)\big)\Big)
\Big(\sigma(x_1)-\sum_{i=2}^4\theta_i\big(\sigma(x_1)-\sigma(x_i)\big)\Big)^*
\\&\succeq \sigma(x_1)\sigma(x_1)^*-2\sum_{i=2}^4\big(\sigma(x_1)-\sigma(x_i)\big)\big(\sigma(x_1)-\sigma(x_i)\big)^*\succeq\frac{1}{2}\,\,\lambda \mathbb{I}_{d}.\label{eq:convex-combination-nondegenerate}
\end{equs}
The seemingly arbitrary factor $24$ is simply the constant that comes up in the relevant part of the proof.
Furthermore,
set $\cut_d$ to be the set of smooth functions $\chi:\R^d\to\R$ that are identically $1$ on the ball of radius $1/2$ around the origin and identically $0$ on the complement of the ball of radius $2$ around the origin. For $\chi\in\cut_d$ and $a>0$ we then denote $\chi_a(x)=\chi(a^{-1}x)$. To reduce the many parameters to keep in mind, the reader may take $\alpha-1$ in place of $\gamma$ below as the main example.
\begin{theorem}\label{thm:L}
Let \cref{asn:b} and \cref{asn:sigma} hold. Let $\chi\in\cut_d$,  $p\in[1,\infty)$, $\theta\in[0,1]$,  and let $X$ and $Y$ be two solutions to \eqref{eq:main}. For $f\in \cC^1$ define the process
\begin{equ}\label{eq:J}
J_t=\int_0^t\chi_{\rho}(X_r-Y_r)f\big(\theta X_r+(1-\theta)Y_r\big)\,dr.
\end{equ}
Then for any $\gamma\in(-1/(2H),0)$ 
there exists a constant $C$ such that for all $(s,t)\in[0,1]^2$ one has
\begin{equ}\label{eq:L-bound}
\|J_{t}-J_s\|_{L^p(\Omega)}\leq C\|f\|_{\cC^{\gamma}}|t-s|^{1+\gamma H}
\end{equ}
and $C=C(\alpha,\gamma,p,d, \chi, \lambda,H,\|\sigma\|_{\cC^5},C^X_{D, p}, C^X_{S, 2p}, C^Y_{D,p}, C^Y_{S, 2p})$.
\end{theorem}
The theorem is proved after a preparatory lemma. The following corollary is immediate. 
\begin{corollary}\label{cor:L}
    Let \cref{asn:b} and \cref{asn:sigma} hold. Let  $\chi\in\cut_d$, $p\in[1,\infty)$, and let $X$ and $Y$ be two solutions to \eqref{eq:main}. Then the map $f\mapsto\cJ f$ defined for $f\in\cC^2(\R^d; \R^d)$ by
    \begin{equ}
        (\cJ f )_t=\int_0^1\int_0^t\chi_{\rho}(X_r-Y_r)\nabla f\big(\theta X_r+(1-\theta)Y_r\big)\,dr\,d\theta
    \end{equ}
    extends as a continuous linear map from $\cC^{\alpha+}(\R^d; \R^d)$ to $\cC^{1+(\alpha-1)H}([0,1]; L_p(\Omega; \R^{d\times d}))$, whose norm depends only on $\alpha,p,d,\lambda,H,\chi,\|\sigma\|_{\cC^5},C^X_{D, p}, C^X_{S, 2p}, C^Y_{D,p}, C^Y_{S, 2p}$.
    \end{corollary}
    \begin{proof}
        Apply \eqref{eq:L-bound} with $\nabla f$ in place of $f$ 
        and then integrate in $\theta$.
    \end{proof}

    The most important output of the present section is worth a separate notation: in the setting of \cref{cor:L} we define
    \begin{equ}
L:=\cJ b.
    \end{equ}
which is a truncated version of the object anticipated in the introduction \eqref{eq:L-early}.

To prepare for the proof of \cref{thm:L} we state a version of the integration by parts formula.
Although this is a classical and central tool in Malliavin calculus, there are two aspects that complicate matters in the present context. First, for stochastic sewing we shall need \emph{conditional} bounds, therefore the formalism of \cref{sec:partial-malliavin}  is needed and used.
Second, since we do not know if the random variable  $\cX$ below (roughly corresponding to $\theta X_r+(1-\theta)Y_r$ in \eqref{eq:J}) is nondegenerate,
the possible degeneracy is handled by the cutoff function.

\begin{lemma}\label{lem:IBP-with-cutoff}
Let $s\in[0,1]$, $a>0$, $d,d_1, k\in\N$, $\chi'\in\cut_{d_1}$,
%$p\in[1,\infty)$, 
and $h\in\cC^k(\R^d)$. 
%Let further  $g:\R^{d_1}\to\R$ be a radially symmetric smooth function with compact support that is identically $1$ on the ball of radius $1$ around the origin and is identically $0$ on the complement of the ball of radius $2$ around the origin. Set $g_a(x)=g(a^{-1}x)$.
%Define $p^*(k)=....${\color{red} put here the largest $p$ that appears after $k$ integration by parts in the quantity below}
There exists $p^*=p^*(k)\in[1,\infty)$ and $k^*=k^*(k)\in\N$ such that the following hold.
Let $\cX$, $\cY$ be $\R^d$-valued random variables, $\cZ$ be an $\R^{d_1}$-valued random variable, all three belonging to $\bD_s^{k^*,p^*}$, and $\Gamma,\mu$ be  $(0,\infty)$-valued random variables such that $\mu$ is $\cF_s$-measurable and one has
$\E_s^B\Gamma^{-p^*}\leq \mu^{p^*}$ almost surely
%$\|\Gamma^{-1}\|_{L^{p^*}|\cF_s}\leq \mu$
and on the event $\{|\cZ|\leq 2^{k}a\}$ one has 
\begin{equ}\label{eq:cond-Mall-nondegenerate}
    \cM_s:=\big(\scal{D_{\cH^\perp_s} \cX^i,D_{\cH^\perp_s} \cX^j}_{\CH_s^{\perp}}\big)_{i,j=1,\ldots,d}\succeq \Gamma\, \id.
\end{equ}
%be random variables with value satisfy the following:
%\begin{enumerate}[(i)]
%\item  $X,Y\in\mathbb{R}^d$,  $Z\in\R^{d_1}$ and  $\Gamma\in\R$ are all $\mathcal{F}_t$-measurable,
%    \item On the event that $|Z|\leq 2^{k-1}a$, the Malliavin matrix $\mathcal{M}^{X}$ of $X$ satisfies
%    is bounded from below by  $\Gamma\id$, 
    %all finite order moments of $\Gamma^{-1}$ are bounded by $C$.     
    %\item There exists a random variable $C<\infty$ such that for any $\bar p\leq p^*(k)$, $\|\Gamma^{-1}\|_{L^p|\cF_s}\leq C$.
%\end{enumerate}
Then 
there exists a constant $N=N(k,d,d_1,a,\chi')$ such that for any $i_1,\ldots,i_k\in\{1,\ldots,d\}$ one has almost surely%\mate{The term with the $\chi_a(Z)$ is ugly and not important, which is a bad combination. Can think about simplifying it}
    \begin{equs}
\big|&\E_s^B(\cY\chi'_a(\cZ)\partial_{i_1,\ldots,i_k}h(\cX))\big|\\&\leq N\|h\|_{L^\infty}\mu^{k/2}\|\cY\|_{\bD_s^{k^*,p^*}}
%\max_{\tilde a\in[a,2^{k-1}a]}\|\chi_{\tilde a}(Z)\|_{\mathbb{D}_s^{k^*,p^*}}
\max(1,\|\cZ\|_{\bD_s^{k^*,p^*}})^{k^*}
\max(1,\mu\|D_{\cH_s^\perp}\cX\|_{\bD_s^{k^*,p^*}}^2)^{k^*}.
%        \\&\leq N \lambda^{-k}\gamma_{s,t}^k\|h\|_{L^\infty}\Vert DX\Vert_{\mathbb{D}^{k,2^{k+1}q'}[s,t]}^k\Vert Y\Vert_{\mathbb{D}^{k,2^{k+1}q'}([s,t])}\|\chi_{2^{k}a}(Z)\|_{\mathbb{D}^{k,2^{k+2}q}([s,t])}^k,
       \label{est:IBP-main}
    \end{equs}
    %{\color{red}Possible alternative: let $N$ also depend on $a$ and replace the thing involving $Z$ by $\max(1,\|Z\|_{\bD_s^{k^*,p^*}})^{k^*}$. Not much simpler but this way the estimate looks a bit more uniform. Idk..}
 %   where $\gamma_{s,t}:={\max(1,\frac{\|D X\|_{\mathbb{D}^{k,4^{k+2}q'}}^2([s,t])}{\lambda_{s,t}}})^{k}$, $\frac{1}{q'}+\frac{1}{q}=\frac{1}{p}$,   $p,q,q'\in[2,\infty)$.
    %and $\frac{1}{l}+\frac{1}{8k}=1$,
\end{lemma}
\begin{proof}
%\emph{Step 1 (Smuggling cutoffs)}. 
Denote, for any $i=1,\ldots,d$ and $\tilde a\in[a,2^{k-1}a]$,
\begin{equ}
    \Xi_{i,\tilde a}=\chi'_{\tilde a}(\cZ)\sum_{j=1}^m (\mathcal{M}_{s}^{-1})_{i,j}D_{\cH^\perp_s}\cX^j\in \cH^\perp_s.
\end{equ}
Notice that the presence of the cutoff $\chi'_{\tilde a}(\cZ)$ is essential, since we only assume control (or even, the existence) of $\cM^{-1}$ where this cutoff is nonzero.
With this notation, we have
$\chi'_{ a}(\cZ)\partial_ih(\cX)=\langle D_{\cH^\perp_s} h(\cX),\Xi_{i,a}\rangle_{\cH^\perp_s}$ as an immediate consequence of the chain rule.
This implies
$$
\E_s^B \big(\cY \chi'_{a}(\cZ)\partial_ih(\cX)\big)
=\E_s^B\big(h(\cX)\delta_{\cH^\perp_s}(\Xi_{i, a}\cY)\big)
$$
by \eqref{eq:conditional_IBP}. However, in order to iterate this procedure, by the above remark,  further cutoffs are needed.
We claim that one can smuggle in a larger cutoff for free and one has in fact
\begin{equ}\label{eq:more-cutoff}
    \E_s^B \big( \cY \chi'_a(\cZ)\partial_ih(\cX)\big)=\E_s^B\big(\chi'_{2a}(\cZ)h(\cX)\delta_{\cH^\perp_s}(\Xi_{i,a}\cY)\big).
\end{equ}
To this end, note the identities $(\nabla \chi'_{2a})\chi'_{a}=0$ and $\chi'_{2a}\chi'_{a}=\chi'_{a}$. Then by the chain rule we have
\begin{equs}
     \chi'_{a}(\cZ)D h(\cX)&=\chi'_{2a}(\cZ)\chi'_{a}(\cZ)D h(\cX)+(\nabla \chi'_{2a})(\cZ)\chi'_{a}(\cZ) h(\cX) D\cZ
     \\
     &=\chi'_{a}(\cZ)D\big(\chi'_{2a}(\cZ)h(\cX)\big).
\end{equs}
After applying $\Pi_{\cH^\perp_s}$, the same holds with $D_{\cH^\perp_s}$ in place of $D$.
Therefore, 
\begin{equs}
    \E_s^B \big( \cY \chi'_{a}(\cZ)\partial_ih(\cX)\big)&=\E_s^B \big( \cY \langle D_{\cH^\perp_s} h(\cX),\Xi_{i, a}\rangle_{\cH^\perp_s}\big)
=\E_s^B \big( \cY \langle D_{\cH^\perp_s} (\chi'_{2a}(\cZ)h(\cX)),\Xi_{i, a}\rangle_{\cH^\perp_s}\big).
\end{equs}
Integrating by parts by \eqref{eq:conditional_IBP} again, we get  \eqref{eq:more-cutoff} as claimed.

To treat $\partial_{i_1,\ldots,i_k}$ in place of $\partial_i$, we iterate this procedure: define the random variables $W_{i_1,\ldots, i_\ell}$, $\ell=1,\ldots,k$, inductively by
\begin{equ}
    W_{i_1}=\delta_{\cH^\perp_s}(\Xi_{i_1,a}\cY),\qquad W_{i_1,\ldots,i_\ell}=\delta_{\cH^\perp_s}\big(\Xi_{i_\ell,2^{\ell-1}a}W_{i_1,\ldots,i_{\ell-1}}\big).
\end{equ}
Then the above arguments yield
\begin{equs}\label{eq:IBP-main}
   \E_s^B\big(\cY\chi'_a(\cZ)\partial_{i_1,\ldots,i_k} h(\cX)\big)=\E_s^B\big(h(\cX)W_{i_1,\ldots,i_{k}}\big).
  \end{equs}
From \eqref{eq:IBP-main} it is clear that
\begin{equ}\label{hallo}
\big|\E_s^B(\cY\chi'_a(\cZ)\partial_{i_1,\ldots,i_k} h(X))\big|
\leq \|h\|_{L^\infty}\|W_{i_1,\ldots,i_k}\|_{\mathbb{D}^{0,1}_{s}}.
\end{equ}
To bound the right-hand side, we apply \cref{lem:continuity-delta-high} recursively. 
Let $n_1,\ldots, n_k$ be defined as follows. We set $n_1$ to be the $n$ of \cref{lem:continuity-delta-high} with the choice $m=0$, $p=1$. Then for $\ell>1$ we set $n_\ell$ to be the $n$ of \cref{lem:continuity-delta-high}
with the choice $m=n_{\ell-1}$, $p=4^{\ell-1}$.
%We shall bound certain conditional Malliavin norms of $W_{i_1,\ldots,i_k}$ recursively in $k$.
We then write, by alternately applying \cref{lem:continuity-delta-high} and the Cauchy-Schwarz inequality
\begin{equs}
    \|W_{i_1,\ldots,i_k}\|_{\mathbb{D}^{0,1}_{s}}
    &\lesssim \big\|\Xi_{i_k,2^{k-1}a}W_{i_1,\ldots,i_{k-1}}\big\|_{\bD^{n_1,2}_s}
    \\
&\leq     \|\Xi_{i_k,2^{k-1}a}\|_{\bD^{n_1,4}_s}
\big\|W_{i_1,\ldots,i_{k-1}}\big\|_{\bD^{n_1,4}_s}
\\
&\lesssim\|\Xi_{i_k,2^{k-1}a}\|_{\bD^{n_1,4}_s}
\big\|\Xi_{i_{k-1},2^{k-2}a}W_{i_1,\ldots,i_{k-2}}\big\|_{\bD^{n_2,8}_s}
\\
&\vdots
\\
&\lesssim \Big(\prod_{\ell=1}^k\|\Xi_{i_\ell,2^{\ell-1}a}\|_{\bD_s^{n_{k+1-\ell},4^{k+1-\ell}}}\Big)\|\cY\|_{\bD_s^{n_k,4^k}}.\label{hullo}
\end{equs}
It remains to bound the terms in the above product.
To simplify notation, we fix $i\in\{1,\ldots,d\}$ and $\tilde a\in[a,2^{k-1}a]$, let $n= n_k$ and $p= 4^k$, and aim to bound $\|\Xi_{i,\tilde a}\|_{\mathbb{D}_s^{n,p}}$.
Starting with the Cauchy-Schwarz inequality
\begin{equ}\label{hello}
    \|\Xi_{i,\tilde a}\|_{\mathbb{D}_s^{n,p}}
    \leq \|\chi'_{\tilde a}(\cZ) \cM_{s}^{-1}\|_{\mathbb{D}_s^{n,2p}}
    \|D_{\cH^\perp_s}\cX\|_{\mathbb{D}_s^{n,2p}}.
\end{equ}
So it remains to bound 
$\|\chi'_{\tilde a}(\cZ) \cM_{s}^{-1}\|_{\mathbb{D}_s^{n,2p}}$.
%for $p\leq 2^{k+1}$ and $0\leq \ell\leq k$.
%which we do by induction on $\ell=0,\ldots,k$.
As a warm-up, note that for $n=0$ we have by assumption
\begin{equ}
    \|\chi'_{\tilde a}(\cZ) \cM_{s}^{-1}\|_{\mathbb{D}_s^{0,2p}}\leq \big(\E_s\Gamma^{-2p}\big)^{1/2p}\leq\mu,
    %=\|g_{\tilde L}(Z) (\cM_{s,t}^X)^{-1}\|_{L^p|\cF_s}\leq \lambda^{-1}\|\Gamma^{-1}\|_{L^p|\cF_s}\leq \lambda^{-1}C.
\end{equ}
provided $p^*$ is chosen to be bigger than $2\cdot4^k$. We now bound the higher order Malliavin derivatives. To avoid the issues of the possible singularity of $\cM_s^{-1}$, we add a small perturbation $\eps\,\id$, $\eps>0$. 
Note that from the identity $    D((\mathcal{M}_{s}+\eps \,\id)^{-1})=(\mathcal{M}_{s}+\eps \,\id)^{-1}D \mathcal{M}_{s}{(\mathcal{M}_{s}+\eps \,\id)^{-1}}$
and Leibniz’s rule we get%\mate{Here $\otimes$ is meant between elements of $\cH$ and not as matrices. Hopefully understandable.}
\begin{equs}%\label{eq:M-1}
D^\ell&((\mathcal{M}_{s}+\eps \,\id)^{-1})
\\
&=\sum_{j_1+j_2+j_3=\ell-1}\binom{\ell-1}{j_1,j_2,j_3}D^{j_1}((\mathcal{M}_{s}+\eps\,\id)^{-1})\otimes D^{j_2+1}(\mathcal{M}_{s})\otimes D^{j_3}((\mathcal{M}_{s}+\eps\,\id)^{-1}).
\end{equs}
One can iterate this formula for $D^{j_1}((\mathcal{M}_{s}+\eps\,\id)^{-1})$ and $D^{j_3}((\mathcal{M}_{s}+\eps\,\id)^{-1})$, and notice that at each step the number of Malliavin derivatives on $(\mathcal{M}_{s}+\eps\,\id)^{-1}$ strictly decreases. Therefore after at most $\ell$ steps one expresses $D^\ell((\mathcal{M}_{s}+\eps\,\id)^{-1})$ as a sum, where both the number of terms and the coefficients of the terms only depend on $\ell$, and each term is a product of $m$ instances of $(\mathcal{M}_{s}+\eps\,\id)^{-1}$ and $m-1$ instances of Malliavin derivatives of $\cM_{s}$ of order $q_1,\ldots,q_{m-1}$, respectively, where $2\leq m\leq \ell+1$, $1\leq q_i\leq \ell$, and $\sum q_i=\ell$.
After applying $\Pi_{\cH^\perp_s}$ in all $\ell$ variables, the same holds with $D_{\cH_s^\perp}$ in place of $D$.
To treat the latter terms, note that 
  \begin{equs}
    D^{j}_{\cH_s^\perp}\mathcal{M}_{s}=\sum_{j'=0}^{j} \binom{j}{j'}\langle D^{j}_{\cH_s^\perp} \cX,D^{j-j'+1}_{\cH_s^\perp} \cX\rangle_{\mathcal{H}_{s}^\perp,1},
  \end{equs}
  where $\langle h_1\otimes\cdots\otimes h_k,\bar h_1\otimes\cdots\otimes h_m\rangle_{\cH_s^\perp,1}=\langle h_1,\bar h_1\rangle_{\cH_s^\perp}h_2\otimes \cdots\otimes h_{k}\otimes \bar h_2\otimes\cdots\otimes\bar h_{m}.$
  In particular, for any $p'\in[1,\infty)$ one has
  \begin{equ}
\big\|   D^{j}_{\cH_s^{\perp}} \mathcal{M}_{s} \big\|_{\bD_s^{0,p'}}\lesssim \|D_{\cH_s^\perp}\cX\|^2_{\mathbb{D}_{s}^{j,2p'}}.
  \end{equ}
  We now write, with the the sum over $q=(q_1,\ldots,q_{m-1})$ understood as above,
\begin{equs}
   \big\|&D^{\ell'}_{\cH_s^\perp}\big(\chi'_{\tilde a}(\cZ) (\cM_{s}+\eps\,\id)^{-1}\big)\big\|_{\mathbb{D}_s^{0,2p}} 
   \\
   &
   \lesssim\sum_{\ell=0}^{\ell'}\Big\|D^{\ell'-\ell}_{\cH_s^\perp}(\chi'_{\tilde a}(\cZ))\otimes D^{\ell}_{\cH_s^\perp}((\cM_{s}+\eps\,\id)^{-1})\|_{\mathbb{D}_s^{0,2p}} 
   \\
   &\lesssim \sum_{\ell=0}^{\ell'}\|\chi'_{\tilde a}(\cZ)\|_{\mathbb{D}_s^{\ell'-\ell,6p}}
   \sum_{m=2}^{\ell+1}\sum_q\Big\|\bone_{|\cZ|\leq 2\tilde a}\big|(\cM_{s}+\eps\,\id)^{-1}\big|^m\Big\|_{\mathbb{D}_s^{0,6p}}
   \Big\|
   \prod_{j=1}^{m-1} D^{q_j}_{\cH_s^{\perp}}\cM_s
   \Big\|_{\mathbb{D}_s^{0,6p}}
   \\
   &\lesssim \sum_{\ell=0}^{\ell'}\|\chi'_{\tilde a}(\cZ)\|_{\mathbb{D}_s^{\ell'-\ell,6p}}\sum_{m=2}^{\ell+1}\big(\E^B_s\Gamma^{-6mp}\big)^{1/6p}\|D_{\cH_s^\perp}\cX\|_{\bD_s^{\ell,12\ell p}}^{2(m-1)}.
\end{equs}
Summing up over $\ell'=0,\ldots n$, and taking $p^*\geq 12(n+1)p$, we have
\begin{equs}
    %\|\chi_{\tilde a}(Z) \cM_{s}^{-1}\|_{\mathbb{D}_s^{n,2p}}
    %\leq\liminf_{\eps\to0} 
    \|\chi'_{\tilde a}(\cZ) (\cM_{s}+\eps\,\id)^{-1}\|_{\mathbb{D}_s^{n,2p}}
    &\lesssim 
    \|\chi'_{\tilde a}(\cZ)\|_{\mathbb{D}_s^{n,6p}}
    \sum_{m=1}^{n+1}\mu^m\|D_{\cH_s^\perp}\cX\|_{\bD_s^{n,12n p}}^{2(m-1)}
    \\&\lesssim \|\chi'_{\tilde a}(\cZ)\|_{\mathbb{D}_s^{n,6p}} \mu \max(1,\mu\|D_{\cH_s^\perp}\cX\|_{\bD_s^{n,12n p}}^2)^{n}.
\end{equs}
By the lower semicontinuity of the conditional norm and the almost sure convergence $\chi'_{\tilde a}(\cZ) (\cM_{s}+\eps\,\id)^{-1}\to\chi'_{\tilde a}(\cZ) \cM_{s}^{-1}$ we can pass to the $\eps\to0$ limit and get
\begin{equ}\label{hollo}
    \|\chi'_{\tilde a}(\cZ) \cM_{s}^{-1}\|_{\mathbb{D}_s^{n,2p}}\lesssim \|\chi'_{\tilde a}(\cZ)\|_{\mathbb{D}_s^{n,6p}} \mu \max(1,\mu\|D_{\cH_s^\perp}\cX\|_{\bD_s^{n,12n p}}^2)^{n}.
\end{equ}
Combining \eqref{hallo}, \eqref{hullo}, \eqref{hello}, and \eqref{hollo}, the proof is finished.
\end{proof}

\begin{proof}[Proof of \cref{thm:L}]
Define the germ
\begin{equ}\label{eq:L-A-def}
    A_{s,t}=\E_s\int_s^t\chi_{\rho}(\phi_r^{s,X_s}-\phi_r^{s,Y_s})f(\theta \phi_r^{s,X_s}+(1-\theta) \phi_r^{s,Y_s})\,dr
\end{equ}
for $(s,t)\in[0,1]_{\leq}^2$. 
We aim to verify the conditions of \cref{lem:SSL-vanila}.
Introduce a couple of shorthand notation: for $(s,t)\in[0,1]_\leq^2$, $x,y\in\R^d$, $i\in\{1,2\}$ set
\begin{equs} 
\psi_{r}^{s,(x,y)}&=\theta \phi^{s,x}_{r}+(1-\theta)\phi^{s,y}_{r},
%&\qquad\bar\phi_r^{s,(x,y)}&=\phi_r^{s,x}-\phi_r^{s,y},\\
\\
\bar\phi_r^{s,(x,y)}&=\phi_r^{s,x}-\phi_r^{s,y},\\
\I_{s,r}^{(x,y)}&=\chi_\rho(\phi^{s,x}_{r}-\phi^{s,y}_{r}).
%&\qquad  \bar {\mathbb{I}}_{s,u,r}^i&=\chi_\rho(\phi_r^{s,X_s^i}-\phi_r^{u,X_u^i}),&\qquad 
%\tilde{\mathbb{
%I}}_{s,u,r}&=\bar {\mathbb{I}}_{s,u,r}^1\bar {\mathbb{I}}_{s,u,r}^2.
\end{equs}
   Then we can rewrite \eqref{eq:L-A-def} as
    \begin{equ}
    A_{s,t}=\int_s^t\E_s\big(\I_{s,r}^{(x,y)}f(\psi_{r}^{s,(x,y)})\big)\big|_{(x,y)=(X_s,Y_s)}\,dr.
\end{equ}
%When viewing $A_{s,t}$ as a linear operator mapping $f\in \cC^0$ to $L^\infty(\Omega)$, we also write $A_{s,t}f$.
%Similarly, 
For $(s,r)\in[0,1]_\leq^2$ and $x,y\in\R^d$, consider the linear operator on from $\cC^\infty$ to $L^p(\Omega)$ defined by
\begin{equ}
    B_{s,r}^{(x,y)}f=\E_s\big(\I_{s,r}^{(x,y)}f(\psi_{r}^{s,(x,y)})\big)=\E_s^B\big(\I_{s,r}^{(x,y)}f(\psi_{r}^{s,(x,y)})\big),
    \end{equ}
    using \cref{lem:continuous-version} in the second equality. Also from \cref{lem:continuous-version} we
    have that $B_{s,r}^{(x,y)}f$ has a continuous modification in $x,y$. In the sequel we will need to estimate its supremum in $x,y$, for which therefore it suffices to take $x,y\in\Q^d$.
One has first of all the trivial bound 
\begin{equ}
    |B_{s,r}^{(x,y)}f|\leq \|f\|_{\cC^0}.
\end{equ}
Next we bound $B_{s,r}^{(x,y)}f$ (and consequently of $A_{s,t}f$) through the norm of $f$ in a negative regularity space. Consider first the case $f=\partial_{i_1}\partial_{i_2}g$.
We apply \cref{lem:IBP-with-cutoff} with the choices
\begin{equ}
    k=2,\quad h=g,\quad a=\rho, \quad d_1=d,\quad\chi'=\chi,\quad \cY=1,\quad \cZ=\bar \phi^{s,(x,y)}_r,\quad \cX=\psi^{s,(x,y)}_r.
\end{equ}
Take the $p^*$ and $k^*$ obtained from therein.
To find $\Gamma$ such that \eqref{eq:cond-Mall-nondegenerate} is satisfied, take $z\in\R^d$ and note that
\begin{equ}
    z^*\cM_sz=\|\Pi_{\cH_s^{\perp}}z^*D\cX\|_{\cH}^2.
\end{equ}
By \eqref{eq:malliavin-and-jacobi}
we have
\begin{equ}
   z^*D_v\cX= \bone_{v\in[s,r]}z^*\big(\theta J^{v,x}_r\sigma(\phi^{s,x}_r)+(1-\theta)J^{v,y}_r\sigma(\phi^{s,y}_r)\big).
\end{equ}
The right-hand side is clearly continuous on $(s,r)$ so its $L^\infty$ norm is bounded from below by its left limit as $v\to r$.
Therefore on the event $\{|\cZ|\leq 4\rho\}$
we can write, by \eqref{eq:convex-combination-nondegenerate},
\begin{equ}
    \|z^*D\cX\|_{L^\infty([s,r])}^2\geq \big|z^*\big(\theta\sigma(\phi^{s,x}_r)+(1-\theta)\sigma(\phi^{s,y}_r)\big)\big|^2
    %\big(\theta\sigma(\phi^{s,x}_r)+(1-\theta)\sigma(\phi^{s,y}_r)\big)^*z
    \geq(1/2)\lambda |z|^2,
\end{equ}
and thus fixing $H>\gamma>1-2H$, by \cref{cor:interpolation} we have 
\begin{equ}
    \|\Pi_{\cH_s^{\perp}}z^*D\cX\|_{\cH}^2\geq c |z|^2|r-s|^{2H}\min\Big(1,\frac{|z|}{\|z^*D\cX\|_{\cC^\gamma([s,r])}}\Big)^{\vartheta/\gamma}
\end{equ}
with some constant $c>0$ such that $c^{-1}\lesssim 1$.
We therefore take
$$
\Gamma:=%|r-s|^{2H}\bar \Gamma:=
|r-s|^{2H}\frac{c}{1+\|D\cX\|_{\cC^\gamma([s,r])}^{\vartheta/\gamma}}.
$$
%where $c>0$ is the implicit constant coming from the preceding inequality.
Hence it is natural to choose $\mu=|t-s|^{-2H}(\bar\mu_x+\bar\mu_y)$, with
\begin{equ}\label{eq:bar-mux-def}
    \bar\mu_x:=c^{-1}\Big(\E_s^B\big|1+ \|D\phi^{s,x}_r\|_{\cC^\gamma([s,r])}^{\vartheta/\gamma}\big|^{p^*}\Big)^{1/p^*}.
\end{equ}
%{\color{red}(to put later where it is actually true) Let us remark that it is crucial that the random variable $\bar\mu$ does not depend on $x$ and $y$.}

Next we bound the terms appearing on the right-hand side of \eqref{est:IBP-main}, starting with $\bar\mu_x$.
For the remainder of the proof we use the convention that $p'$ denotes an exponent depending on $\alpha,\gamma,H$.  Further,  $\Xi'$ will denote a random variable which might depend on $s$ and the usual parameters but not on $x$. In addition,   $\|\Xi'\|_{L_m(\Omega)}\lesssim_m 1$ for all $m\geq 1$, that is, its $L_m$ norms depend only on the parameters given in the statement of theorem and $m$. 

Recalling \eqref{eq:malliavin-and-jacobi} we can start with the
straightforward bound
\begin{equ}\label{eq:minor0}
    \bar\mu_x\lesssim \E_s^B \big|1+\|J^{\cdot,x}_r\|_{\cC^{\gamma}([s,r])}\big|^{p'}\E_s^B\big|1+\|\sigma(\phi^{s,x}_\cdot)\|_{\cC^{\gamma}([s,r])}\big|^{p'}.
\end{equ}
From the boundedness of $\sigma$ and $\nabla\sigma$ and \eqref{eq:controlled_phi_as} we have
\begin{equ}
\big|1+\|\sigma(\phi^{s,x}_\cdot)\|_{\cC^{\gamma}([s,r])}\big|\lesssim\big|1+[\blue{B}]_{\cR^{H_-}}\big|^{p'}
\end{equ}
and so 
\begin{equ}
    \E_s^B\big|1+\|\sigma(\phi^{s,x}_\cdot)\|_{\cC^{\gamma}}\big|^{p'}\leq \Xi'.
\end{equ}
As for the first term in \eqref{eq:minor0}, we have from 
\cref{lem:bounds_for_J_sigma(phi)}
that almost surely
\begin{equ}
    \E^B_s\sup_{u\in[s,r]}|J^{u,x}_r|^{p'}\leq \Xi'.
\end{equ}
We emphasize that $\Xi'$ can be chosen to not depend on $x\in\Q^d$. 
As for the seminorm, we have from
\eqref{eq:Jacobian"flow"} and the continuity of all of the random fields that for any $x\in\R^d$
\begin{equs}      \,[J^{\cdot,x}_r]_{\cC^{\gamma}([s,r])}&\lesssim \sup_{y\in\Q^d}|\nabla J^{0,y}_r|\sup_{y\in\Q^d}|(J^{0,y}_s)^{-1}|\sup_{y\in\Q^d}[\overleftarrow{\phi}^{0,y}_{\cdot}]_{\cC^{\gamma}([s,r])}
    \\
    &\quad+\sup_{y\in\Q^d}| J^{0,y}_r|\sup_{y\in\Q^d}|\nabla(J^{0,y}_s)^{-1}|\sup_{y\in\Q^d}[\overleftarrow{\phi}^{0,y}_{\cdot}]_{\cC^{\gamma}([s,r])}
    \\
    &\quad+\sup_{y\in\Q^d}| J^{0,y}_r|\sup_{y\in\Q^d}[(J^{0,y}_\cdot)^{-1}
    ]_{\cC^{\gamma}([s,r])}
    \\
    &\leq \Xi'(1+\sup_{y\in\Q^d}[\overleftarrow{\phi}^{0,y}_{\cdot}]_{\cC^{\gamma}([s,r])}),
\end{equs}
using \cref{lem:bounds_for_J_sigma(phi)} again
in the last inequality.
As for the last term, notice that for $u<v$
\begin{equs}
  \big|\overleftarrow{\phi}^{0,y}_{u}
  -
  \overleftarrow{\phi}^{0,y}_{v}\big|
  &=
  \big|\overleftarrow{\phi}^{0,\phi^{u,y}_v}_{v}
  -
  \overleftarrow{\phi}^{0,y}_{v}\big|
  \leq \sup_{z\in\Q^d}|(J^{0,z}_v)^{-1}| |\phi^{u,y}_v-y|
  \leq \Xi'|u-v|^\gamma
\end{equs}
by the same arguments as above. So we can conclude
\begin{equ}
    \mu\leq |t-s|^{-2H}\Xi'.
\end{equ}
Moving forward in \eqref{est:IBP-main}, since $\cY=1$, the next nontrivial term is $\cZ=\bar\phi^{s,(x,y)}_r$. Here we invoke \cref{lem:bound_controlled_flow} and \cref{lem:flow-D-with-power}  to conclude
\begin{equ}
   \max(1,\|\cZ\|_{\bD_s^{k^*,p^*}})^{k^*}
 \leq \Xi'.
\end{equ}
Finally, applying \cref{lem:flow-D-with-power} again we get
\begin{equ}
    \|D_{\cH_s^\perp} \cX\|_{\bD_s^{k^*,p^*}}^2\leq |t-s|^{2H}\Xi'
\end{equ}
and therefore
\begin{equ}
\max(1,\mu\|D_{\cH_s^\perp}\cX\|_{\bD_s^{k^*,p^*}}^2)^{k^*}\leq \Xi'.
\end{equ}
Putting these bounds together, \eqref{est:IBP-main} yields 
\begin{equ}
   | B^{(x,y)}_{s,r}\partial_i\partial_j g|\leq \Xi'\|g\|_{L^\infty}|r-s|^{-2H}.
\end{equ}
Since this holds uniformly in $(x,y)\in\Q^{2d}$, it also holds uniformly in $(x,y)\in \R^{2d}$, and thus after substituting $(x,y)=(X_s,Y_s)$ we get
\begin{equ}
 \Big\|\big(B^{(x,y)}_{s,r}\partial_i\partial_j g\big)\big|_{(x,y)=(X_s,Y_s)}\Big\|_{L^p(\Omega)}\lesssim\|g\|_{L^\infty}|r-s|^{-2H}.
\end{equ}
It remains to interpolate. That is, by applying \cref{lem:interpolation} (with $Tf=\big(B^{(x,y)}_{s,r}f\big)\big|_{(x,y)=(X_s,Y_s)}$, $k=2$, $\ell=0$, $\Gamma_0=1$, and $\Gamma_{-2}=C|r-s|^{-2H}$, and $\beta=\gamma$) and integrating in $r$ we get that
\begin{equ}\label{eq:estimate-A-for-L}
    \|A_{s,t}\|_{L^p(\Omega)}\lesssim \|f\|_{\cC^{\alpha-1}}|t-s|^{1+(\alpha-1)H}.
\end{equ}
Since $1+\gamma H>1/2$ by our standing assumption $\gamma>-1/(2H)$,
the condition \cref{lem:SSL-vanila} (i) is satisfied with $\beta_1=1+\gamma H$ and $\Gamma_1=C\|f\|_{\cC^{\alpha-1}}$. 
\\

Next we move on to the treatment of $\E_s\delta A_{s,u,t}$. For this part we introduce the further cutoff function
\begin{equ}
    \tilde{\mathbb{I}}_{s,u,r}^{(x,y,z,w)}=\chi_\rho(\phi_r^{s,x}-\phi_r^{u,z})\chi_\rho(\phi_r^{s,y}-\phi_r^{u,w}).
\end{equ}
To help the reader's intuition, let us remark that, with the choices $(z,w)=(X_u,Y_u)$ and $(x,y)=(X_s,Y_s)$ that we make below, one expects $\tilde{\mathbb{I}}_{s,u,r}^{(z,w,x,y)}$ to be mostly $1$ regardless of whether the two solutions are close.
The purpose of this cutoff, similarly to the previous argument, is to ensure the invertability of a certain Malliavin matrix. In fact, both in this function and in others we often use the rewriting
\begin{equ}
    \tilde{\mathbb{I}}_{s,u,r}^{(x,y,z,w)}=\tilde{\mathbb{I}}_{u,u,r}^{(a,b,z,y)}|_{(a,b)=(\phi^{s,x}_u,\phi^{s,y}_u)}.
\end{equ}
We write
\begin{equs}
     \E_s\delta A_{s,u,t}     
     &=\E_s\int_u^t\bbI_{s,r}^{(X_s,Y_s)}f(\psi_{r}^{s,(X_s,Y_s)})-\bbI_{u,r}^{(X_u,Y_u)}f(\psi_{r}^{u,(X_u,Y_u)})\,dr
     \\&=\E_s\int_u^t\bbI_{u,r}^{(\phi^{s,X_s}_u,\phi^{s,Y_s}_u)}f(\psi_{r}^{u,(\phi^{s,X_s}_u,\phi^{s,Y_s}_u)})-\bbI_{u,r}^{(X_u,Y_u)}f(\psi_{r}^{u,(X_u,Y_u)})\,dr
     \\
     &=:\E_s\int_u^t C^{(a,b,z,w)}_{u,r}f|_{(a,b,z,w)=(\phi^{s,X_s}_u,\phi^{s,Y_s}_u,X_u,Y_u)}
     %\big|_{(x,y)=(X^1_s,X^2_s)}
     \,dr,
     \end{equs}
     where the last equality serves as the definition of the linear map $f\mapsto C^{(a,b,z,w)}_{u,r}f$.
Define furthermore
\begin{equ}
    \tilde{C}^{(a,b,z,w)}_{u,r}f=C^{(a,b,z,w)}_{u,r}f\tilde{\mathbb{I}}^{(a,b,z,w)}_{u,u,r}, \qquad\hat{C}^{(a,b,z,w)}_{u,r}f=C^{(a,b,z,w)}_{u,r}f-\tilde{C}^{(a,b,z,w)}_{u,r}f.
\end{equ}
By conditional Jensen's inequality and ``taking out what is known'' we have
\begin{equs}
    \big\|&\E_s\big(C^{(a,b,z,w)}_{u,r}f|_{(a,b,z,w)=(\phi^{s,X_s}_u,\phi^{s,Y_s}_u,X_u,Y_u)}\big)\big\|_{L^p(\Omega)}
    \\&\leq
    \big\|\E_u\big(C^{(a,b,z,w)}_{u,r}f|_{(a,b,z,w)=(\phi^{s,X_s}_u,\phi^{s,Y_s}_u,X_u,Y_u)}\big)\big\|_{L^p(\Omega)}
    \\
    &=\big\|\E_u\big(C^{(a,b,z,w)}_{u,r}f\big)|_{(a,b,z,w)=(\phi^{s,X_s}_u,\phi^{s,Y_s}_u,X_u,Y_u)}\big\|_{L^p(\Omega)}. 
\end{equs}
The main objects of interests are therefore
\begin{equ}
B^{(a,b,z,w)}_{u,r}f:=\E_uC^{(a,b,z,w)}_{u,r}f=\E_u^BC^{(a,b,z,w)}_{u,r}f%,\qquad B^{(x,y,z,w)}_{s,u,r}f=\E_s\big( C^{(a,b,z,w)}_{u,r}f|_{(a,b)=(\phi^{s,x}_u,\phi^{s,y}_u)}\big)
\end{equ}
and the analogous objects with tilde and hat.

The main steps will be bounding $\tilde B^{(a,b,z,w)}_{u,r}f$ and $\hat B^{(a,b,z,w)}_{u,r}f$ by a random variable uniformly in the parameters $(a,b,z,w)$, whose $L^p(\Omega)$ norm behaves in the expected way in terms of $|r-u|$ once the above substitutions are performed. As before, since it is clear that continuous modifications (in $a,b,z,w$) exist, we assume all coordinates to be rational.

For the two terms we use different decompositions of $C$. For $\tilde B^{(a,b,z,w)}_{u,r}f$, we write
\begin{equs}
    C^{(a,b,z,w)}_{u,r}f&=C^{1,(a,b,z,w)}_{u,r}f+C^{2,(a,b,z,w)}_{u,r}f
    \\
    &:=\bbI^{(a,b)}_{u,r}\big(f(\psi^{u,(a,b)}_r)-f(\psi^{u,(z,w)}_r)\big)+\big(\bbI^{(a,b)}_r-\bbI^{(z,w)}_r\big)f(\psi^{u,(z,w)}_r).
\end{equs}
With the corresponding $\tilde C^{i}$ and $\tilde B^{i}$, we now estimate $\tilde B^1$.
\begin{equs}
    \tilde B^{1,(a,b,z,w)}_{u,r}f=\int_0^1&\E_u\Big(\tilde{\bbI}^{(a,b,z,w)}_{u,u,r}\bbI^{(a,b)}_{u,r}
    \\&
    \times\nabla f\big(\theta'\psi_r^{u,(a,b)}+(1-\theta')\psi_r^{u,(z,w)}\big)\big(\psi_r^{u,(a,b)}-\psi_r^{u,(z,w)}\big)
    \Big)d\theta'.
\end{equs}
We estimate the above quantity for each coordinate (denote by $j$) of $\nabla$ and each $\theta'$ separately, that is, we bound the quantity
\begin{equ}
    \E_u^B\Big(\tilde{\bbI}^{(a,b,z,w)}_{u,u,r}\bbI^{(a,b)}_{u,r}\partial_j f\big(\theta'\psi_r^{u,(a,b)}+(1-\theta')\psi_r^{u,(z,w)}\big)\big(\psi_r^{u,(a,b)}-\psi_r^{u,(z,w)}\big)
    \Big)
\end{equ}
using
 \cref{lem:IBP-with-cutoff}. Let us assume $f=\partial_{i_1}\partial_{i_2}g$. The the objects in \cref{lem:IBP-with-cutoff} are chosen as
\begin{equ}
    \cX=\theta'\psi_r^{u,(a,b)}+(1-\theta')\psi_r^{u,(z,w)},
\end{equ}
\begin{equ}
    k=3,\quad h=g,\quad a=\rho,\quad \cY=\psi_r^{u,(a,b)}-\psi_r^{u,(z,w)},
\end{equ}
\begin{equ}
d_1=3d,\,\,\chi'(x_1,x_2,x_3)=\chi(x_1)\chi(x_2)\chi(x_3),\,\, \cZ=(\phi^{u,a}_r-\phi^{u,z}_r,\phi^{u,b}_r-\phi^{u,w}_r,\phi^{u,a}_r-\phi^{u,b}_r).
\end{equ}
Take the corresponding $k^*$ and $p^*$ from the lemma.
By the same argument as in the previous case, by \eqref{eq:convex-combination-nondegenerate} we have that on the event $\{|\cZ|<8\rho\}$
\begin{equ}
    \|z^*D\cX\|_{L^\infty([u,r])}^2
    \geq(1/2)\lambda |z|^2.
\end{equ}
Therefore we can again take
\begin{equ}
    \Gamma:=|r-u|^{2H}\frac{c}{1+\|D\cX\|_{\cC^\gamma([u,r])}^{\alpha/\gamma}},
\end{equ}
where $c>0$ is a constant and $c^{-1}\lesssim 1$.
Similarly to the previous case, we choose $\mu=|r-u|^{-2H}(\bar\mu_a+\bar\mu_b+\bar\mu_z+\bar\mu_w)$, where we recall the definition \eqref{eq:bar-mux-def} (with $s$ therein replaced by $u$). As already argued in the previous case, we have
\begin{equ}\label{eq:onestep1}
    \mu\leq |r-u|^{-2H}\Xi'.
\end{equ}
The estimates on $\cX$ and $\cZ$ are also very analogous to the previous case, yielding
\begin{equ}\label{eq:onestep2}
   \max(1,\|\cZ\|_{\bD_u^{k^*,p^*}})^{k^*},\,\max(1,\mu\|D_{\cH_u^\perp}\cX\|_{\bD_u^{k^*,p^*}}^2)^{k^*}
 \leq \Xi'.
\end{equ}
What \emph{is} different compared to the previous case is the bound on $\cY$, since this time it allows us to gain a factor that is small when $a-z$ and $b-w$ are small.
We group terms like
\begin{equ}
    \cY=\theta\big(\phi^{u,a}_r-\phi^{u,z}_r\big)+(1-\theta)\big(\phi^{u,b}_r-\phi^{u,w}_r\big),
\end{equ}
and estimate the two terms in an identical way, so we focus on the first one.
One has
\begin{equs}
    \|\phi^{u,a}_r-\phi^{u,z}_r\|_{\mathbb{D}_u^{k^*,p^*}}&=\Big\|\int_0^1(a-z)J^{u,\theta'' a+(1-\theta'')z}_r\,d\theta''\Big\|_{\mathbb{D}_u^{k^*,p^*}}
    \\
    &\leq |a-z|\int_0^1\big\|J^{u,\theta'' a+(1-\theta'')z}_r\big\|_{\mathbb{D}_u^{k^*,p^*}}\,d\theta''.\label{eq:onestep2.5}
\end{equs}
By \cref{lem:Jacobi-Malliavin-Derivatives} for each $\theta''$, one has 
$\big\|J^{u,\theta'' a+(1-\theta'')z}_r\big\|_{\mathbb{D}_u^{k^*,p^*}}\leq \Xi'$ almost surely and therefore
\begin{equ}\label{eq:onestep3}
    \|\cY\|_{\mathbb{D}_u^{k^*,p^*}}\leq (|a-z|+|b-w|)\Xi'
\end{equ}
almost surely (the nullset may depend on $a,b,z,w$ but that doesn't bother us).
By \cref{lem:IBP-with-cutoff}, \eqref{eq:onestep1}, \eqref{eq:onestep2}, \eqref{eq:onestep3} we get
\begin{equ}
\big| \tilde B^{1,(a,b,z,w)}_{u,r}\partial_{i_1}\partial_{i_2}g\big|\leq\|g\|_{L^\infty}|r-u|^{-3H}(|a-z|+|b-w|)\Xi'
\end{equ}
almost surely for all $a,b,z,w$. By continuity, this in fact holds for all $a,b,z,w$ simultaneously almost surely, and therefore
\begin{equs}
\big\| \big(&\tilde B^{1,(a,b,z,w)}_{u,r}\partial_{i_1}\partial_{i_2}g\big)|_{(a,b,z,w)=(\phi^{s,X_s}_u,\phi^{s,Y_s}_u,X_u,Y_u)}\big\|_{L^p(\Omega)}
\\&\leq\|g\|_{L^\infty}|r-u|^{-3H}\big\|(|\phi^{s,X_s}_u-X_u|+|\phi^{s,Y_s}_u-Y_u|)\Xi'\big\|_{L^p(\Omega)}
\\
&\lesssim\|g\|_{L^\infty}|r-u|^{-3H}|s-u|^{1+\alpha H}
\end{equs}
using Cauchy-Schwarz inequality and \cref{Prop:sta-ini} to get the last line.

A much simpler bound can be obtained by simply writing
\begin{equ}
   \big| \tilde B^{1,(a,b,z,w)}_{u,r}g\big|\leq\|\nabla g\|_{L^\infty} \|\cY\|_{\mathbb{D}^{0,1}_u}\leq\|\nabla g\|_{L^\infty} (|a-z|+|b-w|)\Xi',
\end{equ}
which, by the same argument as before, yields
\begin{equs}
\big\| \big(\tilde B^{1,(a,b,z,w)}_{u,r}g\big)|_{(a,b,z,w)=(\phi^{s,X_s}_u,\phi^{s,Y_s}_u,X_u,Y_u)}\big\|_{L^p(\Omega)}\lesssim \|\nabla g\|_{L^\infty}|s-u|^{1+\alpha H}.
\end{equs}
Interpolation, i.e. \cref{lem:interpolation} with $Tf=\big(\tilde B^{1,(a,b,z,w)}_{u,r}g\big)|_{(a,b,z,w)=(\phi^{s,X_s}_u,\phi^{s,Y_s}_u,X_u,Y_u)}$, $k=2$, $\ell=1$, $\Gamma_{-2}=C|r-u|^{-3H}|s-u|^{1+\alpha H}$, $\Gamma_1=C|s-u|^{1+\alpha H}$, $\beta=\gamma$, yields
\begin{equs}
\big\| \big(\tilde B^{1,(a,b,z,w)}_{u,r}f\big)|_{(a,b,z,w)=(\phi^{s,X_s}_u,\phi^{s,Y_s}_u,X_u,Y_u)}\big\|_{L^p(\Omega)}\lesssim \|f\|_{\cC^{\gamma}}|r-u|^{(\gamma-1)H}|s-u|^{1+\alpha H}.
\end{equs}
Since $(\gamma-1)H>-1/2-H>-1$, $|r-u|^{(\gamma-1)H}$ is integrable in $r$, and we get
\begin{equ}
    \int_{u}^t\big\| \big(\tilde B^{1,(a,b,z,w)}_{u,r}f\big)|_{(a,b,z,w)=(\phi^{s,X_s}_u,\phi^{s,Y_s}_u,X_u,Y_u)}\big\|_{L^p(\Omega)}\,dr\lesssim
    \|f\|_{\cC^{\alpha-1}}|t-s|^{2+(\gamma+\alpha-1)H}.
\end{equ}

Moving on to the next term $\tilde B^{2,(a,b,z,w)}_{u,r}f$, observe  that 
\begin{equ}
 (\mathbb{I}_{u,r}^{(a,b)}-\mathbb{I}_{u,r}^{(z,w)})\tilde{\mathbb{
I}}_{u,u,r}^{(a,b,z,w)}
=  (\mathbb{I}_{u,r}^{(a,b)}-\mathbb{I}_{u,r}^{(z,w)})\tilde{\mathbb{
I}}_{u,u,r}^{(a,b,z,w)}\chi_{6\rho}\big(\phi_r^{u,z}-\phi_r^{u,w}\big).
\end{equ}
Indeed, on the set where $\chi_{6\rho}\big(\phi_r^{u,z}-\phi_r^{u,w}\big)$ is not $1$, that is, $|\phi_r^{u,z}-\phi_r^{u,w}|>6\rho$, at least one of $|\phi_r^{u,a}-\phi_r^{u,z}|$, $|\phi_r^{u,b}-\phi_r^{u,w}|$, or $|\phi_r^{u,a}-\phi_r^{u,b}|$ is bigger than $2\rho$, and thus $ (\mathbb{I}_{u,r}^{(a,b)}-\mathbb{I}_{u,r}^{(z,w)})\tilde{\mathbb{
I}}_{u,u,r}^{(a,b,z,w)}=0$.
With this observation, we can rewrite $\tilde B^{2,(a,b,z,w)}_{u,r}f$ as
\begin{equs}
    \tilde B^{2,(a,b,z,w)}_{u,r}f&=\int_0^1\E_u\Big(\bbI_{u,u,r}^{(a,b,z,w)}\chi_{6\rho}\big(\phi_r^{u,z}-\phi_r^{u,w}\big)f(\psi_r^{u,(z,w)})
    \\
    &\qquad\times \nabla\chi_\rho\big(\theta'\phi^{u,a}_r-\theta'\phi^{u,b}_r+(1-\theta')\phi^{u,z}_r-(1-\theta')\phi^{u,w}_r\big)
    \\
    &\qquad\times \big(\phi^{u,a}_r-\phi^{u,b}_r-\phi^{u,z}_r+\phi^{u,w}_r\big)
    \Big)\,d\theta'.
\end{equs}
To estimate the integrand, we follow the previous steps. First we replace $f$ by $\partial_{i_1}\partial_{i_2}g$ and
use \cref{lem:IBP-with-cutoff} once again, now with the choices
\begin{equ}
    k=2,\quad h=g,\quad a=6\rho,\quad    \cX=\psi_r^{u,(z,w)};
    \end{equ}
    \begin{equ}
       \cY= \partial_j\chi_\rho\big(\theta'\phi^{u,a}_r-\theta'\phi^{u,b}_r+(1-\theta')\phi^{u,z}_r-(1-\theta')\phi^{u,w}_r\big)
\big(\phi^{u,a}_r-\phi^{u,b}_r-\phi^{u,z}_r+\phi^{u,w}_r\big);
\end{equ}
\begin{equ}
d_1=3d,\,\,\chi'(x_1,x_2,x_3)=\chi_{1/6}(x_1)\chi_{1/6}(x_2)\chi(x_3),\,\, \cZ=(\phi^{u,a}_r-\phi^{u,z}_r,\phi^{u,b}_r-\phi^{u,w}_r,\phi^{u,z}_r-\phi^{u,w}_r).
\end{equ}
We have already seen in the treatment of $A_{s,t}$, choosing $\Gamma$ and $\mu$ analogously to therein, one has
\begin{equ}\label{eq:B2onestep1}
    \mu\lesssim |r-u|^{-2H}\Xi'
\end{equ}
and
\begin{equ}\label{eq:B2onestep2}
\max(1,\|\cZ\|_{\bD_s^{k^*,p^*}})^{k^*},\,\max(1,\mu\|D_{\cH_s^\perp}\cX\|_{\bD_s^{k^*,p^*}}^2)^{k^*}
 \leq \Xi'.
\end{equ}
As for the bound on $\cY$, it follows from \cref{lem:bound_controlled_flow} and \cref{lem:flow-D-with-power} that
\begin{equ}
  \big\|\partial_j\chi_\rho\big(\theta'\phi^{u,a}_r-\theta'\phi^{u,b}_r+(1-\theta')\phi^{u,z}_r-(1-\theta')\phi^{u,w}_r\big)\big\||_{\mathbb{D}_u^{k^*,q}}  \leq \Xi'(q).
\end{equ}
Combined with \eqref{eq:onestep2.5}, this yields
\begin{equ}
    \|\cY\|_{\mathbb{D}_u^{k^*,p^*}}\leq (|a-z|+|b-w|)\Xi'
\end{equ}
for all $q\in[1,\infty)$.
So in fact we are in a very similar setting as before, but with even one less derivative hitting $f$. Repeating the previous arguments therefore yields
\begin{equ}
    \int_{u}^t\big\| \big(\tilde B^{2,(a,b,z,w)}_{u,r}f\big)|_{(a,b,z,w)=(\phi^{s,X_s}_u,\phi^{s,Y_s}_u,X_u,Y_u)}\big\|_{L^p(\Omega)}\,dr\lesssim
    \|f\|_{\cC^{\gamma}}|t-s|^{2+(\alpha+\gamma) H}.
\end{equ}
At this point the only remaining term that is left to bound is
\begin{equs}
    \hat B^{(a,b,z,w)}_{u,r}f&= \E_u^B\Big(\mathbb{I}_{u,r}^{(a,b)}f(\psi_{r}^{u,(a,b)})\big(1-\tilde{\mathbb{
I}}_{u,u,r}^{(a,b,z,w)}\big)\Big)-
\E_u^B\Big(\mathbb{I}_{u,r}^{(z,w)}f(\psi_{r}^{u,(z,w)})\big(1-\tilde{\mathbb{
I}}_{u,u,r}^{(a,b,z,w)}\big)\Big).
%\\&-\int_u^t \E\Big(\E\big(\mathbb{I}_{u,r}^{(z,w)}f(\psi_{r}^{u,(z,w)})\big(1-\tilde{\mathbb{I}}_{s,u,r}\big)\big|_{(a,b,z,w)=(\phi_u^{s,x},\phi_u^{s,y},X_u^1,X_u^2)}\big)\big|_{(x,y)=(X_s^1,X_s^2)}\Big)dr
\end{equs}
The two terms are completely identical under the relabeling $a\leftrightarrow z$, $b\leftrightarrow w$, so we bound using the notation of the first one. Rather unsurprisingly, we use \cref{lem:IBP-with-cutoff} once again. Take $f=\partial_{i_1}\partial_{i_2}g$ and make the choices
\begin{equ}
    k=2,\quad h=g,\quad a=\rho, \quad d_1=d,\quad\chi'=\chi,\quad \cZ=\bar \phi^{s,(a,b)}_r,\quad \cX=\psi^{s,(a,b)}_r,
\end{equ}
\begin{equ}
    \cY=1-\chi_\rho(\phi_r^{u,a}-\phi_r^{u,z})\chi_\rho(\phi_r^{u,b}-\phi_r^{u,w}).
\end{equ}
This is in fact almost the same setup as in the treatment of $A_{s,t}$,
the only difference is the choice of $\cY$. % To estimate the Malliavin norm of this new $\cY$, it clearly suffices to bound the Malliavin norm of $\chi_\rho(\phi_r^{u,a}-\phi_r^{u,z})$ with certain higher integrability exponents.
Note that since $1-\chi_\rho$ is smooth and bounded and is $0$ on a neighborhood of the origin, one has $|1-\chi_\rho(x)\chi(y)|\lesssim |x|+|y|$, and so by the chain rule
 there exists $q^*$ such that
\begin{equs}
    \|\cY\|_{\mathbb{D}_u^{k^*,p^*}}&\lesssim\max(1,\|\phi_r^{u,a}-\phi_r^{u,z}\|_{\mathbb{D}_u^{k^*,q^*}}^{q^*},\|\phi_r^{u,b}-\phi_r^{u,w}\|_{\mathbb{D}_u^{k^*,q^*}}^{q^*}\big)
    \\
    &\qquad\times\big(\|\phi_r^{u,a}-\phi_r^{u,z}\|_{\mathbb{D}_u^{k^*,q^*}}+\|\phi_r^{u,b}-\phi_r^{u,w}\|_{\mathbb{D}_u^{k^*,q^*}}\big)
    \\
    &\leq (|a-z|+|b-w|)\Xi',
\end{equs}
using \eqref{eq:onestep2.5} in the last inequality.
Therefore in the application of \cref{lem:IBP-with-cutoff} we have precisely the same bounds as for the term $\tilde B^{2,(a,b,z,w)}_{u,r}$ before.
Hence, just as for that term, we can conclude
\begin{equ}
    \int_{u}^t\big\| \big(\hat B^{(a,b,z,w)}_{u,r}f\big)|_{(a,b,z,w)=(\phi^{s,X_s}_u,\phi^{s,Y_s}_u,X_u,Y_u)}\big\|_{L^p(\Omega)}\,dr\lesssim
    \|f\|_{\cC^{\gamma}}|t-s|^{2+(\alpha+\gamma) H}.
\end{equ}
We have bounded all components of $\E_s\delta A_{s,u,t}$. We can conclude
\begin{equ}
    \|\E_s\delta A_{s,u,t}\|_{L^p(\Omega)}\lesssim \|f\|_{\cC^{\gamma}}|t-s|^{2+(\alpha+\gamma-1) H}.
\end{equ}
From the assumptions $\alpha>1-1/(2H)$ and $\gamma>-1/(2H)$ it follows that the exponent is larger than $1$, and so
the condition \cref{lem:SSL-vanila} (ii)
is satisfied with $\beta_2=2+(\alpha+\gamma-1) H$ and $\Gamma_2=C\|f\|_{\cC^{\gamma}}$.

We now claim that with the process $J$ playing the role of $\cA$, the inequalities \cref{lem:SSL-vanila} (I)-(II) are satisfied. Note that this concludes the proof, since \cref{lem:SSL-vanila} (III) is precisely the claim of \cref{thm:L}. 

The inequality \cref{lem:SSL-vanila} (I) trivially holds with $K_1=\|f\|_{L^\infty}$. As for \cref{lem:SSL-vanila} (II), we have
\begin{equs}
    \big\|&\E_s \big(J_t-J_s-A_{s,t})\|_{L^p(\Omega)}
    \\&=\Big\|\E_s\int_s^t \chi_{\rho}(X_r-Y_r)f\big(\theta X_r+(1-\theta)Y_r\big)
    \\&\qquad-\chi_{\rho}(\phi_r^{s,X_s}-\phi_r^{s,Y_s})f(\theta \phi_r^{s,X_s}+(1-\theta) \phi_r^{s,Y_s})\,dr\Big\|_{L^p(\Omega)}
    \\
    &\lesssim (1+\|f\|_{\cC^1})|t-s|\sup_{r\in[s,t]}\big(\|X_r-\phi^{s,X_s}_t\|_{L^p(\Omega)}+\|Y_r-\phi^{s,Y_s}_t\|_{L^p(\Omega)}\big)
    \\
    &\lesssim (1+\|f\|_{\cC^1})|t-s|^{2+\alpha H}.
\end{equs}
Therefore inequality \cref{lem:SSL-vanila} (II) also holds with $K_2=C(1+\|f\|_{\cC^1})$. The proof is finished.
\end{proof}

\section{The iterated integral and the joint rough path}
We now move on to the construction of (a truncated version of) the rough path lift above $(L,B)$.
For this, we construct the integral of $B$ against $L$ anticipated in the introduction \eqref{eq:Q-intro}.
Note that by imposing geometricity, this also gives a meaning to the integral of $L$ against $B$ (which is therefore an integral against $B$ that is neither rough nor Young).
\begin{theorem}\label{thm:K}
Let \cref{asn:b} and \cref{asn:sigma} hold. Let $\chi\in\cut_d$, $p\in[1,\infty)$, $\theta\in[0,1]$, $\ell\in\{1,\ldots,d\}$ and let $X$ and $Y$ be two solutions to \eqref{eq:main}. For $f\in \cC^1$ define the two-parameter process
\begin{equ}\label{eq:K}
K_{s,t}=\int_s^t(B_r^\ell-B_s^\ell)\Big(\chi_\rho(X_r-Y_r)f\big(\theta X_r+(1-\theta)Y_r\big)\Big)\,dr.
\end{equ}
Then there exists a constant $C=C(\alpha,p,d,\lambda,H,\|\sigma\|_{\cC^5},C^X_{D, p}, C^X_{S, 2p}, C^Y_{D,p}, C^Y_{S, 2p})$ such that
\begin{equ}\label{eq:K-bound}
\|K_{s,t}\|_{L^p(\Omega)}\leq N\|f\|_{\cC^{\alpha-1}}|t-s|^{1+\alpha H}.
\end{equ}
\end{theorem}
Comparing \eqref{eq:J} and \eqref{eq:K}, note that $K_{s,t}=\int_{s}^t(B_u^\ell-B_s^\ell)\,dJ_u$.
Before we move on with the proof, the following corollary is immediate.
\begin{corollary}\label{cor:K}
    Let \cref{asn:b} and \cref{asn:sigma} hold. Let $\chi\in\cut_d$, $p\in[1,\infty)$, and let $X$ and $Y$ be two solutions to \eqref{eq:main}.  Then the map $f\mapsto\cK f$ defined for $f\in\cC^2(\R^d; \R^d)$ by   
\begin{equ}
(\cK f)_{s,t}=\int_0^1\int_s^t(B_r-B_s)\otimes\Big(\chi_\rho(X_r-Y_r)\nabla f\big(\theta X_r+(1-\theta)Y_r\big)\Big)\,dr\,d\theta
\end{equ}
  extends as a continuous linear map from $\cC^{\alpha+}(\R^d; \R^d)$ to $\cC^{1+\alpha H}_{2}([0,1]; L_p(\Omega ; \R^{d_0} \otimes \R^{d\times d}))$, whose norm depends only on $\alpha,p,d,d_1,\lambda,H,\|\sigma\|_{\cC^5},C^X_{D, p}, C^X_{S, 2p}, C^Y_{D,p}, C^Y_{S, 2p}$. Moreover, one has the identity
    \begin{equ}\label{eq:chen-ver1}
        \delta (\cK f)_{s,u,t}=B_{s,u}\otimes(\cJ f)_{u,t}.
    \end{equ}
    
    \end{corollary}
    \begin{proof}
        Apply \eqref{eq:K-bound} for each coordinate of $B$ and with each coordinate of $\nabla f $ in place of $f$ and then integrate in $\theta$. The identity \eqref{eq:chen-ver1} is elementary for $f \in\cC^2$ and therefore holds for all $f \in \cC^{\alpha+}$ by continuity.
    \end{proof}

\begin{proof}[Proof of \cref{thm:K}]
We fix some $S,T\in[0,1]_{\leq}^2$ and all other time points in the rest of the proof will be taken from $[S,T]$.
Define the germ
\begin{equ}\label{eq:K-A-def}
    A_{s,t}=\E_s\int_s^t(B_r^\ell-B_S^\ell)\Big(\chi_{\rho}(\phi_r^{s,X_s}-\phi_r^{s,Y_s})f(\theta \phi_r^{s,X_s}+(1-\theta) \phi_r^{s,Y_s})\Big)\,dr
\end{equ}
for $(s,t)\in[S,T]_{\leq}^2$. 
We aim to verify the conditions of \cref{lem:SSL-vanila}. The steps are very similar to the proof of \cref{thm:L}, with $\gamma=\alpha-1$.
In fact, there is only one difference in each of the steps:  in the application of \cref{lem:IBP-with-cutoff} one gains an extra factor of an increment of $B_r^\ell-B_S^\ell$ in the term $\cY$. To estimate the Malliavin norm of this extra factor, notice that clearly $D^k\cY=0$ for $k\geq 2$, while for any $q_1,q_2\geq 1$
\begin{equ}
    \big\|(\E_s|B_r^\ell-B_S^\ell|^{q_1})^{1/q_1}\big\|_{L^{q_2}(\Omega)}\leq |r-S|^{H}, \qquad \|D(B_r^\ell-B_S^\ell)\|_{\cH}\lesssim |r-S|^H.
\end{equ}
\iffalse
\begin{equs}                               \label{eq:mixed_A}
    \red{\text{Can we pu here the exact bounds for $\|A_{s,t}\|$}}
\end{equs}
\begin{equs}                               \label{eq:mixed_delta_A}
    \red{\text{Can we pu here the exact bounds for $\|\delta A_{s,u,t}\|$ so that O can refer to those later on?}}
\end{equs}
\fi
It follows that in the application of \cref{lem:SSL-vanila} one can gain an extra factor $|T-S|^H$ in both $\Gamma_1$ and $\Gamma_2$.
That is, one has
\begin{equ}\label{eq:estimate-A-for-K}
    \|A_{s,t}\|_{L^p(\Omega)}\lesssim \|f\|_{\cC^{\alpha-1}}|T-S|^{H}|t-s|^{1+(\alpha-1)H},
\end{equ}
\begin{equ}\label{eq:estimate-deltaA-for-K}
    \|\E_s\delta A_{s,u,t}\|_{L^p(\Omega)}\lesssim \|f\|_{\cC^{\alpha-1}}|T-S|^H|t-s|^{2+2\alpha H-2H}.
\end{equ}
It is also clear as before that the process $(\cA_t)_{t\in[S,T]}$ has to coincide with $(K_{S,t})_{t\in[S,T]}$, and therefore \cref{lem:SSL-vanila} (III) with the choice $s=S$, $t=T$, yields the bound, 
\begin{equ}
    \|K_{S,T}\|_{L^p(\Omega)}\lesssim \|f\|_{\cC^{\alpha-1}}|T-S|^H|T-S|^{1+(\alpha-1)H}
\end{equ}
as desired.
\end{proof}

The integral constructed above enables one to lift the pair $(B, L)$ to a rough path. Unlike a usual path, however, here the components have different regularities. The extension of the usual definitions is rather straightforward and is summarised below. As before, the distinction between $B$ as a path, $B$ as the first component of the rough path $\blue{B}$, and $B$ as a component of the mixed rough path above $(B, L)$ is crucial, so we distinguish mixed rough paths with a new color.

For the remainder of the section we fix finite dimensional Euclidean spaces $V,E,K$ and exponents  $\beta \in (1/2, 1)$, $\gamma \in (1/3, 1/2)$. For later applications one may keep in mind $\beta$ to stand for the regularity of $L$  and $\gamma$ to stand for the regularity of $B$.  More precisely, $\beta$ will be $1+(\alpha-1)H_+$ and $\gamma$ will be $H_-$ (see \eqref{eq:Choice of H-}). 

\begin{definition}[Rough paths of mixed regularity]      
 We denote by $\mathcal{R}^{\gamma, \beta}([0, 1]; V\times E)$ the collection of all elements ${\color{purp}g}=(g^\bullet, g^\circ, \bbg^{\bullet\bullet},  \bbg^{\circ\bullet}, \bbg^{\bullet\circ})$, such that 
    $(g^\bullet, \bbg^{\bullet\bullet}) \in  \mathcal{R}^\gamma([0, 1]; V)$,
    $g^\circ \in \cC^\beta([0, 1];E)$,     
        $ (\bbg^{\bullet\circ}, \bbg^{\circ\bullet}) \in  \cC^{\beta+\gamma}_2([0,1]^2;  (V\otimes E) \times (E\otimes V) )$, and that the Chen's relations
        \begin{equs}                \label{eq:chen_mixed}
            \bbg^{\bullet\circ}_{s, t}-\bbg^{\bullet\circ}_{s, u}- \bbg^{\bullet\circ}_{u, t}= g^\bullet_{s, u} \otimes g^\circ_{u, t}, \qquad \bbg^{\circ\bullet}_{s, t}-\bbg^{\circ\bullet}_{s, u}- \bbg^{\circ\bullet}_{u, t}= g^\circ_{s, u} \otimes g^\bullet_{u,t}       
        \end{equs}
        hold for all $(s,u,t)\in[0,1]^3_\leq$.
        Moreover, we set 
        \begin{equ}
         \,    [\purp{g}]_{\mathcal{R}^{\gamma, \beta}}= [g^{\bullet}]_{\cC^\gamma}+ [g^{\circ}]_{\cC^\beta}+[\mathbb{g}^{\bullet \bullet}]_{\cC^{2\gamma}_2}+[\mathbb{g}^{\circ  \bullet}]_{\cC^{\gamma+ \beta}_2}+[\mathbb{g}^{\bullet \circ}]_{\cC^{\gamma+ \beta}_2}.
        \end{equ}
\end{definition}
As before, the set $\mathcal{R}^{\gamma, \beta}$ is equipped with the metric
 \begin{equs}
     d_{\gamma, \beta} (\purp{g}, \purp{h}):= \|g^\bullet-h^\bullet\|_{\cC^\gamma}+\|g^\circ-h^\circ\|_{\cC^\beta}+ [\purp{g}-\purp{h}]_{\mathcal{R}^{\gamma, \beta}}. 
 \end{equs}

\begin{remark}
For $i,j\in\{\bullet,\circ\}$, the process $\bbg^{ij}$ stands for the iterated integral of $g^i$ against $g^j$.
 Since the iterated integral of $g^\circ$ against itself is canonically well-defined as a Young integral, it is not included separately in the definition.
 %Since it is of order $2 \beta>1$, they are redundant for integration purposes, hence not included in the definition. 
\end{remark}

Given %$\bg \in \mathcal{R}^{\beta, \gamma}([0, 1]; V \times W)$
$g=(g^\bullet,g^\circ)\in\cC^{\gamma}([0,1];V)\times\cC^{\beta}([0,1];E)$ 
we denote by $\mathcal{D}_{g}^{2\gamma} ([0, 1]; K )$ the set of all functions $\purp{f} = (f, \D_\bullet f,\D_\circ f) :[0, 1] \to K \times  \mathcal{L}(V; K)   \times  \mathcal{L}(E;K) $, such that 
\begin{equs}
 \, [\purp{f}]_{\mathcal{D}_g^{2\gamma}}:= \sup_{{s \neq t}} \frac{| f_{s,t}-  \D_\bullet f_s g^\bullet_{s, t} -\D_\circ f_s g^\circ_{s, t}|}{|t- s|^{\gamma}} +[\D_\bullet f]_{\cC^{\gamma}([0, 1])}+[\D_\circ f]_{\cC^{\gamma}([0, 1])} < \infty.
\end{equs}
For  $\purp{f}= (f, \D_\bullet f, \D_\circ f)  \in \mathcal{D}_{g}^{2 \gamma} ([0, 1]; K )$ and $\purp{h}=(h, \D_\bullet h, \D_\circ h) \in \mathcal{D}_{g}^{2
\gamma} ([0, 1]; \mathcal{L}(K, K') )$, we set 
\begin{equs}                            \label{eq:def:purple_product}
\purp{hf} : = (hf, e \mapsto  (\D_\bullet h e) f+h (\D_\bullet fe), v \mapsto  (\D_\circ h v) f+h (\D_\circ f v).
\end{equs}
Then, one has $\purp{hf}\in \mathcal{D}_{g}^{2\gamma} ([0, 1]; K' )$, and  similarly to \eqref{eq:controlled_product} we have
\begin{equs}     \label{eq:controlled_product_mixed}
\|\purp{hf}\|_{\cD^{2\gamma}_{g}([S,T])} \leq C \| \purp{h}  \|_{\cD^{2\gamma}_{g}([S,T])}\| \purp{f}  \|_{\cD^{2\gamma}_{g}([S,T])}(1+[g^\bullet]_{ \mathcal{C}^\gamma([S,T])}+[g^\circ]_{ \mathcal{C}^\beta([S,T])})^2,
\end{equs}
for a universal constant $C$. 

Then for  $\purp{g} \in   \mathcal{R}^{ \gamma,\beta}([0, 1]; V \times E)$ and $\purp{f} \in \mathcal{D}_{g}^{2\gamma} ([0, 1]; \mathcal{L}( E;  K))$ we can define the $K$-valued rough integral with respect to $g^\circ$ as
\begin{equs}
    \int_0^t \purp{f_r} \, d\purp{g^\circ_r} := \lim_{|\mathcal{P}| \to 0} \sum_{ [s, u] \in \mathcal{P}_{[0, t]}} A_{s, u},
\end{equs}
where  $A_{s, u} := f_s g^\circ_{s,u}+ \D_\bullet f_s \bbg^{\bullet\circ}_{s, u}. 
$
Indeed, 
by \eqref{eq:chen_mixed},  for $0 \leq s \leq v \leq u \leq 1$ we have
\begin{equs}
    \delta A_{s, v, u}& = - (f_{s,v}-  \D_\bullet f_s g^\bullet_{s, v}) g^\circ_{v,u} - \D_\bullet f_{s,v} \bbg^{\bullet\circ }_{v, u}
    \\
    & = - (f_{s,v}-\D_\circ f_s g^\circ_{s, v}- \D_\bullet f_s g^\bullet_{s, v} ) g^\circ_{v,u} - \D_\bullet f_{s,v} \bbg^{\bullet\circ}_{v, u} - (\D_\circ f_s g^\circ_{s, v})g^\circ_{v,u},
\end{equs}
so that 
\begin{equs}
     |\delta A_{s, v, u}| &  \leq 2 [\purp{f}]_{\mathcal{D}_{g}^{2\gamma}}  [\purp{g}]_{\mathcal{R}^{\beta, \gamma}} |u-s|^{2\gamma +\beta } + \|\D_\circ f_s\|_{\cC^0} [\purp{g}]^2_{\mathcal{R}^{\beta, \gamma}} |u-s|^{2\beta}\\
    &  \lesssim  \|\purp{f}\|_{\mathcal{D}_{g}^{2\gamma}} (1+ [\purp{g}]_{\mathcal{R}^{\beta, \gamma}}^2) |t-s|^{\delta}, 
\end{equs}
where $\delta= (2\gamma +\beta)\wedge 2\beta >1$. 
By the sewing lemma the above limit exists and satisfies the bound
\begin{equs}
     |\int_s^t \purp{f_r} \, d\purp{g^\circ_r} - f_s g^\circ_{s,t}- \D_\bullet f_s \bbg^{\bullet \circ}_{s, t}|  \lesssim 
     \|\purp{f}\|_{\mathcal{D}_{g}^{\lambda}} (1+ [\purp{g}]_{\mathcal{R}^{\beta, \gamma}}^2) |t-s|^{\delta},
\end{equs}
which in turn implies that 
\begin{equs}
    |\int_s^t \purp{f_r} \, d\purp{g^\circ_r} - f_s g^\circ_{s,t}|  \lesssim 
    \|\purp{f}\|_{\mathcal{D}_{g}^{2 \gamma}}  (1+ [\purp{g}]_{\mathcal{R}^{\beta, \gamma}}^2) |t-s|^{2\gamma}. 
\end{equs}
As a consequence, for any $0\leq S\leq T\leq 1$,
\begin{equ}
    \purp{\int_S^\cdot \purp{f_r \, dg^\circ_r}}:=\Big(\int_S^\cdot \purp{f_r} \, d\purp{g^\circ_r},0,f_\cdot \Big)\in\mathcal{D}_g^{2\gamma}([S,T];K).
\end{equ}

Similarly, for $\purp{f} \in \mathcal{D}_{g}^{2\gamma} ([0, 1]; \mathcal{L}( V;  K))$ we can define the $K$-valued rough integral with respect to $g^\bullet$ as
\begin{equs}
    \int_0^t \purp{f_r} \, d\purp{g^\bullet_r} := \lim_{|\mathcal{P}| \to 0} \sum_{ [s, u] \in \mathcal{P}_{[0, t]}} A'_{s, u},
\end{equs}
where 
\begin{equs}
  A'_{s,u}:=   f_s g^\bullet_{s,u}+  \D_\circ f_s \bbg^{\circ\bullet}_{s,u}+\D_\bullet f_s \bbg^{\bullet\bullet}_{s, u}. 
\end{equs}
Indeed, as before, one can compute 
\begin{equs}
   |\delta A'_{s, u, v}| & =   | - (f_{s,v}-  \D_\circ f_s g^\circ_{s, v} -\D_\bullet f_s g^\bullet_{s, v} ) g^\bullet_{v,u} - \D_\circ f_{s,v} \bbg^{\circ\bullet}_{v, u}- \D_\bullet f_{s,v} \bbg^{\bullet\bullet}_{v, u} |
   \\
   & \lesssim   
   \|\purp{f}\|_{\mathcal{D}_{g}^{2\gamma}} (1+ [\purp{g}]_{\mathcal{R}^{\beta, \gamma}}) |t-s|^{3\gamma},
\end{equs}
(recall that  $3\gamma \wedge(2\gamma+\beta)=3\gamma >1$). 
By the sewing lemma the above limit exists and satisfies the bound
\begin{equs}
    |  \int_s^t \purp{f_r} \, d\purp{g^\bullet_r}-  f_s g^\bullet_{s,t}- \D_\circ f_s \bbg^{\circ\bullet}_{s,t}+\D_\bullet f_s \bbg^{\bullet\bullet}_{s, t} |  \lesssim \|\purp{f}\|_{\mathcal{D}_{g}^{2\gamma}} (1+ [\purp{g}]_{\mathcal{R}^{\beta, \gamma}}) |t-s|^{3\gamma},
\end{equs}
which implies that 
\begin{equs}
     |  \int_s^t \purp{f_r} \, d\purp{g^\bullet_r}-  f_s g^\bullet_{s,t} |  \lesssim \|\purp{f}\|_{\mathcal{D}_{g}^{2\gamma}} (1+ [\purp{g}]_{\mathcal{R}^{\beta, \gamma}}) |t-s|^{2\gamma}.
\end{equs}
As a consequence, for any $0\leq S\leq T\leq 1$,
\begin{equ}
    \purp{\int_S^\cdot \purp{f_r \, dg^\bullet_r}}:=\Big(\int_S^\cdot \purp{f_r} \, d\purp{g^\bullet_r},f_\cdot,0 \Big)\in\mathcal{D}_g^{2\gamma}([S,T];K).
\end{equ}

The following two lemmata are straightforward extensions of the standard Kolmogorov criterion for paths to rough paths to rough paths of mixed regularity. They can be proved in the same was as  \cite[Theorem 3.1 and Theorem 3.3]{Friz-Hairer}.
\begin{lemma}
    Let $(G^\bullet, G^\circ):  \Omega \times [0,1] \to V \times E$ and $(\mathbb{G}^{\bullet \bullet}, \mathbb{G}^{\circ \bullet}, \mathbb{G}^{\bullet \circ}) : \Omega \times [0, 1]^2 \to (V \otimes V ) \times (E \otimes V) \times (V \otimes E)$. Assume that for all $(s,u,t) \in [0,1]^3_\leq$, $(G^\bullet, \mathbb{G}^{\bullet \bullet})$ satisfies \eqref{eq:Chen_single_regularity} and $(G^\bullet, G^\circ, \mathbb{G}^{\bullet \circ}, \mathbb{G}^{\circ \bullet})$ satisfy \eqref{eq:chen_mixed}. Let $p \geq 2$ and $\tilde{\gamma}_\bullet, \tilde{\gamma}_\circ \in (0,1)$ with $\tilde{\gamma}_\bullet, \tilde{\gamma}_\circ >1/p$. Assume further that  there exist constants $K_{ij}$, $i,j \in \{ \bullet, \circ\}$, such that for all $(s,t)$
    \begin{equs}
        \| G^i_{s,t}\|_{L_p}\leq  K_i |t-s|^{\tilde{\gamma}_i}, \qquad   
        \| \mathbb{G}^{i j }_{s,t}\|_{L_{p/2}}\leq  K_{ij} |t-s|^{\tilde{\gamma}_i+\tilde{\gamma}_j}
    \end{equs}
    for $i,j \in \{ \bullet, \circ\}$ with $(i,j) \neq (\circ, \circ)$. Then for any $\gamma_\bullet< \tilde{\gamma}_\bullet-1/p $, $\gamma_\circ< \tilde{\gamma}_\circ-1/p $, there exists a modification of $\purp{G}:= (G^\bullet, G^\circ,\mathbb{G}^{\bullet \bullet}, \mathbb{G}^{\circ \bullet}, \mathbb{G}^{\bullet \circ} )$ such that with probability one, we have $\purp{G} \in \mathcal{R}^{\gamma_\bullet, \gamma_\circ}$. Moreover, there is a constant $C=C(\gamma_\bullet, \tilde{\gamma}_\bullet, \gamma_\circ, \tilde{\gamma}_\circ, p)$ such that 
    \begin{equs}
        \| [G^i]_{\cC^{\gamma_i}} \|_{L_p} \leq C K_i, \qquad \| [\mathbb{G}^{ij}]_{\cC^{\gamma_i+\gamma_j}_2} \|_{L_p} \leq  C (K_{ij}+K_iK_j) . 
    \end{equs}
\end{lemma}

\begin{lemma}
   Let $p \geq 2$, $\tilde{\gamma}_\bullet, \tilde{\gamma}_\circ  >1/p$, and $\gamma_\bullet< \tilde{\gamma}_\bullet-1/p $, $\gamma_\circ< \tilde{\gamma}_\circ-1/p$. Let $\purp{G}= (G^\bullet, G^\circ,\mathbb{G}^{\bullet \bullet}, \mathbb{G}^{\circ \bullet}, \mathbb{G}^{\bullet \circ} )$,  $\purp{\bar{G}}= (\bar{G}^\bullet, \bar{G}^\circ,\bar{\mathbb{G}}^{\bullet \bullet}, \bar{\mathbb{G}}^{\circ \bullet}, \bar{\mathbb{G}}^{\bullet \circ} )$ random variables with values in $\mathcal{R}^{\gamma_\bullet, \gamma_\circ}$. Assume further that there exist 
   constants   $K_{ij}$, $i,j \in \{ \bullet, \circ\}$, such that for all $(s,t)$
    \begin{equs}
        \| G^i_{s,t}-\bar{G}^i_{s,t}\|_{L_p}\leq  K_i |t-s|^{\tilde{\gamma}_i}, \qquad   
        \| \mathbb{G}^{i j }_{s,t}-  \bar{\mathbb{G}}^{i j}_{s,t}\|_{L_{p/2}}\leq  K_{ij} |t-s|^{\tilde{\gamma}_i+\tilde{\gamma}_j}
    \end{equs}
    for $i,j \in \{ \bullet, \circ\}$ with $(i,j) \neq (\circ, \circ)$. Then, there exists a constant $C=C(\gamma_\bullet, \tilde{\gamma}_\bullet, \gamma_\circ, \tilde{\gamma}_\circ, p)$, such that 
    \begin{equs}
        \| [G^i-\bar{G}^i]_{\cC^{\gamma_i}} \|_{L_p} \leq N K_i, \qquad \| [\mathbb{G}^{ij}-\bar{\mathbb{G}}^{ij}]_{\cC^{\gamma_i+\gamma_j}_2} \|_{L_p/2} \leq  N ( K_{ij}+K_iK_j). 
    \end{equs}
\end{lemma}
The following is a direct consequence of the previous two lemmata combined with \cref{cor:L} and \cref{cor:K}. For $ f \in \cC^{\alpha+}(\R^d; \R^d)$,  set $\tilde\cK^{jki}_{s,t} f=(\cJ^{jk}_{s,t}f)B^i_{s,t} -\cK^{ijk}_{s,t}$ for $i=1,\ldots,d_0$ and $j,k=1,\ldots,d$.  
\begin{corollary}    \label{cor:maps_in_Rouph_Paths}
The map $ f \mapsto \purp{Gf }$ defined for $ f \in\cC^2(\R^d; \R^d)$ by
    \begin{equ}
       \purp{Gf}=\big(B, \cJ f, \mathbb{B}, \tilde\cK f,\cK f \big)
    \end{equ}
    extends uniquely to a continuous map from $\cC^{\alpha+}(\R^d; \R^d)$ to $L^p\big(\Omega;\cR^{H_-, 1+(\alpha-1)H_+}(\R^{d_0} \times \R^{d\times d}) \big)$. Moreover, there exists a constant $C=C(\alpha,p,d,d_0,\lambda,H,H_-,\|\sigma\|_{\cC^5},C^X_{D, p}, C^X_{S, 2p}, C^Y_{D,p}, C^Y_{S, 2p})$  
    such that for all $f, \tilde{f} \in \cC^{\alpha+}(\R^d; \R^d)$, we have 
    \begin{equs}
       \|  [\purp{Gf}]_{\cR^{H_-, 1+(\alpha-1)H_+}} \|_{L_p} & \leq C(1+\| f\|_{\cC^\alpha}),
       \\
       d_{H_-, 1+(\alpha-1)H_+} (\purp{Gf}, \purp{G\tilde{f}}) & \leq C \|f-\tilde{f}\|_{\cC^\alpha}.                     \label{eq:Lipschitz_bound_G}
    \end{equs}
\end{corollary}

From now on, we set 
\begin{equs}
    \purp{G}=(G^{\bullet}, G^{\circ}, \mathbb{G}^{\bullet \bullet}, \mathbb{G}^{\circ \bullet}, \mathbb{G}^{\bullet \circ}):= \purp{Gb}, 
\end{equs}
that is 
\begin{equs}
    \purp{G}\,  ``=" \, \Big( B, L, \int B  \otimes \, dB , \int L  \otimes \, dB, \int B  \,  \otimes dL \Big). 
\end{equs}
\section{Closing the equation}\label{sec:closing}

Throughout this section, we fix two solutions $X, Y$ of \eqref{eq:main} and we set $Z=X-Y$. Moreover, we introduce an operator $I \in \mathcal{L}(\R^d; \mathcal{L}( \R^{d \times d} ; \R^d))$ and an operator valued process $ \Sigma : \Omega \times [0,1] \to \mathcal{L}( \R^d; \mathcal{L}( \R^{d_0} ; \R^d))$, defined  as follows: for $z \in \R^d$, we set 
\begin{equs}            \label{eq:def_I_and_Sigma}
   Iz  := \Bigg( (c^{ij})_{i,j=1}^d \mapsto \Big(\sum_{j=1} z^j c^{ij} \Big)_{i=1}^d \Bigg), \qquad 
     \Sigma_tz  &= \Bigg( (\xi^\varrho)_{\varrho=1}^{d_0} \mapsto \Big( \sum_{l=1}^d \sum_{\varrho=1}^{d_0} \Sigma^{li\varrho}_tz^l_t b\xi^\varrho \Big)_{i=1}^d \Bigg), 
     \end{equs}
     where 
     \begin{equs}
     \Sigma^{li\varrho}_t:=  \int_0^1 \D_{l} \sigma^{i\varrho}(\theta X_t+(1-\theta)Y_t) \, d \theta.  
\end{equs}
We will see below that $Z$ is controlled by $G$ and that its Gubinelli derivatives with respect  to $G^\bullet$ and $G^{\circ}$ are given by $\Sigma Z$ and $IZ$, respectively. 

Next, we define processes which will play the role of Gubinelli derivatives of $I Z$ and $\Sigma Z$. Let us start with $IZ$. For this, let us define two processes $\D_\bullet (IZ) : \Omega \times [0,1] \to \mathcal{L}(\R^{d_0}; \mathcal{L}(\R^{d \times d}; \R^d))$ and 
$\D_\circ (IZ) : \Omega \times [0,1] \to \mathcal{L}(\R^{d\times d}; \mathcal{L}(\R^{d \times d}; \R^d))$ as follows: for $\xi =(\xi^\varrho)_{\varrho=1}^{d_0}$

\begin{equs}
   \D_\bullet (IZ)_t \tilde{b} 
   &=  \Bigg( (c^{ij})_{i,j=1}^d \mapsto \big( \sum_{j, l=1}^d \sum_{\varrho=1}^{d_0} \Sigma^{lj\varrho} Z^l_t  c^{ij} \xi^\varrho  \Big)_{i=1}^d \Bigg) 
\end{equs}
and for $\tilde{c}=(\tilde{c}^{ij})_{i,j=1}^d$
\begin{equs}
     \D_\circ (IZ)_t \tilde{c} =  \Bigg( (c^{ij})_{i,j=1}^d\mapsto  \Big(\sum_{j,l=1}^d  Z^l_t \tilde{c}^{jl} c^{ij}  \Big)_{i=1}^d \Bigg).
\end{equs}

We also set $\blue{Z}= (Z, \Sigma Z)$. Notice that by the fundamental theorem of calculus, we have that $\sigma (X)-\sigma(Y)= \Sigma Z$, which in particular implies that 
\begin{equ}
    \blue{Z}=  (Z, \Sigma Z)= \blue{X-Y}= (X-Y, \sigma(X)- \sigma(Y)) \in  \mathcal{D}^{2H_-}_{G^\bullet}([0, 1]; \R^d), 
\end{equ}
with probability one. Moreover, we set $\blue{I}:=(I,0) \in \mathcal{D}^{2H_-}_{G^\bullet}([0, 1]; \mathcal{L}(\R^d; \mathcal{L}(\R^{d\times d}; \R^d)))$. Consequently, $\blue{IZ}$ as defined in \eqref{eq:def_blue_product} belongs to  $\mathcal{D}^{2H_-}_{G^\bullet}([0, 1];\mathcal{L}(\R^{d\times d}; \R^d))$. In addition, it is straightforward that $\blue{IZ}$ coincides with $(Z, \D_\bullet(IZ))$. Similarly, we set 
$\blue{\Sigma}=(\Sigma, \D_\bullet \Sigma )$, where  $\D_\bullet \Sigma : \Omega \times [0,1] \to \mathcal{L}(\R^{d_0}; \mathcal{L}(\R^d; \mathcal{L}(\R^{d_0}; \R^d)))$, is defined as follows: for $\tilde{\xi} \in \R^d$
\begin{equs}
    \D_\bullet \Sigma_t \tilde{\xi}= \Bigg( (z^l)_{l=1}^d \mapsto \Bigg( (\xi^\varrho)_{\varrho=1}^{d_0} \mapsto \Big(   \sum_{l, q=1}^d \sum_{\varrho, \tilde{\varrho}=1}^{d_0} \hat{\Sigma}^{li\varrho \tilde{\varrho}}_t \tilde{\xi}^{\tilde{\varrho}}z^l\xi^\varrho  \Big)_{i=1}^d   \Bigg)                   \Bigg),
\end{equs}
where 
\begin{equs}
    \hat{\Sigma}^{qli\varrho \tilde{\varrho}}_t= \int_0^1 \D_{ql}\sigma^{i\varrho}(\theta X_t+(1-\theta)Y_t)(\theta \sigma^{q\tilde{\varrho}}(X_t)+(1-\theta) \sigma^{q \tilde{\varrho}}(Y_t)) \, d \theta.
\end{equs}
Then, we have that  $\blue{\Sigma} \in \mathcal{D}^{2H_-}_{G^\bullet}([0, 1];  \mathcal{L}(\R^d; \mathcal{L}(\R^{d_0}; \R^d))) $, which follows from the fact that $\blue{X}, \blue{Y} \in \mathcal{D}^{2H_-}_{G^\bullet}([0, 1]; \R^d)$. Consequently, we also have that $\blue{\Sigma Z}$ (as defined by \eqref{eq:def_blue_product}) belongs to $ \mathcal{D}^{2H_-}_{G^\bullet}([0, 1]; \mathcal{L}(\R^{d_0}; \R^d))$. In fact, $\blue{\Sigma Z}$ is nothing but $\blue{\sigma(X)-\sigma(Y)}$ but it will be convenient to have it in this linear in $Z$ form.

For the next lemma, recall that $\rho$ is defined in \eqref{eq:rho def}
\begin{lemma}          \label{lem:Z_is_controlled_by_G}
Let $\tau=\inf\{t \in [0,1] : |Z_t|> \rho\}\wedge 1$ and set $\purp{Z}=(Z,\Sigma Z,IZ)$ and $\purp{I}=(I,0, 0)$. The following hold.
\begin{enumerate}[(i)]
    \item \label{item:Z_is_controlled_purple} We have $\purp{Z} \in  \mathcal{D}^{2H_-}_{G}([0, \tau], \R^d)$ and $\purp{I} \in \mathcal{D}^{2H_-}_{G}([0, 1]; \mathcal{L}(\R^d; \mathcal{L}(\R^{d\times d}; \R^d)))$, with probability one. 
    \item \label{item:identification_IZ} We have that $ \purp{IZ} \in \mathcal{D}^{2H_-}_{G}([0, \tau], \mathcal{L}(\R^{d \times d};  \R^d))$ and $\purp{IZ}= (IZ,\D_\bullet (IZ) , \D_\circ (IZ)) $. 
    \item \label{item:drifts_int_Z_dG} With probability one,  for all $t \in [0, \tau]$,  we have 
    \begin{equs}             \label{eq:drifts_int_Z_dG}
        D^X_t-D^Y_t = \int_0^t  \purp{ IZ_s} \, d \purp{G^\circ_s}. 
    \end{equs}
\end{enumerate}
 \end{lemma}

 \begin{proof}
   First, let us set 
   \begin{equs}        \label{eq:germ_of_Q}
       A_{s,t}:= (IZ_s)G^\circ_{s,t}+ \D_\bullet (I Z)_s  \mathbb{G}^{ \bullet \circ}_{s, t}. 
   \end{equs}
 We claim that with probability one, for all $t \in [0, 1]$, the limit 
 \begin{equs}                \label{eq:sewing_Q}
     Q_t:= \lim_{\substack{\cP\in\pi_{[0,t]}\\ |\cP|\to0}}\sum_{[s,u]\in\cP}A_{s,u}
 \end{equs}
 exists and defines a process $Q$ such that almost surely $Q \in \mathcal{C}^{1+(\alpha-1)H_+} ([0, 1])$, and 
 \begin{equs}     \label{eq:remainder_Q}
     |Q_{s, t}- (I Z)_s  G^\circ_{s, t}| \lesssim |t-s|^{1+(\alpha-1)H_++H_-}. 
 \end{equs}
 Indeed, for $0\leq s<u<v\leq 1$, we have 
 \begin{equs}
      \delta A_{s,u,v} & = -(IZ_{s,u})G^\circ_{u,t}+\D_\bullet (IZ)_s \delta \mathbb{G}^{\bullet \circ}_{s,u,t}- \D_\bullet (IZ)_{s,u} \mathbb{G}^{\bullet \circ}_{u,t}
     \\
     & = -((IZ_{s,u})- \D_\bullet (IZ)_s G^{\bullet}_{s,u}) G^\circ_{u,t}- \D_\bullet (IZ)_{s,u} \mathbb{G}^{\bullet \circ}_{u,t}
 \end{equs}
where we have used  Chen's relation \eqref{eq:chen_mixed}. Consequently, we have 
 \begin{equs}
     |\delta A_{s,u,v}| \lesssim \big( [\blue{I Z}]_{\mathcal{D}^{2H_-}_{\mathbf{G}^\bullet}} [G^\circ]_{\mathcal{C}^{1+(\alpha-1)H_+}}+ [\D_\bullet (IZ)]_{\mathcal{C}^{H_-}} [\mathbb{G}^{\bullet \circ }]_{\mathcal{C}^{1+(\alpha-1)H_++H_-}_2}\big) |v-s|^{3H_-},
 \end{equs}
 where we have used that $1+(\alpha-1)H_++2H_- \geq 1/2+2H_- \geq  3H_-$. Since $3H_->1$, the claims follows from Gubinelli's sewing lemma, and the fact that 
 \begin{equs}
     |A_{s,v}| & \lesssim \| I Z \|_{\cC^0}[G^\circ]_{\mathcal{C}^{1+(\alpha-1)H_+}} |t-s|^{1+(\alpha-1)H_+}
     \\
     & \qquad + \| \D_\bullet(IZ)\|_{\cC^0}[\mathbb{G}^{\bullet \circ }]_{\mathcal{C}^{1+(\alpha-1)H_++H_-}_2}|t-s|^{1+(\alpha-1)H_++H_-}
     \\
     &  \lesssim \Big( \| Z \|_{\cC^0}[G^\circ]_{\mathcal{C}^{1+(\alpha-1)H_+}} 
+ \| \sigma \|_{\cC^1} \| Z \|_{\cC^0}[\mathbb{G}^{\bullet \circ }]_{\mathcal{C}^{1+(\alpha-1)H_++H_-}_2}\Big) |t-s|^{1+(\alpha-1)H_+} 
\\
& \lesssim |t-s|^{1+(\alpha-1)H_+}.
 \end{equs}
 Next, let  $b \in \mathcal{C}^\infty$, such that $\|b^n-b\|_{\mathcal{C}^\alpha}  \leq 2^{-2n}$.  Let us define 
 \begin{equs}
     G^{\circ,n}_t:= \int_0^t \int_0^1 \chi_\rho(X_r-Y_r) \nabla b^n(\theta X_r+(1-\theta) Y_r ) \, d \theta \, dr, \qquad \mathbb{G}^{\bullet \circ,n}_{s, t}:= \int_s^t G^{\bullet}_{s, r} \, \otimes \, dG^{\circ, n}_r,
 \end{equs}
 where the last one is a Riemann integral,  since $G^{\circ ,n} \in \mathcal{C}^1([0, 1])$,  and notice that $\mathbb{G}^{\bullet \circ,n} \in \mathcal{C}^{1+H_-}_2([0, 1]^2)$. It follows then that with probability one, for all $t \in [0, 1]$, we have 
 \begin{equs}
     \int_0^t IZ_r \, d G^{\circ,n}_r =  \lim_{\substack{\cP\in\pi_{[0,t]}\\ |\cP|\to0}}\sum_{[s,u]\in\cP}(IZ_s) G^{\circ,n}_{s,u}=\lim_{\substack{\cP\in\pi_{[0,t]}\\ |\cP|\to0}}\sum_{[s,u]\in\cP} A_{s, t}^n,
 \end{equs}
 where $A^n_{s,u}:= (IZ_s) G^{\circ,n}_{s,u}+ \D_\bullet (I Z)_s \mathbb{G}^{\bullet \circ, n}_{s,u}$. Consequently, 
 \begin{equs}
     Q_t- \int_0^t Z_r \, dG^{\circ,n}_r = \lim_{\substack{\cP\in\pi_{[0,t]}\\ |\cP|\to0}}\sum_{[s,u]\in\cP} (A_{s,u}-A^n_{s,u}). 
 \end{equs}
 As before, we have 
      \begin{equs}
      |A_{s,v}-A^n_{s,v}| & \leq \|IZ\|_{\mathcal{C}^0}[G^\circ-G^{\circ,n}]_{\mathcal{C}^{1+(\alpha-1)H_+}} |v-s|^{1+(\alpha-1)H_+} 
      \\
      & \qquad + \|\D_\bullet(IZ)\|_{\mathcal{C}^0}[\mathbb{G}^{\bullet \circ}-\mathbb{G}^{\bullet \circ,n}]_{\mathcal{C}^{1+(\alpha-1)H_++H_-}_2} |v-s|^{1+(\alpha-1)H_++H_-}
      \\
      & \lesssim \big( [G^\circ-G^{\circ,n}]_{\mathcal{C}^{1+(\alpha-1)H_+}}  + [\mathbb{G}^{\bullet \circ}-\mathbb{G}^{\bullet \circ,n}]_{\mathcal{C}^{1+(\alpha-1)H_++H_-}_2}\big) |v-s|^{1+(\alpha-1)H_+},
      \\
     |\delta (A-A^n)_{s,u,v}| & = |(IZ_{s,u}-\D_\bullet(IZ)_s G^\bullet_{s,u})(G^\circ_{u,t}-G^{\circ,n}_{u,t}) + \D_\bullet (IZ)_{s,u} (\mathbb{G}^{\bullet \circ}_{u,t}-\mathbb{G}^{\bullet \circ,n}_{u,t})|
     \\
     & \lesssim \big( [\blue{IZ}]_{\mathcal{D}^{2 H_-}_{\mathbf{G}^\bullet}} [G^\circ-G^{\circ,n}]_{\mathcal{C}^{1+(\alpha-1)H_+}}+ 
     \\
     & \qquad + \|\D_\bullet (IZ)\|_{\mathcal{C}^{H_-}} [\mathbb{G}^{\bullet \circ}-\mathbb{G}^{\bullet \circ,n}]_{\mathcal{C}^{1+(\alpha-1)H_++H_-}_2}\big) |v-s|^{1+(\alpha-1)H_++2H_-}
     \\
     & \lesssim \big( [G^\circ-G^{\circ,n}]_{\mathcal{C}^{1+(\alpha-1)H_+}}+  [\mathbb{G}^{\bullet \circ}-\mathbb{G}^{\bullet \circ,n}]_{\mathcal{C}^{1+(\alpha-1)H_++H_-}_2}\big) |v-s|^{3H_-}.
 \end{equs}
 Consequently, by Gubinelli's sewing lemma we get that with probability one, for all $0 \leq s <t \leq 1$, we have 
 \begin{equs}    
    &  |Q_{s, t}-\int_s^t Z_r \, dG^{\circ, n}_r| 
     \\
     \lesssim & \,  \big( [G^\circ-G^{\circ,n}]_{\mathcal{C}^{1+(\alpha-1)H_+}}  + [\mathbb{G}^{\bullet \circ}-\mathbb{G}^{\bullet \circ,n}]_{\mathcal{C}^{1+(\alpha-1)H_++H_-}_2}\big) |v-s|^{1+(\alpha-1)H_+}. \label{eq:Q-approx}
 \end{equs}
 Next we derive an almost sure bound for  the right hand side of the above relation. By \eqref{eq:Lipschitz_bound_G} from \cref{cor:maps_in_Rouph_Paths}, we have 
\begin{equs}
    \|\big( [G^\circ-G^{\circ,n}]_{\mathcal{C}^{1+(\alpha-1)H_+}}  + [\mathbb{G}^{\bullet \circ}-\mathbb{G}^{\bullet \circ,n}]_{\mathcal{C}^{1+(\alpha-1)H_++H_-}_2}\big)\|_{L_2(\Omega)}  \lesssim  \|b-b^n\|_{\mathcal{C}^\alpha} \lesssim  2^{-2n}. 
\end{equs}
Hence, by Markov's inequality we get 
\begin{equs}
 \mathbb{P}\Big(  [G^\circ-G^{\circ,n}]_{\mathcal{C}^{1+(\alpha-1)H_+}}  + [\mathbb{G}^{\bullet \circ}-\mathbb{G}^{\bullet \circ,n}]_{\mathcal{C}^{1+(\alpha-1)H_++H_-}_2} \geq 2^{-n}\Big)   \lesssim   2^{2n}  2^{-4n} = 2^{-2n}. 
\end{equs}
Consequently, by the Borel-Cantelli lemma, there exists a random variable $\xi$ such that with probability one, for all $n \in \mathbb{N}$, we have 
\begin{equs}
 \,  [G^\circ-G^{\circ,n}]_{\mathcal{C}^{1+(\alpha-1)H_+}}  + [\mathbb{G}^{\bullet \circ}-\mathbb{G}^{\bullet \circ,n}]_{\mathcal{C}^{1+(\alpha-1)H_++H_-}_2}\leq \xi 2^{-n}.   
\end{equs}
Combining the above with \eqref{eq:Q-approx}, we get that with probability one
\begin{equs}       \label{eq:Q-approx_C_beta}
    \big\| Q- \int_0^\cdot IZ_r \, dG^{\circ,n}_r \big\|_{\mathcal{C}^{1+(\alpha-1)H_+}} \lesssim \xi 2^{-n}. 
\end{equs}
Next, notice that by the fundamental theorem of calculus we have that 
\begin{equs}
    \int_0^t b^n(X_r)-b^n(Y_r) \, dr = \int_0^t IZ_r \, dG^{\circ, n}, \qquad \text{for $t \in [0, \tau]$}. 
\end{equs}
By the above combined with \eqref{eq:Q-approx_C_beta} and \eqref{item:def_D} of \cref{def_solution}, we get 
\begin{equs}      \label{eq:drifts_Q}
    \| D^X-D^Y- Q\|_{\mathcal{C}^0([0, \tau])}& = \lim_{n \to \infty}\big\| D^X-D^Y- \int_0^\cdot IZ_r \, dG^{\circ, n} \big\|_{\mathcal{C}^0([0, \tau])}
    \\
    &= \lim_{n \to \infty}\big\| D^X-D^Y- \int_0^\cdot b^n(X_r)-b^n(Y_r)  \, dr \big\|_{\mathcal{C}^0([0, \tau])} =0,
\end{equs}
where the limits are taken in probability. Consequently, for any $0\leq s < t < \tau$, we have 
\begin{equs}
    Z_{s, t}= Q_{s,t} + \int_s^t \blue{ \Sigma Z_r} \, d \blue{G^\bullet_r}, 
\end{equs}
which in turn gives 
\begin{equs}
    |Z_{s, t} - (IZ_s) G^\circ_{s, t} -(\Sigma Z_s) G^\bullet_{s,t}|
    & \leq | Q_{s,t}- (IZ_s) G^\circ_{s, t}|  + |\int_s^t \blue{\Sigma Z_r} \, d \blue{G^\bullet_r}-(\Sigma Z_s) G^\bullet_{s,t}| 
\\
& \lesssim |t-s|^{1+(\alpha-1)H_++H_-}+ |t-s|^{2H_-} \lesssim |t-s|^{2H_-},
\end{equs}
where we have used \eqref{eq:remainder_Q} and the fact that $\blue{\Sigma Z} \in \mathcal{D}^{2H_-}_{G^\bullet}([0, 1])$. The above together with the fact that $IZ, \Sigma Z \in \mathcal{C}^{H_-}$ shows that 
 $\purp{Z}=(Z, \Sigma Z, I Z)$ belongs to the space $\mathcal{D}^{2H_-}_{G}([0, \tau], \R^d)$, with probability one. The fact that $\purp{I}=(I,0,0)$  belongs to $\mathcal{D}^{2H_-}_{G^\bullet}([0, 1]; \mathcal{L}(\R^d; \mathcal{L}(\R^{d\times d}; \R^d)))$ is trivial. Hence, we are done with \eqref{item:Z_is_controlled_purple}. 

 Moving to \eqref{item:identification_IZ}, the fact that $\purp{IZ}$ belongs to $\mathcal{D}^{2H_-}_{G}([0, \tau], \mathcal{L}(\R^{d \times d};  \R^d))$
is a direct consequence of \eqref{item:Z_is_controlled_purple}.  The fact that $\purp{IZ}:= (IZ,\D_\bullet(IZ) , \D_\circ (IZ))$ simply follows from the definition of the processes $\D_\bullet(IZ)$ and  $\D_\circ (IZ)$. 

Finally, \eqref{item:drifts_int_Z_dG} follows from \eqref{eq:drifts_Q} and the fact that 
\begin{equ}
    Q_t= \int_0^t  \purp{IZ_s} \, d \purp{G^\circ_s}, 
\end{equ}
which in turn simply follows from the definition of $Q$, that is, from \eqref{eq:germ_of_Q} and \eqref{eq:sewing_Q}. This finishes the proof. 
\end{proof}

Next, let us set $\purp{\Sigma}=(\Sigma, \D_\bullet \Sigma, 0)$, which clearly belongs to $\mathcal{D}^{2H_-}_{G}([0, 1];  \mathcal{L}(\R^d; \mathcal{L}(\R^{d_0}; \R^d)))$,   since $\blue{\Sigma}=(\Sigma, \D_\bullet \Sigma) \in \mathcal{D}^{2H_-}_{G^\bullet}([0, 1];  \mathcal{L}(\R^d; \mathcal{L}(\R^{d_0}; \R^d)))$. Consequently, we have that $\purp{\Sigma Z} \in \mathcal{D}^{2H_-}_{G}([0, 1];  \mathcal{L}(\R^{d_0};\R^d)))$.

Our main focus is the following.

\begin{lemma}                    \label{lem:blue_equals_purple}
    With probability one, for all $t \in [0,1]$ we have 
    \begin{equs}                               \label{eq:equality _of_rough_integrals}
        \int_0^t \blue{\Sigma_r Z_r} \, d \blue{G^\bullet_r}=  \int_0^t \purp{\Sigma_r Z_r} \, d \purp{G^\bullet_r}. 
    \end{equs}
\end{lemma}
Recall that 
\begin{equs}
    \blue{\Sigma Z}&=(
    \Sigma Z, \Sigma (\Sigma Z)(\cdot) +\D_\bullet\Sigma (\cdot) Z ),
    \\
     \purp{\Sigma Z}&=(
    \Sigma Z, \Sigma (\Sigma Z)(\cdot) +\D_\bullet\Sigma (\cdot) Z,\Sigma (IZ)(\cdot)).
\end{equs}

By recalling the definition of the integrals appearing in \eqref{eq:equality _of_rough_integrals}, it is clear that  to have \eqref{eq:equality _of_rough_integrals} it suffices to show that 
\begin{equs}         \label{eq:extra_term=0}
    \lim_{|\mathcal{P}| \to 0} \sum_{ [s, u] \in \mathcal{P}_{[0, t]}} \big(\Sigma_s(IZ_s)\big) \mathbb{G}^{\circ \bullet}_{s,u}=0.
\end{equs}
Notice that  by definition of the two integrals appearing in \eqref{eq:equality _of_rough_integrals}, it follows that with probability one, the limit appearing on the left hand side 
exists. 
Moreover, in terms of indices, \eqref{eq:extra_term=0} is written as follows: for each $i \in \{1,\dots, d\}$, 
\begin{equs}                                      \label{eq:extra_term=0_indice}
 \lim_{|\mathcal{P}| \to 0} \sum_{ [s, u] \in \mathcal{P}_{[0, t]}}  C^i_{s,u}=0,
 \end{equs}
 where, 
 \begin{equs}
     C^i_{s,u}:= \sum_{l=1}^d \sum_{\varrho=1}^{d_0}\Sigma^{li\varrho}_s \sum_{q=1}^d Z^q_s (\tilde{\mathcal{K}}^{lq \varrho}_{s,u}b).
\end{equs}
 Further, notice that by definition of $\tilde{\mathcal{K}}$ (see \cref{cor:maps_in_Rouph_Paths}), we have that $C^i_{u,s}=C^{i,1}_{u,s}+C^{i,2}_{u,s}$, where 
\begin{equs}
 C^{i,1}_{s,u}:= \sum_{l=1}^d \sum_{\varrho=1}^{d_0}\Sigma^{li\varrho}_s \sum_{q=1}^d Z^q_s (\mathcal{J}^{lq}_{s,u}b) B^{\varrho}_{s,u}, \qquad  C^{i,2}_{s,u}:= \sum_{l=1}^d \sum_{\varrho=1}^{d_0}\Sigma^{li\varrho}_s \sum_{q=1}^d Z^q_s (\mathcal{K}^{\varrho lq }_{s,u}b)
\end{equs}
We next show that the terms above are of order $|s-u|^{3H_-}$.

\begin{lemma}                 \label{lem:C1} 
    For each $i \in \{1,\dots , d\}$, with probability one we have 
\begin{equs}
   \sup_{(s,t) \in [0,1]^2_{\leq}} \frac{|C^{i,1}_{s,t} |}{|t-s|^{3H_-}}< \infty.
\end{equs}
\end{lemma}
\begin{proof}
    By the boundedness of $\Sigma$, we have 
   \begin{equs}
       |C^{i, 1}_{s,t}| & \lesssim  \sum_{l=1}^d  \big| \sum_{q=1}^d Z^q_s (\mathcal{J}^{lq}_{s,t}b) B^{\varrho}_{s,t} \big| 
       \lesssim  |t-s|^{H_-} \sum_{l=1}^d  \big| \sum_{q=1}^d Z^q_s (\mathcal{J}^{lq}_{s,t}b) \big| .                  \label{eq:to_estimate_IZL}
   \end{equs}
   Next, recall that by \cref{def_solution}, it is immediate that  $\blue{Z}= (Z, \Sigma Z) \in \mathcal{D}^{2H_-}_{G^\bullet}$, which in particular gives that 
   \begin{equ}   \label{eq:Z_controlled_blue}
       Z_{s,t}= (\Sigma_sZ_s) G^\bullet_{s,t}+O(|t-s|^{2H_-}).
   \end{equ}
   On the other hand, in \cref{lem:Z_is_controlled_by_G} we showed that  $\purp{Z}=(Z,\Sigma Z,IZ) \in  \mathcal{D}^{2H_-}_{G}$, which in particular implies that 
   \begin{equs} 
   \label{eq:Z_controlled_purple}
       Z_{s,t}=(IZ_s)G^\circ_{s,t}+  (\Sigma_sZ_s) G^\bullet_{s,t}+O(|t-s|^{2H_-}).
       \end{equs}
       Subtracting \eqref{eq:Z_controlled_blue} from \eqref{eq:Z_controlled_purple} gives that $(IZ_s)G^\circ_{s,t}=O(|t-s|^{2H_-})$. Keeping in mind the definition of $I$ (see, \eqref{eq:def_I_and_Sigma}) and the fact that again by definition $G^\circ=(\mathcal{J}^{lq}b)_{l,q=1}^d$, we have that    
       \begin{equs}
     \sum_{l=1}^d  \big| \sum_{q=1}^d Z^q_s (\mathcal{J}^{lq}_{s,t}b) \big| \lesssim |t-s|^{2H_-}.
\end{equs}
Replacing the above in \eqref{eq:to_estimate_IZL} gives the desired result. 
\end{proof}
\begin{lemma}                              \label{lem:C2}
For any $p \geq 1$ and $i \in \{1,\dots , d\}$ we have
\begin{equs}
 \sup_{(s,t) \in [0,1]^2_{\leq}}    \frac{\|C^{i,2}_{s,t} \|_{L_p} }{|t-s|^{3H_-}}< \infty.
\end{equs}
\end{lemma}

\begin{proof}
    By the boundedness of $\Sigma$, it suffices to show that for each $l \in \{1, \dots, d\}$ and $\varrho \in \{1, \dots, d_0\}$ we have 
    \begin{equs}
       \Big\| \sum_{q=1}^d Z^q_s (\mathcal{K}^{\varrho lq }_{s,t}b) \Big\|_{L_p} \lesssim |t-s|^{3H_-}.
    \end{equs}
    To this end, let $b_n \in \cC^\infty$ such that $b_n \to b$ in $\cC^\alpha$ and $\|b_n\|_{\cC^\alpha} \leq 2\|b\|_{\cC^\alpha}$ for all $n \in \mathbb{N}$. By \cref{cor:K} we have that 
  \begin{equs}                                   \label{eq:ZK_limit}
       \Big\| \sum_{q=1}^d Z^q_s (\mathcal{K}^{\varrho lq }_{s,t}b) \Big\|_{L_p} = \lim_{n \to \infty}\Big\| \sum_{q=1}^d Z^q_s (\mathcal{K}^{\varrho lq }_{s,t}b_n) \Big\|_{L_p}.
    \end{equs}
    For fixed $n$, we have 
     \begin{equs}         
    \sum_{q=1}^d Z^q_s (\mathcal{K}^{\varrho lq }_{s,t}b_n)  &  = \int_s^t \int_0^1  B^\varrho_{s,r} \chi_{\rho}(X_r-Y_r)  Z_r \nabla b_n^l(\theta X_r+(1-\theta)Y_r) \, d \theta dr
      \\
      & \qquad +  \int_s^t \int_0^1  B^\varrho_{s,r} \chi_{\rho}(X_r-Y_r)   Z_{s,r}\nabla b_n^l(\theta X_r+(1-\theta)Y_r) \, d \theta dr
      \\
      &=: \Gamma^n_1(s,t)+\Gamma^n_2(s,t).    \label{eq:ZK_dec}
    \end{equs}
   By the fundamental theorem of calculus we have 
    \begin{equs}
       | \Gamma^n_1(s,t)| & = \Big| \int_s^t  B^\varrho_{s,r} \chi_{\rho}(X_r-Y_r)  \big( b^l_n(X_r)-b^l_n(Y_r)\big)  \, dr \Big|
       \\
       & = \Big| \int_s^t  B^\varrho_{s,r} \, dD^n_r \Big|, 
    \end{equs}
    where 
    \begin{equs}
        D^n_t= \int_0^t \chi_{\rho}(X_r-Y_r)  \big( b^l_n(X_r)-b^l_n(Y_r)\big)  \, dr.
    \end{equs}
    Consequently, 
    \begin{equs} 
         |\Gamma^n_1(s,t)|  \lesssim [B]_{\cC^{H_-}} [D^n]_{\cC^{1+\alpha H_+}} |t-s|^{H_-+1+\alpha H_+}.
    \end{equs}
 By \cref{thm:L}  (with the choice $\gamma = \alpha$, $\theta=0$ and $\theta=1$)  we see that $\|D^n\|_{\cC^{1+\alpha H}([0,1]; L_{2p})} \lesssim \|b_n\|_{\cC^\alpha} 
    $. Moreover, by Kolmogorov's continuity criterion, for $p$ sufficiently large and $\varepsilon'$ sufficiently small, we have $\|D^n\|_{L_{2p}(\Omega; \cC^{1+\alpha H_+})} \lesssim \|D^n\|_{\cC^{1+\alpha H-\varepsilon'}([0,1]; L_{2p})}$, hence we obtain 
    \begin{equs}
         \| \Gamma^n_1(s,t)\|_{L_p} &  \lesssim \| B \|_{L_{2p}(\Omega;\cC^{H_-})}\|D^n\|_{L_{2p}(\Omega; \cC^{1+\alpha H_+})}  |t-s|^{H_-+1+\alpha H_+}
         \\
         & \lesssim \|b_n\|_{\cC^\alpha}   |t-s|^{3H_-},
    \end{equs}
    where we have used that $1+\alpha H_+>2H_-$, which follows by the choice of $H_-, H_+$ (see, also \eqref{eq:Choice of H-}). . Consequently, we have 
    \begin{equs}           \label{eq:estimate_Gamma_1}
        \limsup_{n \to \infty}   \| \Gamma^n_1(s,t)\|_{L_p} \lesssim \|b\|_{\cC^\alpha}  |t-s|^{3H_-}.
    \end{equs}
    To estimate $\Gamma_2^n$, let $(S, T) \in [0, 1]^2_{\leq}$ and for $q \in \{1, \dots, d\}$ consider the process $(\mathcal{Q}_t)_{t \in [S,T]}$ given by 
    \begin{equs}
       \mathcal{Q}_t= \int_S^t \int_0^1  B^\varrho_{S,r} \chi_{\rho}(X_r-Y_r)   Z_{S,r}\nabla b_n^l(\theta X_r+(1-\theta)Y_r) \, d \theta dr.
    \end{equs}
    It is easy to see that the process $
    \mathcal{Q}$ is reconstructed by the germs 
    \begin{equs}
        Q_{s,t}&= \E_s \int_s^t \int_0^1  B^\varrho_{S,r} \chi_{\rho}(\phi^{s,X_s}_r-\phi^{s,Y_s}_r) (\phi^{s,X_s}_r-\phi^{s,Y_s}_r-Z_S)  \nabla b_n^l(\theta\phi^{s,X_s}_r+(1-\theta)\phi^{s,Y_s}_r) \, d \theta dr.
    \end{equs}
    First, we want to obtain that 
    \begin{equs}                     \label{eq:Q_s_t}
          \|Q_{s,t}\|_{L_p} \lesssim \|b\|_{\cC^\alpha} |t-s|^{1+(\alpha-1)H} |T-S|^{2H_-}.
    \end{equs}
  For this, we decompose $Q$ as 
   \begin{equs}
        Q_{s,t}= & Z_{S,s}\E_s \int_s^t \int_0^1  B^\varrho_{S,r} \chi_{\rho}(\phi^{s,X_s}_r-\phi^{s,Y_s}_r))  \nabla b_n^l(\theta \phi^{s,X_s}_r+(1-\theta)\phi^{s,Y_s}_r) \, d \theta dr
        \\
        &+ \E_s \int_s^t \int_0^1  B^\varrho_{S,r} \chi_{\rho}(\phi^{s,X_s}_r-\phi^{s,Y_s}_r) (\phi^{s,X_s}_r-\phi^{s,Y_s}_r-Z_s)  \nabla b_n^l(\theta\phi^{s,X_s}_r+(1-\theta)\phi^{s,Y_s}_r) \, d \theta dr
        \\
        =&:Q^1_{s,t}+Q^2_{s,t}.             \label{eq:First_Dec_Q}
    \end{equs}
 Notice that $Q^1_{s,t}=Z_{S,s}A'_{s,t}$, where $A'_{s,t}$ is  the integral over $\theta \in (0, 1)$  of the germ in \eqref{eq:K-A-def}, with the choice $f=\nabla b^l_n$. By using then \eqref{eq:estimate-A-for-K}, we get  
    \begin{equs}            \label{eq:Q1_s_t}
        \|Q^1_{s,t} \|_{L_p} &\leq  \|Z_{S,s}\|_{L_{2p}} \|A'_{s,t}\|_{L_{2p}} \\
        &\lesssim \| [Z]_{\cC^{H_-}}\|_{L_{2p}} |T-S|^{H_-} \|b_n\|_{\cC^\alpha} |t-s|^{1+(\alpha-1)H} |T-S|^{H}
        \\
        & \lesssim \|b_n\|_{\cC^\alpha} |t-s|^{1+(\alpha-1)H} |T-S|^{2H_-}.
        \end{equs}
        The estimate for $Q^2_{s,t}$ is essentially a repetition of the arguments for the bound \eqref{eq:estimate-A-for-L} in proof of \cref{thm:L}. One has 
        \begin{equs}
            Q^2_{s,t}= \int_s^t\E_s\big(\I_{s,r}^{(x,y)}\mathbb{J}_{s,r}^{(x,y)}\nabla b^l_n(\psi_{r}^{s,(x,y)})\big)\big|_{(x,y)=(X_s,Y_s)}\,dr, 
        \end{equs}
        where $\I_{s,r}^{(x,y)}$ and $\psi_{r}^{s,(x,y)}$ are as in the proof of \cref{thm:L}, while the extra term $\mathbb{J}_{s,r}^{(x,y)}$ is given by 
        \begin{equs}
            \mathbb{J}_{s,r}^{(x,y)}:= B^\varrho_{S,r}(\phi^{s,x}_r-\phi^{s,y}_r-x+y).
        \end{equs}
        One can repeat then the arguments that lead to \eqref{eq:estimate-A-for-L}, with the following modification: in the application of \cref{lem:IBP-with-cutoff} instead of $\mathcal{Y}=1$ we choose $\mathcal{Y}= \mathbb{J}_{s,r}^{(x,y)}$,
   whose conditional Malliavin norm needs to be estimated to gain the extra factor of $|T-S|^{2H_-}$. First, notice that for $q\geq 1$
   \begin{equs}
       \big(\E_s|B^\varrho_{S,r}|^q\big)^{1/q}&\leq  \big(\E_s[B]_{\cC^{H_-}}^q\big)^{1/q}|T-S|^{H_-},
       \\
\big(\E_s\|D_{\mathcal{H}_s^{\perp}}B^\varrho_{S,r}\|_{\mathcal{H}}^q\big)^{1/q}&\leq  \|D B^\varrho_{S,r}\|_{\mathcal{H}} \lesssim |T-S|^H,
   \end{equs}
   while for $k \geq 2$, we have $D_{\mathcal{H}_s^{\perp}}^k B^\varrho_{S,r}=0$. Moreover, by \cref{cor:phi-x} we have 
   \begin{equs}
        \big(\E_s|\phi^{s,x}_r-\phi^{s,y}_r-x+y|^q\big)^{1/q} \lesssim  \big(\E_s|(1+[\blue{B}]_{\cR^{H_-}})^{q(1+(4H_-+3))/H_-}\big)^{1/q}|T-S|^{H_-}
   \end{equs}
   and for $k \geq 1$, we have 
    \begin{equs}
        \big(\E_s\|D_{\mathcal{H}_s^{\perp}}^k(\phi^{s,x}_r-\phi^{s,y}_r-x+y)\|_{\R^d \otimes \mathcal{H}^{\otimes k}}^q\big)^{1/q} & \leq \big(\E_s\|D^k(\phi^{s,x}_r-\phi^{s,y})\|_{\R^d \otimes\mathcal{H}^{\otimes k}}^q\big)^{1/q} 
        \\
        &\leq 2\big(\E_s|\Lambda_n |^q\big)^{1/q}|T-S|^{H_-},
   \end{equs}
   where we have used \cref{lem:flow-D-with-power}. By these considerations it follows that for $s \in [S,T]$, $k^*\in \mathbb{N}$, $p^* \geq 2$ there exists a random variable $\Xi'_s$ which satisfies the following:
   for any $p \geq 1$, $\|\Xi'_s\|_{L_p} \leq C=C(k^*, p^*, p)<\infty$ (but independent of $s$)
  and  for all $x, y \in \R^d$, $r \in [s,T]$, with probability one, we have 
   \begin{equs}
       \|\mathbb{J}_{s,r}^{(x,y)}\|_{D^{k^*, p^*}_s} \leq \Xi'_s |T-S|^{2H_-}. 
   \end{equs}
   With this estimate in hand, one can repeat the arguments that lead to \eqref{eq:estimate-A-for-L} and conclude that 
   \begin{equs}
    \|Q^2_{s,t}\|_{L_p}  \lesssim \| b_n\|_{\cC^\alpha} |T-S|^{2H_-} |t-s|^{1+(\alpha-1)H},
   \end{equs}
   which combined with \eqref{eq:Q1_s_t} and the fact that $\|b_n\|_{\cC^\alpha} \leq \|b\|_{\cC^\alpha}$ gives \eqref{eq:Q_s_t}. 

To bound $\delta Q_{s,u,t}$, we decompose $Q$ in a different manner. Namely, we have 
     \begin{equs}
        Q_{s,t}&= \E_s \int_s^t \int_0^1  B^\varrho_{S,r} \chi_{\rho}(\phi^{s,X_s}_r-\phi^{s,Y_s}_r) (\phi^{s,X_s}_r-\phi^{s,Y_s}_r)\nabla b_n^l(\theta\phi^{s,X_s}_r+(1-\theta)\phi^{s,Y_s}_r) \, d \theta dr
        \\
        &  \qquad - Z_S\E_s \int_s^t \int_0^1  B^\varrho_{S,r} \chi_{\rho}(\phi^{s,X_s}_r-\phi^{s,Y_s}_r) \nabla b_n^l(\theta\phi^{s,X_s}_r+(1-\theta)\phi^{s,Y_s}_r) \, d \theta dr  
        \\
        &=  \E_s \int_s^t   B^\varrho_{S,r} \chi_{\rho}(\phi^{s,X_s}_r-\phi^{s,Y_s}_r) \big( b_n^l(\phi^{s,X_s}_r)- b_n^l(\phi^{s,Y_s}_r) \big) \, dr
        \\
          &  \qquad - Z_S\E_s \int_s^t \int_0^1  B^\varrho_{S,r} \chi_{\rho}(\phi^{s,X_s}_r-\phi^{s,Y_s}_r) \nabla b_n^l(\theta\phi^{s,X_s}_r+(1-\theta)\phi^{s,Y_s}_r) \, d \theta dr  
          \\
          &= \tilde{Q}^1_{s,t}-\tilde{Q}^2_{s,t}-Z_S\tilde{Q}^3_{s,t},
    \end{equs}
    where $\tilde{Q}^1_{s,t}$ is the germ in \eqref{eq:K-A-def} with the choice $\theta=1$ and $f= b^l$,  $\tilde{Q}^2_{s,t}$ is the germ in \eqref{eq:K-A-def} with the choice $\theta=0$ and $f= b^l$,  and $\tilde{Q}^3_{s,t}$ is the integral over $\theta \in (0, 1)$ of the germ in \eqref{eq:K-A-def} with the choice $f= \nabla b^l$. Consequently, by using \eqref{eq:estimate-deltaA-for-K}, it  is easy to see that 
    \begin{equs}
        \|\E_s\delta Q_{s,u,t}\|_{L_p} \lesssim \|b_n\|_{\cC^{\alpha}}|T-S|^H|t-s|^{2+2\alpha H-2H}. 
    \end{equs}
   We can apply  \cref{lem:SSL-vanila} to get 
    \begin{equs}
        \|\mathcal{Q}_{S,T}\|_{L_p} & \lesssim \|b_n\|_{\cC^\alpha} \Big( |T-T|^{1+(\alpha-1)H} |T-S|^{2H_-}+|T-S|^H|T-S|^{2+2\alpha H-2H}\big) 
        \\
        & \lesssim \|b_n\|_{\cC^\alpha} |T-S|^{3H_-}.
    \end{equs}
Recall that by definition $\Gamma^n_2(S,T)= \mathcal{Q}_{S,T}$, and since in the above $S,T$ were arbitrary, we get 
\begin{equs}                                                   \label{eq:estimate_Gamma_2}
    \limsup_{n \to \infty}\|\Gamma^n_2(s,t)\|_{L_p} \lesssim  \|b\|_{\cC^\alpha} |t-s|^{3H_-}. 
\end{equs}
Finally, by \eqref{eq:ZK_limit}, \eqref{eq:ZK_dec}, \eqref{eq:estimate_Gamma_1}, and \eqref{eq:estimate_Gamma_2}, we conclude that 
\begin{equs}
       \Big\| \sum_{q=1}^d Z^q_s (\mathcal{K}^{\varrho lq }_{s,t}b) \Big\|_{L_p} \lesssim \|b\|_{\cC^\alpha} |t-s|^{3H_-},
\end{equs}
which brings the proof to an end. 
\end{proof}

We can now proceed with the proof of \cref{lem:blue_equals_purple}. 

\begin{proof}[Proof of  \cref{lem:blue_equals_purple}]
Recall that according to the discussion after the statement of the lemma, we know that, with probability one, the limit  of $\sum_{ [s, u] \in \mathcal{P}_{[0, t]}}  C^i_{s,u}$ as $|\mathcal{P}| \to 0$ exists and we only have to show that it is equal to zero. Recall also that 
\begin{equs}
    \sum_{ [s, u] \in \mathcal{P}_{[0, t]}}  C^i_{s,u}=  \sum_{ [s, u] \in \mathcal{P}_{[0, t]}}  C^{i,1}_{s,u}+ \sum_{ [s, u] \in \mathcal{P}_{[0, t]}}  C^{i,2}_{s,u}.
\end{equs}
By \cref{lem:C1}, since $3H_->1$,  the limit as $|\mathcal{P}| \to 0$ of the first term on the right hand side exists with probability one and it is zero. This implies that the second term on the right hand side also converges with probability one. One the other hand, by \cref{lem:C2} it follows that it converges to zero in $L_p$. Hence, the almost sure limit is also zero. 
This shows that with probability one we have \eqref{eq:extra_term=0_indice}, which in turn implies \eqref{eq:equality _of_rough_integrals}. This finishes the proof. 
\end{proof}

\section{Proof of the main result}

We now have all the tools to complete the proof of \cref{thm:main}. 

\begin{proof}[Proof of \cref{thm:main}]
Let $X, Y$ be solutions of \eqref{eq:main}. Then, by \cref{def_solution}, for the difference $Z:=X-Y$, with probability one, for all $t \in [0,1]$, we have 
\begin{equs}
    Z_t= D^X_t-D^Y_t+ \int_0^t \blue{\Sigma_r Z_r} \, d \blue{G^\bullet_r}.
\end{equs}
From this equality, by \cref{lem:Z_is_controlled_by_G} and \cref{lem:blue_equals_purple} we obtain that with probability one, for all $t \in [0, \tau]$ (recall the definition of $\tau$ in \cref{lem:Z_is_controlled_by_G}), 
\begin{equs}
    Z_t=  \int_0^t  \purp{ IZ_s} \, d \purp{G^\circ_s} + \int_0^t \purp{\Sigma_r Z_r} \, d \purp{G^\bullet_r}.
\end{equs}
The above is a linear rough differential equation for the controlled path $\purp{Z}=(Z, \Sigma Z, I Z)$, starting at $Z_0=0$. Consequently, the only solution of the above equation is identically zero, which shows that $X_t=Y_t$ for $t \in [0 , \tau]$. Taking into account now the definition of $\tau$, it follows that $\tau=1$ almost surely, which shows the uniqueness. 

Next we move to the existence part.  Since in \cite{KM} weak existence was shown (under an even weaker condition $\alpha>1/2-1/(2H)$) and above we have shown strong uniqueness, it is classical (going back originally to \cite{Y-W}) that these imply strong existence. Here we provide a sketch based on a characterisation of convergence in probability from \cite{Gyongy_Krylov}.

For $n \in \mathbb{N}$, let $b^n \in \mathcal{C}^\infty$ with $\lim_{n \to \infty}\|b^n-b\|_{\cC^\alpha}=0$, and  consider $X^n$ given by the unique solution of the equation
\begin{equs}
    dX^n_t= b^n(X^n_t) \,  dt + \blue{\sigma(X^n_t)} \, d\blue{B_t}, \qquad X^n_0=x,
\end{equs}
and let us set 
$$
D^n= \int_0^t b^n(X^n_s) \, ds.
$$
Let $X^{k(n)}, X^{k'(n)}$ be arbitrary subsequences and consider the sequence 
$$
Z_n = (X^{k(n)}, D^{k(n)},  X^{k'(n)}, D^{k(n)},  \blue{B} , W) \in \mathcal{Z}
$$ 
where $\mathcal{Z}$ is the Polish space $\cC^{H_-+} \times \cC^{1+\alpha H_++} \times \cC^{H_-+} \times \cC^{1+\alpha H_++} \times \mathcal{R}_{\geo}^{H-} \times \cC^{H+}$. 
Then, following the arguments on \cite[Proof of Theorem 3]{KM}, the following can be shown:
\begin{enumerate}[(i)]
    \item there exists  subsequence of $Z_n$ for $n \in \mathbb{N}_1 \subset \mathbb{N}$ and a probability space $(\bar{\Omega}, \bar{\cF}, \bar{\mathbb{P}})$ carrying $\mathcal{Z}$-valued random variables $\bar{Z}_n, \bar{Z}$, $n \in \mathbb{N}_2$, with 
\begin{equs}
    \bar{Z}_n& = (\bar{X}^{k(n)}, \bar{D}^{k(n)},  \bar{X}^{k'(n)}, \bar{D}^{k(n)},  \blue{\bar{B}^n
} , \bar{W}^n),
\\
 \bar{Z}&= (\bar{X}^1, \bar{D}^1, \bar{X}^2 , \bar{D}^2,  \bar{\blue{B}
} , \bar{W}),
\end{equs}
such that $\bar{Z}_n \overset{d}{=}Z_n$ for $n \in \mathbb{N}_1$ and $Z_n \to Z$ $\bar{\mathbb{P}}$-almost surely as $\mathbb{N}_1 \ni n \to \infty$. 

\item $\bar{W}$ is an $\bar{\mathbb{F}}$-Brownian motion, where $\bar{\mathbb{F}}=(\bar{\cF}_t)_{t \in [0,1]}$, $\bar{\cF}_t= \sigma( \bar{W}_s, \bar{X}^i_s, \bar{D}^i_s, i=1,2, s \leq t)$ and $\bar{\mathbb{F}}=(\bar{\cF}_t)_{t \in [0,1]}$. 

\item $\bar{\blue{B}
}$ is the Gaussian rough path lift of the fractional Brownian motion 
\begin{equs}
    \bar{B}_t= \int_0^t K(t,r) \, d \bar{W}_r.
\end{equs}
\item For each $i =1,2$, $\bar{X}^i$ is a solution with drift component $\bar{D}^i$, driven by the same noise $\bar{\blue{B}
}$ and with respect to the same filtration $\bar{\mathbb{F}}$. 
\end{enumerate}
Consequently, by the uniqueness part $\bar X^1=\bar X^2$ (hence $\bar D^1=\bar D^2$), and the characterisation of convergence in probability of Gy\"ongy and Krylov (\cite{Gyongy_Krylov}), we have that the original sequence $(X^n, D^n)$ converges in probability to some $(X, D)$. From this, again by using the arguments of  the \cite[Proof of Theorem 3]{KM}, it is easy to see that $X$ is a strong solution to \cref{def_solution} with drift component $D$.
\end{proof}

\begin{appendix}
% \appendix   
\section{Auxiliary lemmas} 

  \label{sec:appendix}
%\addcontentsline{toc}{section}{\protect\numberline{}Appendix}

\begin{lemma}  \label{lem:buckling_D_norm}
    Let $(S, T) \in [0,1]^2_\leq$ and $\blue{f} \in \mathcal{D}_g^{2\gamma}([s,t])$.  Suppose that there exist constants $C_1, C_2$ such that for all $(s,t) \in [S,T]_{\leq}^2$ we have
    \begin{equs}
       \,  [\blue{f}]_{\mathcal{D}_g^{2\gamma}([s,t])} \leq C_1   [\blue{f}]_{\mathcal{D}_g^{2\gamma}([s,t])} |t-s|^{\gamma}+C_2.
    \end{equs}
    Then, it holds that 
    \begin{equs}
       \,  [\blue{f}]_{\mathcal{D}_g^{2\gamma}([S,T])} \leq 2(1+[g]_{\cC^\gamma([S, T])}) (2C_1)^{1/\gamma} C_2. 
    \end{equs}
\end{lemma}
\begin{proof}
    Let us take a partition $\phi=\{s_0, ..., s_m\}$ of $[S, T]$ such that $|\pi| \leq (2 C_1)^{-1/\gamma}$ and $m \leq (2C_1)^{1/\gamma}$. It then follows by assumption that 
    \begin{equs}
       \,  [\blue{f}]_{\mathcal{D}_g^{2\gamma}([s_i,s_{i+1}])} \leq 2 C_2
    \end{equs}
    Next, by \eqref{eq:sum_partition}, we get 
    \begin{equs}
         \,  [\blue{f}]_{\mathcal{D}_g^{2\gamma}([S,T])} \leq 2(1+[g]_{\cC^\gamma([S, T])}) m C_2.
    \end{equs}
\end{proof}

\begin{proof}[Proof of \cref{lem:orthogonal_density}]
    First, it follows from the definitions of $\mathcal{S}(V \otimes \mathcal{H})$ and $\mathbb{D}^{k, p}(V \otimes \mathcal{H})$ that linear combinations of elements of the form 
    $ F v \otimes h$ with  $F\in \mathcal{S}$, $v \in V$, and  $h \in \mathcal{H}$ are dense in $\mathbb{D}^{k, p}(V \otimes \mathcal{H})$. In addition, $\mathcal{S}_{ON}(V \otimes \mathcal{H})$ is a linear space and therefore it suffices to show that  $F v \otimes h$ can be approximated by elements from  $\mathcal{S}_{ON}(V \otimes \mathcal{H})$. Notice that for some $h_1, \dots h_m \in \mathcal{H}$ and $f \in \cC_{\pol}^\infty(\R^m)$ we have 
    \begin{equs}
       \zeta:=  F v \otimes h & = f(B(h_1), \dots B(h_m)) \sum_{i,k=1}^\infty  \langle v, q_k \rangle_V \langle h, e_i \rangle_{\cH} q_k \otimes e_i 
        \\
        & = f\big(\sum_{i=1}^\infty \langle h_1, e_i \rangle_{\cH} B(e_i), \dots,  \sum_{i=1}^\infty \langle h_m, e_i \rangle_{\cH} B(e_i)\big) \sum_{i,k=1}^\infty  \langle v, q_k \rangle_V \langle h, e_i \rangle_{\cH} q_k \otimes e_i,
    \end{equs}
    where we have used that $B$ is an isonormal process. Then, we set 
    \begin{equs}
        \zeta_n: = F_n  \sum_{i,k=1}^n \langle v, q_k \rangle_V \langle h, e_i \rangle_{\cH} q_k \otimes e_i,
    \end{equs}
    where 
    \begin{equs}
       F_n:=  f\big(\sum_{i=1}^n \langle h_1, e_i \rangle_{\cH} B(e_i), \dots,  \sum_{i=1}^n \langle h_m, e_i \rangle_{\cH} B(e_i)\big).
    \end{equs}
    We clearly have $\zeta_n \in \mathcal{S}_{ON}(V \otimes \mathcal{H})$. By the triangle inequality and Parseval's identity  we have 
    \begin{equs}
        \| \zeta-\zeta_n\|_{V \otimes \mathcal{H}}^2  \leq 2|F-F_n|^2  \sum_{i,k=1}^\infty  \langle v, q_k \rangle_V^2 \langle h, e_i \rangle_{\cH}^2+ 2|F_n|^2 \sum_{ i\vee k>n}\langle v, q_k \rangle_V^2 \langle h, e_i \rangle_{\cH}^2,
    \end{equs}
    which gives 
    \begin{equs}
         \| \zeta-\zeta_n\|_{L_p(\Omega; V \otimes \mathcal{H}) } & \leq 2\|F-F_n\|_{L_p(\Omega)}\|v\otimes h\|_{V\otimes \cH}
         \\
         & \qquad + 2\|F_n\|_{L_p(\Omega)}\Big(\sum_{ i\vee k>n}\langle v, q_k \rangle_V^2 \langle h, e_i \rangle_{\cH}^2\Big)^{1/2}.      \label{eq:zeta-zeta_n}
    \end{equs}
    Next, notice that since $f$ has polynomial growth, for any $q \geq 2$, there exists $l \in \mathbb{N}$  such that 
    \begin{equs}
        \E |F_n|^q & \lesssim  1+  \sum_{j=1}^m \E \Big| \sum_{i=1}^n \langle h_j, e_i \rangle_{\cH} B(e_i)  \Big|^l    \lesssim  1+  \sum_{j=1}^m \Big\| \sum_{i=1}^n \langle h_j, e_i \rangle_{\cH} B(e_i)  \Big\|_{L_2(\Omega)} ^l 
        \\
        & \lesssim 1+  \sum_{j=1}^m \Big(  \sum_{i=1}^\infty  \langle h_j, e_i \rangle_{\cH}^2  \Big)^{l/2} \lesssim 1,
    \end{equs}
    where in the second inequality we have used Gaussianity.
    This shows that $F_n$ are bounded uniformly in $n$ in $L_q$ for any $q \geq 1$. Consequently, the second summand in \eqref{eq:zeta-zeta_n} converges to zero with $n \to \infty$. In addition, by using the continuity of $f$, we see that $|F-F_n| \to 0$ in probability as $n \to \infty$, which combined with the boundedness of $F_n$ in $L_q$ for any $q \geq1$ shows that the first summand in \eqref{eq:zeta-zeta_n}  also converges to zero. Consequently, we get 
    that $\lim_{n \to \infty}\| \zeta-\zeta_n\|_{L_p(\Omega; V \otimes \mathcal{H}) }=0$. Finally, to check that $\lim_{n \to \infty}\| D^\kappa (\zeta-\zeta_n)\|_{L_p(\Omega; V \otimes \mathcal{H}^{\otimes (\kappa+1) }) }=0$ for $\kappa \leq k$, it suffices to Malliavin differentiate the expressions for $\zeta$ and $\zeta_n$ and repeat the above argument.  
\end{proof}

\begin{lemma}\label{lem:interpolation}
%    Chengcheng will write this, thanks!\\ 
    Let $p\in[1,\infty)$, $k,\ell\in\N$, $k$ even. Take a  linear map $T:\mathcal{C}^\infty(\mR^d)\mapsto L^p(\Omega)$ such there exist positive constants $\Gamma_{\ell}\leq\Gamma_{-k}$ so that the following holds: for any $f\in \mathcal{C}^\infty(\mR^d)$ the following bounds hold                       
    %and $\forall\mathbb{k}=(k_1,\ldots,k_d)\in\N^d$ so that $|\mathbb{k}|=k$,
    \begin{equs}
       \|Tf\|_{L^p(\Omega)} &\leq \Gamma_{\ell}\|\nabla^\ell f\|_{\infty}\label{con:interpolation-1},\\
        \|T(\nabla^k f)\|_{L^p(\Omega)} &\leq \Gamma_{-k}\|f\|_{\infty}.  \label{con:interpolation-2}
    \end{equs}
    Suppose furthermore that $Tf_n\to T f$ in $L^p(\Omega)$ whenever $f_n\to f$ uniformly on compacts.
    Then for any $\beta\in[-k,0)$ there exists a constant $C=C(d,\ell,k,\beta)$ such that for all $f\in \mathcal{C}^\infty(\mR^d)$ one has
    \begin{equ}
        \label{est:interpolation}  \|Tf\|_{L^p(\Omega)} \leq C \Gamma_{\ell}^{1-\theta}\Gamma_{-k}^{\theta}\|f\|_{\mathcal{C}^{{\beta}}},   
    \end{equ}
    where $\theta\in[0,1]$ is uniquely defined by $\beta=(1-\theta)\ell-\theta k$.
\end{lemma}

\begin{proof} 
By continuity, it suffices to show \eqref{est:interpolation} with $P_t f$ in place of $f$ on the left-hand side, uniformly in $t\in(0,1]$. First we have, by definition of the H\"older norm \eqref{def:C-}
\begin{equ}
\|T P_tf\|_{L_p(\Omega)}\leq \Gamma_\ell\|\nabla^\ell P_t f\|_{L^\infty}\lesssim \Gamma_\ell t^{-\ell/2}\|P_{t/2} f\|_{L^\infty}\leq \Gamma_\ell t^{-\ell/2}t^{\beta/2}\|f\|_{\cC^\beta}.
\end{equ}
On the other hand, if $k$ is even, so that $k=2m$ for $m\in \N$, then
\begin{equs}
\|T P_t f\|_{L^p(\Omega)}&=\|TP_{t/m}\cdots P_{t/m} f\|_{L^(\Omega)}
\\
&=\|T\int_0^{t/m}\cdots\int_0^{t/m}\partial_t P_{s_1}\cdots\partial_t P_{s_{m}} f\,ds_1\cdots ds_m\|_{L^p(\Omega}
\\
&=\|\int_0^{t/m}\cdots\int_0^{t/m} T\Delta^m P_{s_1+\cdots +s_m}f\,ds_1\cdots\,ds_m\|_{L^p(\Omega)}
\\
&\leq \Gamma_{-k}\int_0^{t/m}\cdots\int_0^{t/m} \|P_{s_1+\cdots +s_m}f\|_{L^\infty}\,ds_1\cdots\,ds_m
\\
&\leq\Gamma_{-k}\|f\|_{\cC^\beta}\int_0^{t/m}\cdots\int_0^{t/m} (s_1+\cdots +s_m)^{\beta/2}\,ds_1\cdots\,ds_m\\&\lesssim \Gamma_{-k}\|f\|_{\cC^\beta} t^{k/2+\beta/2},
\end{equs}
using $\beta/2>-k/2=-m$ in the last step.
Raising the first estimate to the $(1-\theta)$ power and the second to $\theta$, the factors involving $t$ cancel exactly, and we get the desired bound.
\end{proof}

The following lemma follows directly by \cite[Proposition 2.2, Lemma 2.3]{Mate_Peter} and standard Young integration. For the notation $\cC^\beta_\gamma$ therein, recall \eqref{eq:weighted-Holder}.
\begin{lemma}\label{lem:Young_with_blow_up}
Let $s \in [0,1]$,  $\alpha, \beta \in (0,1)$ such that $\alpha+ \beta >1$ and let $\gamma \in (0, \beta)$. Then, for all $f \in \cC^\alpha([s,1])$, $g \in \cC^\beta_\gamma([s,1])$ and  $(u,t) \in [0,1]^2_{\leq}$, we have 
\begin{equs}
\big| \int_u^t f_{u,r} \, dg_r \big| \lesssim [f]_{\cC^\alpha} [g]_{\cC^\beta_\gamma} |t-u|^{\beta+\alpha -\gamma}.
\end{equs}
\end{lemma}
\end{appendix}

 \section*{Acknowledgment}
KD and KL have been supported by  the Engineering \& Physical Sciences Research Council (EPSRC) [grant number EP/Y016955/1].
MG and CL acknowledge funding from the FWF: This research was funded in whole or in part by the Austrian Science Fund (FWF) [10.55776/P34992]. For open access purposes, the authors have applied a CC BY public copyright license to any authors accepted manuscript version arising from this submission.

\bibliographystyle{Martin}
\bibliography{multi-bib}
\end{document}